\documentclass{article}

\usepackage{PRIMEarxiv}

\usepackage[utf8]{inputenc} 
\usepackage[T1]{fontenc}    
\usepackage{hyperref}       
\usepackage{url}            
\usepackage{booktabs}       
\usepackage{amsfonts}       
\usepackage{amsmath}
\usepackage{amssymb}
\usepackage{amsthm}
\usepackage{nicefrac}       
\usepackage{microtype}      
\usepackage{lipsum}
\usepackage{fancyhdr}       
\usepackage{graphicx}       
\graphicspath{{media/}}     

\usepackage{tikz}
\usetikzlibrary{positioning}
\usetikzlibrary{decorations,arrows}
\usetikzlibrary{decorations.markings}
\usepackage{capt-of}
\theoremstyle{plain}
\newtheorem{thm}{Theorem}

\newtheorem{prop}[thm]{Proposition}
\newtheorem{lemma}[thm]{Lemma}

\newtheorem{rem}[thm]{Remark}

\theoremstyle{definition}
\newtheorem{defn}[thm]{Definition}

\numberwithin{equation}{section}
\numberwithin{thm}{section}

\title{Wright-Fisher diffusion and coalescent with a continuum of seed-banks
}

\author{
 Likai Jiao \\
  Humboldt-Universit\"{a}t zu Berlin\\
  \texttt{likai.jiao@hu-berlin.de}\\
}
\begin{document}
\maketitle

\begin{abstract}
This paper generalizes the strong seed-bank model introduced in \cite{blath2016} to allow for more general dormancy time distributions, such as a type of Pareto distribution. Inspired by the method of approximation using models with countably many seed-banks proposed by \cite{greven2022spatial}, we introduce the Wright-Fisher diffusion and coalescent with a continuum of seed-banks. To this end, we first formulate an infinite-dimensional stochastic differential equation, and prove that it has a unique strong solution, refereed to as the \textit{continuum seed-bank diffusion}, which is a kind of Markovian lift of a non-Markovian Volterra process. In order to circumvent the technical difficulty arising from the lack of local compactness, we replace the topology induced by the norm on the state space of the solution with the weak-$^{\star}$ topology, and show that the continuum seed-bank diffusion is also a strong Markov process in this weak-$^{\star}$ setting. Then, we construct a discrete-time Wright-Fisher type model with finitely many seed-banks, and demonstrate that the continuum seed-bank diffusion under the weak-$^{\star}$ topology is the scaling limit of the allele frequency processes in a suitable sequence of such models. Finally, we establish the duality relation between the continuum seed-bank diffusion and a continuous-time continuous-state Markov jump process. The latter is the block counting process of a partition-valued Markov jump process which is referred to as the \textit{continuum seed-bank coalescent}. We prove that this new coalescent process is exactly the scaling limit of the ancestral processes in the sequence of discrete-time Wright-Fisher type models we constructed before.
\end{abstract}

\keywords{Wright-Fisher diffusion\and Coalescent\and Seed-banks \and Continuum \and Infinite-dimensional SDE}

\section{Introduction}
In this paper, we introduce the Wright-Fisher diffusion and coalescent with a continuum of seed-banks, which will be referred to as the \textit{continuum seed-bank diffusion} and the \textit{continuum seed-bank coalescent}, respectively. This pair of processes, interrelated by a certain duality relation, generalizes the single seed-bank model introduced in \cite{blath2016}, as well as the model with countably many \textit{colored} seed-banks proposed by \cite{greven2022spatial}. The concept of \textit{seed-bank} originates from the study of the widespread biological phenomenon of dormancy observed in nature. In the realm of mathematical modeling related to dormancy, \cite{blath2016} constructed a discrete-time Wright-Fisher type model with a single seed-bank, in which the time an individual speeds in the seed-bank, referred to as the \textit{dormancy time}, follows a geometric distribution. On the timescale of the active population size, the scaling limit of the allele frequency process in this model is known as the \textit{seed-bank diffusion}. From the corresponding \textit{seed-bank coalescent}, we can see that the limiting dormancy time follows an exponential distribution. In order to consider more general distributions, which typically imply non-Markovian mechanisms, \cite{greven2022spatial} proposed the method of approximation using models with countably many seed-banks labeled by the parameters of different exponential distributions. When an individual goes into dormancy, it randomly selects one of the seed-banks to enter. After appropriately adjusting the labels, the resulting dormancy time distribution exhibits an asymptotic power-law tail. Inspired by this idea, we will formulate a more general \textit{continuum} seed-bank model, in which the labels of the seed-banks take values in $(0,\infty)$. The probability of an individual entering different seed-banks is determined by some specific finite Borel measure $\mu$. Based on the different choices of the measure $\mu$, the range of dormancy time distributions is significantly expanded. For example, when $\mu$ is the Gamma distribution, the dormancy time distribution is a type of Pareto distribution, which basically corresponds to the heavy-tailed distribution given by \cite{greven2022spatial}. The main contributions of the present paper can be summarized as follows:
\begin{enumerate}
	\item By observing the equivalence between the seed-bank diffusion and a Volterra process, we consider the stochastic Volterra equation with a more general completely monotone kernel (Equation (\ref{eq:v})), which is equivalent to an infinite-dimensional stochastic differential equation (Equation (\ref{eq:inf})) based on the Bernstein-Widder theorem (see Theorem \ref{Bernstein}). We prove that Equation (\ref{eq:inf}) has a unique strong solution under the finite first moment condition on the measure $\mu$ (see Theorem \ref{th:wellposed}). The solution, referred to as the \textit{continuum seed-bank diffusion}, is a strong Markov process, thus Equation (\ref{eq:inf}) can be viewed as an infinite-dimensional Markovian lift of the non-Markovian Equation (\ref{eq:v}).
	\item The state space of the continuum seed-bank diffusion is generally not locally compact (see Proposition \ref{prop:original}), which makes it challenging to establish tightness arguments. To circumvent this technical difficulty, we replace the topology induced by the norm with the weak$^{\star}$ topology so that the new state space becomes a compact metrizable space. We show that the continuum seed-bank diffusion is also a strong Markov process in this weak$^{\star}$ setting. 	
	\item When the total mass of the measure $\mu$ is an integer, we construct a discrete-time Wright-Fisher type model with finitely many seed-banks, and we demonstrate that there exist a sequence of such models in which the allele frequency processes converge in distribution to the continuum seed-bank diffusion, on the timescale of the active population size, if the state space is endowed with the weak$^{\star}$ topology (see Theorem \ref{th:scaling}). This provides a scaling limit interpretation for the continuum seed-bank diffusion.
	\item We define the continuum seed-bank $K$-coalescent by (\ref{eq:K}) for any finite Borel measure $\mu$ on $(0,\infty)$, which is a continuous-time continuous-state Markov jump process, and it corresponds to the cases studied in \cite{blath2016} and \cite{greven2022spatial} when $\mu$ is discrete. We first establish the duality relation between the continuum seed-bank diffusion and the blocking counting process of the continuum seed-bank $K$-coalescent (see Theorem \ref{th:dual}). Then, by a projective limit argument, we construct the continuum seed-bank coalescent (see Proposition \ref{prop:projective}). Finally, we prove that this coalescent process is exactly the scaling limit of the ancestral process in the discrete-time Wright-Fisher type model we constructed before (see Proposition \ref{prop:ancestral}).	 
\end{enumerate}

The Wright-Fisher diffusion holds a significant position as a fundamental probabilistic tool in the field of mathematical population genetics. It represents the scaling limit
of the basic discrete-time Wright-Fisher model (\cite{fisher1923}, \cite{wright1931}) which characterizes the evolution of neutral allele frequencies in a fixed-size haploid population. Reversing time, if one samples a group of individuals at the same generation in the Wright-Fisher model and traces back their ancestors, the scaling limit of the corresponding partition-valued ancestral process is the well-known Kingman coalescent (\cite{kingman1982genealogy},\cite{kingman1982coalescent}). With the further incorporation of other evolutionary forces such as mutation, selection, recombination and migration into the basic Wright-Fisher model, the coalescent theory has become indispensable for understanding the effects of various evolutionary forces and their interactions during lineage formation. See \cite{wakeley2009coalescent} for an overview of the coalescent theory.

In recent years, dormancy has also been recognized as an important evolutionary force. When a population is exposed to unfavorable environmental conditions such as resource limitations, some individuals may enter the \textit{seed-bank}, where they remain in a dormant state of reduced metabolism until the environment becomes more hospitable. Through this buffering mechanism, dormant individuals can receive a certain degree of protection not only against harsh conditions but also against some classical evolutionary forces, which increases the persistence of genotypes and contributes to the maintenance of biodiversity. Seed-banks, as reservoirs of dormant individuals, are highly prevalent in microbial communities and can even play a dominant role within certain populations. See \cite{lennon2011microbial}, \cite{shoemaker2018evolution} and \cite{lennon2021principles} for overviews on the study of dormancy.

A mathematical model of the Wright-Fisher type with dormancy was first introduced by \cite{kaj2001coalescent} which allows individuals in the new generation to be descendants of the individuals that lived $B$ generations ago, where the dormancy time $B$ is a random variable taking values in $\{1,2,\cdots,m\}$ for some positive integer $m$. This mechanism delays the coalescence of ancestors in the ancestral process, resulting in a \textit{stretched} genealogy, and the corresponding coalescent process was shown to be a time-changed Kingman coalescent. Due to the boundedness of $B$, this seed-bank model is more suitable for populations of macro-organisms, where the dormancy period is usually negligible on the population-size scale. However, for certain micro-organisms that can remain in the dormant state for long periods of time, \cite{gonzalez2014strong} proposed to use a seed-bank model incorporating long-range ancestral jumps, which is referred to as the \textit{strong}  model in contrast to the \textit{weak} model with a bounded dormancy time. 

For the strong seed-bank models within the framework of \cite{kaj2001coalescent}, as indicated by the studies of two specific examples in \cite{blath2013ancestral} and \cite{blath2015genealogy}, their coalescent structures are outside the Kingman coalescent universality class. Therefore, a new strong seed-bank model along with its coalescent process were introduced in \cite{blath2016}. Different from the non-Markovian nature of the forward process in \cite{kaj2001coalescent}, the Wright-Fisher model with a \textit{geometric} seed-bank in \cite{blath2016} is still Markovian. The scaling limit of the allele frequency process is the Wright-Fisher diffusion with an \textit{exponential} seed-bank, which is characterized by a $2$-dimensional stochastic differential equation. Here, the \textit{geometric} and \textit{exponential} imply the distributions of the dormancy times. Although this seed-bank model can be regarded as a special two-island model where the population is divided into \textit{active} and \textit{dormant} subpopulations and no reproduction occurs in the latter, there are significant qualitative differences between the \textit{seed-bank coalescent} and the \textit{structured coalescent} for two islands. See Section 4 of \cite{blath2016} for more details.

Now, a natural question within the framework of \cite{blath2016} is how to generalize the dormancy time distribution from \textit{exponential} to more general ones, especially those with heavy tails, which will cause the seed-bank model to lose the Markov property again. When studying the impact of seed-banks on the long-time behavior of spatially structured populations, \cite{greven2022spatial} proposed the method of approximation by a countable number of \textit{colored} seed-banks, as the linear combinations of exponential survival functions can approximate survival functions with power-law tails. The main purpose of this paper is to mathematically refine this asymptotic idea. It leads to the introduction of the Wright-Fisher diffusion with a continuum of seed-banks, which is characterized by a more general infinite-dimensional SDE. 

Before proceeding with further elaboration, it is helpful to review more details of the Wright-Fisher diffusion with at most countably many seed-banks. Consider a haploid population where each individual carries an allele of type from $\{A, a\}$. In the absence of evolutionary forces except for genetic drift and dormancy, the SDE for the generalized Wright-Fisher diffusion is as follows:
\begin{equation}\label{eq:countable}
\left\{\begin{array}{l}
dX_t=\sum\limits_{i=1}^{k} c_i\left(Y_t^i-X_t\right) d t+\sqrt{X_t\left(1-X_t\right)} d W_{t},~t>0,\\
dY_t^i=c_i K_i\left(X_t-Y_t^i\right) d t,~t>0,~i=1,2,\cdots,k,\\
(X_0, \vec{Y}_0)=(x, \vec{y}) \in [0,1]^{k+1},
\end{array}\right.
\end{equation}
where $k\in \{1,2,\cdots, \infty\}$, $X$ and $\vec{Y}=\{Y^{1}, Y^{2}, \cdots, Y^{k}\}$ are the frequencies of type-$A$ alleles in the active population and the seed-banks, respectively, $W$ is a standard Brownian motion, and $c_{i}, K_{i}$ are positive parameters associated with the discrete-time model such that $c:=\sum\limits_{i=1}^{k} c_i<\infty$. The key observation, as already noticed by \cite{blath2019structural} in the differential form, is that Equation (\ref{eq:countable}) is equivalent to the following stochastic Volterra equation (SVE):
\begin{equation}\label{eq:countable_v}
	X_{t}=g_{m}(t)+c\int_{0}^{t}K_{m}(t-s)X_{s}ds-c\int_{0}^{t}X_{s}ds+\int_{0}^{t}\sqrt{X_{s}(1-X_{s})}dW_{s}, t\geq 0,
\end{equation}
where $g_{m}(t):=x+\sum\limits_{i=1}^{k}\frac{y_{i}}{K_{i}}(1-e^{-c_{i}K_{i}t})$, and  
$K_{m}(t):=\sum\limits_{i=1}^{k}\frac{c_{i}}{c}(1-e^{-c_{i}K_{i} t})$ 
is a mixture of exponential distribution functions. In the biological sense, equation (\ref{eq:countable_v}) can be roughly interpreted as follows: If the dormancy time $B$ does not exceed $t-s$ of which the probability is $P(B\leq t-s):=K_{m}(t-s)$, then the individuals entering dormancy at the earlier time $s$ should contribute to $X_{t}$. Similarly, $g_{m}(t)$ represents the contribution of the initial value $(x,\vec{y})$ to $X_{t}$.

Stochastic Volterra equations have been extensively studied since \cite{berger1980volterra}. For the convolution-type SVEs, if the kernel function is completely monotone (see Definition \ref{defn:completemonotone}), then the equation can be related to an infinite-dimensional Markov process under certain integrability conditions by the Bernstein-Widder theorem (see Theorem \ref{Bernstein}). This method was initially applied to fractional Brownian motion in \cite{coutin1998fractional}, and has been further extended to rough volatility models by \cite{abi2019multifactor}. In the context of the present paper, we generalize Equation (\ref{eq:countable_v}) to the following form:
\begin{equation}\label{eq:v}
    X_{t}=g(t)+c\int_{0}^{t}K(t-s)X_{s}ds-c\int_{0}^{t}X_{s}ds+\int_{0}^{t}\sqrt{X_{s}(1-X_{s})}dW_{s}, t\geq 0,
\end{equation}
where $g(t):=x+\int_{(0, \infty)}y(\lambda)\frac{1-e^{-\lambda t}}{\lambda}\mu(d\lambda)$, $K(t):=\int_{(0,\infty)}(1-e^{-\lambda t})\frac{\mu}{c}(d\lambda)$, and $\mu$ is a finite Borel measure on $(0, \infty)$ with total mass $c>0$. Specifically, taking $\mu=\sum\limits_{i=1}^{k}c_{i}\delta_{c_{i}K_{i}}$, one gets Equation (\ref{eq:countable_v}). The equivalence between this more general SVE and the following infinite-dimensional SDE can be directly verified:
\begin{equation}\label{eq:inf}
\left\{\begin{array}{l}
dX_{t}=\int_{(0,\infty)} Y_{t}(\lambda) \mu(d \lambda) d t-c X_{t} d t+\sqrt{X_{t}\left(1-X_{t}\right)} d W_{t}, t>0,\\
dY_{t}(\lambda)=\lambda\left(X_{t}-Y_{t}(\lambda)\right) dt, \lambda\in(0,\infty), t>0, \\
\left(X_{0}, Y_{0}(\lambda)\right)=(x, y(\lambda))\in D,
\end{array}\right.
\end{equation}
where $D=[0,1]\times \left\{y:(0, \infty) \rightarrow \mathbb{R} \text { is Borel measurable, and } 0 \leq y \leq 1, \mu \text {-a.e.}\right\}$. Under the condition that

\begin{equation}\label{condition}
c^{\prime}:=\int_{(0,\infty)}\lambda \mu(d\lambda)<\infty,	
\end{equation}

Equation (\ref{eq:inf}) will be formulated as a stochastic evolution equation (SEE) in $\mathbb{R} \times L^{1}((0,\infty),\mathcal{B}(0,\infty),\mu; \mathbb{R})$:
\begin{equation}\label{eq:SEE}
\left\{\begin{array}{l}
d Z_{t}=A Z_{t} d t+F\left(Z_{t}\right) d t+B\left(Z_{t}\right) d W_{t}, t>0, \\
Z_{0}=(\xi, \eta(\lambda)),
\end{array}\right.
\end{equation}
where $Z_{t}:=(X_{t}, Y_{t}(\lambda))$, $A(x, y(\lambda)):=(-cx, -\lambda y(\lambda))$, 
$F(x, y(\lambda)):=(\int_{(0,\infty)}y(\lambda)\mu(d\lambda), \lambda x)$, $B(x, y(\lambda)):=(\sqrt{x(1-x)}, 0)$, and $(\xi, \eta(\lambda))$ is a $D$-valued random variable. 

The well-posedness of equation (\ref{eq:SEE}) is provided by 

\begin{thm}\label{th:wellposed}~

\begin{enumerate}
	\item The equation (\ref{eq:SEE}) has a $D$-valued continuous unique strong solution;
	\item Let $Z^{1}$, $Z^{2}$ be two solutions to Equation (\ref{eq:SEE}) on some filtered probability space, then for any $T>0$, 
\begin{equation}
E[\sup_{t\in [0, T]}\|Z^{1}_{t}-Z^{2}_{t}\|]\leq C E[\|Z^{1}_{0}-Z^{2}_{0}\|],
\end{equation}
where $C$ is a positive constant depending only on $c$, $c^{\prime}$ and $T$.
\end{enumerate}
\end{thm}

Since Equation (\ref{eq:SEE}) is essentially a $1$-dimensional SDE combined with an infinite-dimensional ODE, the proof of Theorem \ref{th:wellposed} is mainly based on \cite{yamada1971uniqueness}, \cite{shiga1980infinite}, \cite{da2014stochastic} and \cite{gorajski2014equivalence}, see Section \ref{sec:wellposedness}. In the subsequent text, the solution $Z$ will be referred to as the \textit{continuum seed-bank diffusion}. As it is explained for equation (\ref{eq:countable_v}), the overall effect under the measure $\mu$ of a continuum of seed-banks with different rates $\lambda$ is reflected by the cumulative distribution function (CDF) $K$ of the dormancy time $B$. For example, if $\mu$ is the Gamma distribution $\Gamma(a, b)$ for $a>0,b>0$, then $K(t)=1-\frac{1}{(1+\frac{t}{b})^{a}},t\geq 0$, which is the CDF of a type of Pareto distribution. It is classical to show that the continuum seed-bank diffusion is a strong Markov process, and then we derive its martingale problem formulation, see Section \ref{sec:markov}. Moreover, it belongs to the infinite-dimensional polynomial process defined in \cite{cuchiero2021infinite}, see Remark \ref{rem:poly}. 

Recall that in \cite{blath2016}, the seed-bank diffusion was introduced as the scaling limit of the allele frequency process in a discrete-time Wright-Fisher type model, whereas in the previous discussion, it was obtained by generalizing the kernel function $K$. For a more intuitive understanding, one may also want to interpret the continuum seed-bank diffusion as the scaling limit of some discrete-time model. The second part of this paper is dedicated to providing such an interpretation. However, before doing so, there is a proposition indicating that the state space $D$ endowed with the subspace topology of $\mathbb{R}\times L^{1}(\mu)$ is locally compact if and only if the measure $\mu$ is discrete, which will make it challenging to demonstrate tightness in the proof of convergence. 

To circumvent this technical difficulty, the state space $D$ will be embedded into $\mathbb{R}\times\mathcal{M}(0,\infty)$ based on the isometry from $\mathbb{R}\times L^{1}(\mu)$ to $\mathbb{R}\times(\mathcal{M}(0,\infty),||\cdot||_{TV})$ given by $i_{\mu}: (x, f)\mapsto (x, f.\mu)$, where $\mathcal{M}(0,\infty)$ is the space of finite signed measures on $(0,\infty)$, $||\cdot||_{TV}$ denotes the total variation norm, and $f.\mu:=\int_{\cdot}fd\mu$ denotes the indefinite integral. Since $\mathcal{M}(0,\infty)$ can be viewed as the dual space of $C_{0}(0,\infty)$ (The space of continuous functions vanishing at $0$ and $\infty$ on $(0,\infty)$ equipped with the supremum norm) by the Riesz-Markov-Kakutani representation theorem (see e.g. \cite{cohn2013measure}), and $D$ is bounded in $L^{1}(\mu)$, then by the Alaoglu theorem (see e.g. \cite{schaefer1971locally}), the image $i_{\mu}(D)$ is weak$^{*}$ relatively compact. Moreover, $i_{\mu}(D)$ is metrizable and closed under the weak$^{*}$ topology. We regard $D$ and $i_{\mu}(D)$ as the same and denote the metric as $d$, then $(D, d)$ is a compact metric space according to the above statements.

After replacing the state space $D$ with $(D, d)$, Equation (\ref{eq:inf}) will be understood as
\begin{equation}\label{eq:measure}
\left\{\begin{array}{l}
d X_{t}=(\mu_{t}(0, \infty)-c X_{t}) dt+\sqrt{X_{t}\left(1-X_{t}\right)} d W_{t}, t>0,\\
d \mu_{t}=\lambda.\mu X_{t}dt-\lambda.\mu_{t} d t, t>0,\\
(X_{0}, \mu_{0})=(x, y.\mu),
\end{array}\right.	
\end{equation}
where the seed-bank component $\mu_{t}:=Y_{t}.\mu$ is measure valued. Under this coarser topology, the solution $Z$ is still a strong Markov process, and the corresponding martingale problem formulation will also be derived in Section \ref{sec:markov}. Moreover, it can be shown that $Z$ is a Feller process, which is generally difficult to verify under the original topology due to the lack of local compactness. The related scaling limit problem can be handled by the classical approach (see e.g. \cite{ethier2009markov}), what remains to be done is to construct suitable discrete-time models.

When the total mass $c$ of the measure $\mu$ is an integer, the required Wright-Fisher type model will be introduced in Section \ref{sec:scaling}. The difference from the model in \cite{blath2016} for finitely many seed-banks is that the number of individuals exchanged between the active population and seed-banks in each generation is a multinomial distributed random vector, instead of being deterministic. This is because the number of seed-banks keeps increasing during the approximation, while $c$ remains unchanged. As the result of the randomness, there is a low probability that all individuals entering dormancy will end up in the same seed-bank. Therefore, in this fixed-size model, the size of all seed-banks must be greater than $c$. Under these settings, the following theorem will be proved:

\begin{thm}\label{th:scaling}
For a sequence of Markov chains $\eta_{N_{r}, n_{r}}(X^{N_{r}}, Y^{\vec{M}^{N_{r}}})$, $r=1,2,\cdots$, as their initial distributions weakly converge to $\nu$, 
\begin{equation}
Z_{r}(t)=\eta_{N_{r}, n_{r}}(X^{N_{r}}(\lfloor N_{r}t\rfloor), Y^{\vec{M}^{N_r}}(\lfloor N_{r}t\rfloor)), t\geq 0,	
\end{equation}
converges in distribution to the continuum seed-bank diffusion $\{(X_{t}, Y_{t}.\mu)\}_{t\geq 0}$ with initial distribution $\nu$ on $D_{\mathbb{R}_{+}}(D,d)$, as $r\rightarrow\infty$.
\end{thm}

Here $N_{r}$ and $\vec{M}^{N_r}=\{M^{N_r}_{1},M^{N_r}_{2},\cdots, M^{N_{r}}_{n_{r}}\}$ are sizes of the active population and the $n_{r}$ seed-banks, respectively, such that $N_{r}\geq c$, $\min\limits_{i=1,2,\cdots,n_{r}}M^{N_r}_{i}\geq c$, $\lim\limits_{r\rightarrow\infty}N_{r}=\infty$, and $\lim\limits_{r\rightarrow\infty}n_{r}=\infty$; $(X^{N_{r}}, Y^{\vec{M}^{N_r}})$ are the frequencies of type-$A$ alleles in the $r$-th model, which is a vector-valued Markov chain; $\lfloor \cdot \rfloor$ denotes the floor function, $D_{\mathbb{R}_{+}}(D,d)$ denotes the Skorokhod space of c\`{a}dl\`{a}g functions from $[0,\infty)$ to $(D,d)$, and $\eta_{N_{r}, n_{r}}$ is a measurable mapping which maps the vector $Y^{\vec{M_{r}}}(\lfloor N_{r}t\rfloor)$ to a certain discrete measure (see Definition \ref{defn:mapping}). Roughly speaking, Theorem \ref{th:scaling} states that there exists a sequence of discrete-time models in which the allele frequency processes converge to the continuum seed-bank diffusion on the timescale of order $N_{r}$, as the active population size $N_{r}$ and the number of seed-banks $n_{r}$ increase.

The final part of this paper focuses on the coalescent process of the continuum seed-bank diffusion, which will be referred to as \textit{continuum seed-bank coalescent}. The primary task is to identify the block counting process
(counting the number of blocks in partitions) of the coalescent, which is typically achieved by establishing the duality relation between two Markov processes. This duality method is a basic tool for proving the uniqueness of solutions to martingale problems, and understanding the long-time behavior of Markov processes. For a systematic study on the notion of duality, see e.g. \cite{jansen2014notion}. When the state space is replaced by $(D,d)$, the solution to Equation (\ref{eq:measure}) may be reminiscent of the Fleming-Viot process (\cite{fleming1979some}). However, a similar dual function seems to be inapplicable, and more importantly, when the measure $\mu$ is discrete, it can not degenerate into the case of at most countably many seed-banks.

In \cite{greven2022spatial}, the duality relation between the solution $(X, \vec{Y})$ to Equation (\ref{eq:countable}) and the block counting process $\{(N_{t}, \vec{M}_{t})\}_{t\geq 0}$ is given by (Put them into the product filtered probability space so that they are independent)
\begin{equation}
E[X_{t}^{n_{0}}\prod_{i=1}^{k}(Y^{i}_{t})^{m^{i}_0}]=E[x^{N_{t}}\prod_{i=1}^{k}(y^{i})^{M^{i}_t}],t\geq 0,	
\end{equation}
where $M_{t}$ and the $i$-th component $M^{i}_{t}$ of $M_{t}$ take values in $\{0, 1,\cdots\}$, $y^{i}$ is the $i$-th component of $\vec{y}$, and $(x,\vec{y}), (n_{0}, \vec{m}_{0})$ are the initial values to $(X, \vec{Y})$ and $(N,\vec{M})$, respectively.
It is worth noting that the moment dual function employed here has an equivalent expression
\begin{equation}\label{dualfunction}
F[(x, y(\lambda)),(n,m(d\lambda))]=x^{n}e^{\int_{(0,\infty)} ln y(\lambda)m(d\lambda)},
\end{equation}
where $m(d\lambda):=\sum\limits_{i=1}^{k}m_{i}\delta_{\lambda_{i}}(d\lambda)$, $m_{i}\in \{0,1,\cdots\}$, $\lambda_{i}\in (0,\infty)$, $y(\lambda):=\sum\limits_{i=1}^{k}y_{i}I_{\{\lambda_{i}\}}(\lambda)$, and $e^{-\infty}:=0$. In addtion, if the initial value $y(\lambda)$ in equation (\ref{eq:inf}) is defined everywhere without dependence on a measure $\mu$, then it can be shown that $0\leq Y_{t}(\lambda)\leq 1$ for all $t>0$ and $\lambda\in (0,\infty)$ i.e. the solution $Z$ takes values in $\mathbb{R}\times L^{\infty}((0,\infty),\mathcal{B}(0,\infty))$ of which the dual space can be viewed as $\mathbb{R}\times ba((0,\infty),\mathcal{B}(0,\infty))$. Here $ba((0,\infty),\mathcal{B}(0,\infty))$ denotes the space of finitely additive set functions on $((0,\infty),\mathcal{B}(0,\infty))$ equipped with the total variation norm, which is large enough to include all finite measures. Moreover, any finite integer-valued measure on $((0,\infty),\mathcal{B}(0,\infty))$ can be represented as a weighted sum of Dirac measures. Based on the above observations, if in (\ref{dualfunction}), $m$ is an arbitrary finite measure, then $F$ is well-defined in the case when $y$ is a non-negative bounded measurable function, and the differentiation with respect to $y$ is still feasible by adding a perturbation. Specifically, when $m$ takes integer values, $F$ is of the same form as (\ref{dualfunction}) except that $k$ should be finite.

First add suitable perturbations, then apply It\^{o}'s formula, and finally take the limit. Through this procedure, we recognize the generator of the dual process, and it corresponds to the following Markov jump process $\{(N_{t}, M_{t})\}_{t\geq 0}$ described by its transition rates:
\begin{equation}\label{eq:char1}
(n, m) \mapsto (n^{\prime}, m^{\prime}) \text { at rate }\left\{\begin{array}{cl}
n\mu(B), & (n^{\prime}, m^{\prime})=(n-1, m+\delta_{\lambda}), \lambda\in B, \text{for}~B\in\mathcal{B}(0,\infty),\\
\lambda m(\{\lambda\}), & (n^{\prime}, m^{\prime})=(n+1, m-\delta_{\lambda}),\\
C^{2}_{n}, & (n^{\prime}, m^{\prime})=(n-1, m),
\end{array}\right.
\end{equation}
where $n\in \{0,1,\cdots\}$, $m$ is a finite integer-valued measure, and $C^{2}_{n}=\frac{n(n-1)}{2}$. The state space for $(n,m)$ will be denoted by $\mathbb{N}_{0}\times \bigoplus\limits_{(0,\infty)}\mathbb{N}_{0}$\footnote{The direct sum $\bigoplus\limits_{(0,\infty)}\mathbb{N}_{0}$ is the subspace of $\mathbb{N}_{0}^{(0,\infty)}$ for whose elements there are finitely many non-zero components.} due to the aforementioned special form of $m$. From two martingale statements about $(X, Y)$ and $(N, M)$, respectively, we obtain the following result:

\begin{thm}\label{th:dual}
If $0\leq Y_{0}(\lambda)\leq 1$ is $\mathcal{F}_{0}$-measurable for all $\lambda\in (0,\infty)$, then 
\begin{equation}
E[F((X_t, Y_t),(N_{0}, M_{0}))]=E[F((X_{0}, Y_{0}),(N_{t}, M_{t}))].
\end{equation}
In particular, if $(X_{0}, Y_{0})=(x, y)$ and $(N_{0}, M_{0})=(n, \sum\limits_{i=1}^{K_{0}}m_{i}\delta_{\lambda_{i}})$, then
\begin{equation}
E[X_t^n\prod_{i=1}^{K_{0}}Y_{t}(\lambda_{i})^{m_{i}}]=E[x^{N_{t}}\prod_{i=1}^{K_{t}}y(\Lambda_{i,t})^{M_{i,t}}],
\end{equation}
where $M_{t}=\sum\limits_{i=1}^{K_{t}}M_{i,t}\delta_{\Lambda_{i,t}}.$
\end{thm}

Now that the blocking counting process $(N,M)$ is known, the coalescent process of a finite-sized sample can be constructed directly. For $K\in \{1,2,\cdots\}$, let
$\mathcal{P}_{K}$ be the set of partitions of $\{1,2,\cdots,K\}$, and define the space of marked partitions as $\mathcal{P}_{K}^{f}=\{(P_{K}, f)|P_{K}\in \mathcal{P}_{K}, f\in [0,\infty)^{|P_{K}|}\},$
where $f$ is the \textit{flag}, $0$ represents \textit{active}, $\lambda\in (0,\infty)$ represents \textit{dormant} with rate $\lambda$, and $|\cdot|$ denotes the number of blocks. The continuum seed-bank K-coalescent process $\{\Pi^{K}_{t}\}_{t\geq 0}$ is a $\mathcal{P}_{K}^{f}$-valued Markov jump process defined as follows:
\begin{equation}\label{eq:K}
\pi \mapsto \pi^{\prime} \text { at rate }\left\{\begin{array}{cl}
\mu(B), & \pi \leadsto \pi^{\prime}, \text{a}~0~\text{becomes}~\lambda\in B,\text{ for } B\in\mathcal{B}(0,\infty),\\
\lambda, & \pi \leadsto \pi^{\prime}, \text{a}~\lambda ~\text{becomes}~0,\\
1, & \pi \sqsupset \pi^{\prime},
\end{array}\right.
\end{equation}

where $\pi\leadsto\pi^{\prime}$ denotes that $\pi^{\prime}$ is obtained by changing the flag of one block of $\pi$, and $\pi\sqsupset \pi^{\prime}$ denotes that  $\pi^{\prime}$ is obtained by merging two $0$-blocks in $\pi$.
It is obvious that $\{(N_{t}, M_{t})\}_{t\geq 0}$ is indeed the block counting process of $\{\Pi_{t}^K\}_{t\geq 0}$ for $N_{0}+||M_{0}||_{TV}=K$. 

The state space $\mathcal{P}^{f}_{K}$ is locally compact and Polish since it can be regarded as a closed subset of $\{1,2,\cdots,B_{K}\}\times[0,\infty)^{K}$, where $B_{K}$ is the $K$-th Bell number i.e. the number of different ways to partition a set with $K$-elements. As the result, we prove the existence and uniqueness (in distribution) of a coalescent process $\{\Pi^{\infty}_{t}\}_{t\geq 0}$ which has the same distribution as $\{\Pi^{K}_{t}\}_{t\geq 0}$ when it is restricted to take values in  $\mathcal{P}^{f}_{K}$by a  projective limit argument. Furthermore, as the dual of Theorem \ref{th:scaling}, we show that the scaling limit of the
ancestral process for a sample of size $K$ from the discrete-time model is exactly the continuum seed-bank $K$-coalescent process. It is reasonable to believe that some existing results in the literature on the properties of the seed-bank coalescent process still hold for the more general situation discussed in this paper, such as \textit{not coming down from infinity} demonstrated in \cite{blath2016}. They will be left for future work.

This paper is organized as follows: In the current Section 1, the motivation and main results have been presented. Section 2 serves as a compilation of some notations, definitions and theorems, providing necessary preliminaries for the subsequent sections. In Section 3, equation (\ref{eq:SEE}) is formulated as an SEE, and its well-posedness is proved. Section 4 focuses on the Markov properties of the solution and the corresponding martingale problem formulations, in both cases when the state space is endowed with the original and the weak$^{\star}$ topology. In Section 5, a discrete-time Wright-Fisher type model is constructed, and it is shown that the scaling limit of the allele frequency process is the continuum seed-bank diffusion in the weak$^{\star}$ topology setting. Finally, in Section 6, the dual process of the solution is derived, and then the continuum seed-bank coalescent process is obtained. 

\section{Preliminaries}
Throughout this paper, the filtered probability space $(\Omega, \mathcal{F},P, \mathfrak{F})$ is assumed to be normal i.e. $(\Omega, \mathcal{F},P)$ is complete and the filtration $\mathfrak{F}$ satisfies the usual conditions.
$W$ is a $1$-dimensional standard $\mathfrak{F}$-Brownian motion on $(\Omega, \mathcal{F},P, \mathfrak{F})$ and this is represented by the $5$-tuple $(\Omega, \mathcal{F},P, \mathfrak{F}, W)$. The filtration generated by the process $X$ is denoted by $\{\mathcal{F}^{X}_{t}\}_{t\geq 0}$ of which the augmentation is $\mathfrak{F}^X$=$\{\overline{\mathcal{F}}^{X}_{t+}\}_{t\geq 0}$. In addition, $\mathbb{N}_{0}=\{0,1,\cdots\}$,
 $\mathbb{N}=\{1,2,\cdots\}$ and $\mathbb{R}_{+}=[0,\infty)$. $I$ denotes the indicator function, $\delta_{i,j}$ denotes the Kronecker delta, and $\delta_{\lambda}$ is the Dirac measure at point $\lambda$. $\vee$ and $\wedge$ are taking the maximum and the minimum, respectively, and $x^{+}=x\vee 0$. $\subsetneq$ denotes proper inclusion, $O(\cdot)$ is the big $O$ notation, and $\sum\limits_{i=0}^{-1}=0$ by convention. The constants $c$ and $c^{\prime}$ will always represent the total mass and the first order moment of the measure $\mu$, respectively.
 
For any metric space $V$, the Borel $\sigma$-algebra $\mathcal{B}(V)$ and real-valued functions will be considered, and
$B(V)$ denotes the space of all bounded measurable functions on $V$, $C_{b}(V)$ denotes the space of all bounded continuous functions. When $V$ is locally compact, $C_{0}(V)$ denotes the space of all continuous functions vanishing at infinity, and $C_{c}(V)$ represents those with compact supports. All of the above are equipped with the supremum norm. Especially, for the Euclidean space $\mathbb{R}^{n}$, $C_{c}^{\infty}(\mathbb{R}^{n})$ denotes the space of all smooth functions with compact support, and the support is denoted by $supp$. Note that $C_{c}(0,\infty)$ is dense in $C_{0}(0,\infty)$. In addition, $C([0,T];V)$ and $C_{\mathbb{R}_+}(V)$ are the spaces of $V$-valued continuous paths, and $D_{\mathbb{R}_+}(V)$ is the Skorohod space of $V$-valued c\`{a}dl\`{a}g paths.

When a ``measure" is mentioned, it implies that it is non-negative. $\mathcal{M}(0,\infty)$ is the space of all finite signed measures on $((0,\infty),\mathcal{B}(0,\infty))$ equipped with the total variation norm $||\cdot||_{TV}$, and its elements are Radon measures since $(0,\infty)$ is Polish. For the finite measure $\mu$ on $((0,\infty),\mathcal{B}(0,\infty))$, and $1\leq p\leq \infty$, $L^{p}(\mu)$
is the shorthand for the space $L^{p}((0,\infty),\mathcal{B}(0,\infty),\mu;\mathbb{R})$, and the norm is denoted by $||\cdot||_{L^{p}}$. $L^{\infty}$ is the shorthand for the space $L^{\infty}((0,\infty),\mathcal{B}(0,\infty);\mathbb{R})$ composed of all bounded Borel measurable functions, and the norm is denoted by $||\cdot||_{\infty}$. $l^{1}(w)$ is the space of all $w$-weighted absolutely summable real sequences. $\mathcal{S}$ denotes the space of all simple measurable functions on $((0,\infty),\mathcal{B}(0,\infty))$. $ba((0,\infty),\mathcal{B}(0,\infty))$ denotes the space of all finitely additive set functions on $((0,\infty),\mathcal{B}(0,\infty))$ equipped with the total variation norm. In addition, for any random variable $\zeta$, $\sigma(\zeta)$ denotes the $\sigma$-algebra generated by $\zeta$. $\mathcal{L}(\zeta)$ denotes its distribution. $\sim$ refers to following the distribution, and $\overset{d}=$ refers to of the same distribution. $Bin$ represents the binomial distribution, and $Hyp$ represents the hypergeometric distribution.  

For real-valued Banach spaces $V$ and $W$, $V\cong M$ denotes that $V$ is isometrically isomorphic to $M$. $\langle,\cdot,\rangle:V\times W\rightarrow \mathbb{R}$ denotes the pairing by which the duality between $V$ and $W$ can be defined e.g. $\langle y, h\rangle=\int_{(0,\infty)}y(\lambda)h(\lambda)\mu(d\lambda)$ for $y\in L^{1}(\mu)$ and $h\in L^{\infty}(\mu)$. For a semigroup $\{S(t)\}_{t\geq 0}$ on some Banach space $V$, it is called ``strongly continuous" if $\lim\limits_{t\rightarrow 0}S(t)f=f$ for all $f\in V$, and it is called ``compact'' if $S(t)$ is a compact operator for all $t>0$. The resolvent set of its generator $A$ defined on $D(A)$ will be denoted by $\rho(A)$, and $R(\alpha,A)=(\alpha-A)^{-1}$ for $\alpha\in \rho(A)$ is the resolvent operator. The adjoint semigroup of $\{S(t)\}_{t\geq 0}$ is denoted by $\{S(t)^{\star}\}_{t\geq 0}$, and the adjoint of $A$ is denoted by $A^{\star}$. By Theorem 1.10.4 in \cite{pazy2012semigroups}, $S(t)^{\star}$ is a strongly continuous when it is restricted to $\overline{D(A^{\star})}$, and the generator $A^+$ is defined as follows:
\begin{defn}\label{defn:part}
$A^{+}$ is the part of $A^{*}$ in $\overline{D(A^{\star})}$, which is defined as $D(A^{+})=\{f\in D(A^{\star}): A^{\star}f\in \overline{D(A^{\star})}\}$, and $A^{+}f=A^{*}f$ for $f\in D(A^{+})$.
\end{defn}

Let $V$ be a separable Banach space, suppose that $A$ generates a strongly continuous semigroup $\{S(t)\}_{t\geq 0}$ on $V$, $F$ and $B$ are mappings from $V$ to itself. For the following SEE 
\begin{equation}\label{example}
	dZ_{t}=AZ_{t}dt+F(Z_{t})dt+B(Z_{t})dW_{t}
\end{equation}
on some $(\Omega, \mathcal{F}, P, \mathfrak{F}, W)$ and the state space $V$, various notions of solutions are defined as follows:

\begin{defn}\label{defn:mild}
An $E$-valued progressively measurable process $Z=\{Z_{t}\}_{t\geq 0}$ on some $(\Omega, \mathcal{F}, P, \mathfrak{F}, W  )$ is called a mild solution to the SEE (\ref{example}) if for $t\geq 0$, 
\begin{equation}\label{eq:mild}
Z_{t}=S(t) Z_{0}+\int_{0}^{t} S(t-s) F(Z_{s}) d s+\int_{0}^{t} S(t-s) B(Z_{s}) d W_{s}, a.s. 
\end{equation}
\end{defn}

\begin{defn}\label{defn:weakened}
An $E$-valued progressively measurable process $Z$ on some $(\Omega, \mathcal{F}, P, \mathfrak{F}, W)$ is called a weakened solution to the SEE (\ref{example}) if for $t\geq 0$, $Z$ is Bochner integrable on $(0,t)$, $\int_{0}^{t}Z_{s}ds \in D(A)$, and
\begin{equation}\label{eq:weakened}
Z_{t}=Z_{0}+A \int_{0}^{t} Z_{s} d s+\int_{0}^{t} F(Z_{s}) d s+\int_{0}^{t} B(Z_{s}) d W_{s}, a.s.
\end{equation}
\end{defn}

\begin{defn}\label{defn:weaka}
An $E$-valued progressively measurable process $Z$ on some $(\Omega, \mathcal{F}, P, \mathfrak{F}, W)$ is  called an analytically weak solution to the SEE (\ref{example}) if for $t\geq 0$ and $f\in D(A^{+})$, $Z$ is Bochner integrable on $(0,t)$ and 
\begin{equation}\label{eq:weaka}
\langle Z_{t}, f \rangle=\langle Z_{0}, f \rangle +\int_{0}^{t} \langle Z_{s}, A^{+}f\rangle ds+\int_{0}^{t} \langle F(Z_{s}), f\rangle d s+\int_{0}^{t} \langle B(Z_{s}) d W_{s}, f\rangle, a.s.
\end{equation}
\end{defn}

\begin{defn}\label{label:stronga}
An $E$-valued progressively measurable process $Z$ on some $(\Omega, \mathcal{F}, P, \mathfrak{F}, W)$ is called an analytically strong solution to the SEE (\ref{example}) if for $t\geq 0$, $Z_{t}\in D(A)$, $AZ$ is Bochner integrable on $(0,t)$, and 
\begin{equation}\label{eq:stronga}
Z_{t}=Z_{0}+\int_{0}^{t} AZ_{s} d s+\int_{0}^{t} F(Z_{s}) d s+\int_{0}^{t} B(Z_{s}) d W_{s},a.s.
\end{equation}
\end{defn} 

Throughout this paper, when one of the above four notions is mentioned, it always refers to a probabilistically weak solution, and the terminology \textit{weak solution} is used specifically for both analytically and probabilistically weak solutions. Similarly, when a \textit{strong solution}is mentioned, it refers to a both analytically and probabilistically strong solution. As for the definitions of \textit{strong solution} and \textit{unique strong solution}, see e.g. \cite{ikeda2014stochastic}. Here, some important points are emphasized for readers' convenience: A strong solution is a Wiener functional of the form
$Z_{t}(\omega)=\Phi(Z_{0}, t, W_{\cdot}(\omega))$ for $t\geq 0$, a.s., where the mapping $\Phi$ defined on $V\times\mathbb{R}_{+}\times\mathcal{W}\footnote{$(\mathcal{W},\overline{\mathcal{B}}_{\infty},\mathcal{L}(W))$ denotes the classical Wiener space.}\mapsto V$ is $\mathcal{B}(V)\times\mathcal{B}(\mathbb{R}_{+})\times \overline{\mathcal{B}}_{\infty}/ \mathcal{B}(V)$ measurable. The equation has a unique strong solution if there exists a $\Phi$ satisfying the following properties:
\begin{defn}\label{defn:strongp}~

\begin{enumerate}
\item For any $\mathfrak{F}$-Brownian motion $W$ and any $\mathcal{F}_{0}$-measurable $V$-valued random variable $Z_{0}$, $Z_{t}(\omega)=\Phi(Z_{0}, t, W_{\cdot}(\omega))$ is a probabilistically weak solution to the equation;
\item For any probabilistically weak solution $(Z, W)$, $Z_{t}(\omega)=\Phi(Z_{0}, t, W_{\cdot}(\omega))$
for $t\geq 0$, a.s.
\end{enumerate}	
\end{defn}

\begin{rem}\label{rem:filtration}
By taking the product filtered probability space, it can always be assumed that $W$ is independent of $Z_{0}$. Let $\mathfrak{F}=\{\sigma(Z_{0})\vee \overline{\mathcal{F}}_{t}^{W}\}_{t\geq 0}$, then $Z$ is $\mathfrak{F}$-progressive and $W$ is an $\mathfrak{F}$-Brownian motion.
\end{rem}

In Section $4$, two martingale problem formulations will be derived. Here, suppose that $V$ is a Polish space, $\mathcal{L}$ is a multi-valued operator (see e.g. \cite{ethier2009markov}) on $B(V)$ or $C_{b}(V)$.
\begin{defn}\label{defn:martingale}
A continuous (or c\`{a}dl\`{a}g) $V$-valued process $\{Z_{t}\}_{t\geq 0}$ on some complete probability space $(\Omega, \mathcal{F}, P)$ is called a solution to the $C_{\mathbb{R}_{+}}(V)$- (or $D_{\mathbb{R}_{+}}(V)$-) martingale problem for $(\mathcal{L}, \nu)$ if 
$f(Z_{t})-f(Z_{0})-\int_{0}^{t}g(Z_{s})ds$
is an $\mathfrak{F}^Z$-martingale for any $(f, g)\in \mathcal{L}$ and $Z_{0}\sim \nu$. A $C_{\mathbb{R}_{+}}(V)$- (or $D_{\mathbb{R}_{+}}(V)$-) martingale problem is said to have a unique solution if any two solutions have the same finite-dimensional distribution.
\end{defn}

The following more general definition of Fr\'{e}chet derivatives (see e.g. \cite{lang2012differential}) will also be required:
\begin{defn}\label{defn:frechet}
Let $X$ and $Y$ be two topological vector spaces, a function $f: X\rightarrow Y$ is said to be Fr\'{e}chet differentiable at a point $x\in X$ if there exists a continuous linear operator
$\nabla f(x): X\rightarrow Y$ such that the function
\begin{equation*}
F_{x}(h)=f(x+h)-f(x)-\nabla f(x)h	\end{equation*}
is tangent to $0$ i.e. for every open $0$-neighborhood $W\subseteq Y$, there exists an open 0-neighborhood $V\subseteq X$ and a function $o:\mathbb{R}\rightarrow \mathbb{R}$ satisfying $\lim\limits_{t\rightarrow 0}\frac{o(t)}{t}=0$, such that $F(tV)\subseteq o(t)W$. Inductively, $\nabla^{2}f(x): X\rightarrow L(X, Y)\footnote{The space of all continuous linear operators from $X$ to $Y$}$ and higher order Fr\'{e}chet derivatives can be defined.
\end{defn} 

The following Theorem (see Proposition 8.8.4 in \cite{da2014stochastic}) will be applied in the proof of Proposition \ref{prop:solution}.
\begin{thm}\label{lemma1}
Let $\{S(t)\}_{t\geq 0}$ be a compact semigroup on Banach space $V$, then
for any $0<\frac{1}{p}<\alpha\leq1$, the operator $G_{\alpha}$ defined as
\begin{equation}
G_{\alpha}f(t)=\int_{0}^{t}(t-s)^{\alpha-1}S(t-s)f(s)ds, t\in[0,T],	
\end{equation}
is compact from $L^{p}([0,T];V)$ to $C([0,T];V)$.
\end{thm}

Finally, the following definition and the related Bernstein theorem motivate the equivalence between equation (\ref{eq:inf}) and equation (\ref{eq:v}). 
See \cite{widder2015laplace} for a proof.

\begin{defn}\label{defn:completemonotone}
A function $f:(0,\infty)\mapsto [0,\infty)$ is completely monotone if it is infinitely differentiable and $(-1)^{k}f^{(k)}(t)\geq 0$ for all $t>0$ and $k\in\mathbb{N}$.
\end{defn}

\begin{thm}[\cite{bernstein1929fonctions}]\label{Bernstein}
A function $K:(0,\infty)\mapsto [0,\infty)$ is completely monotone if and only if there exists a locally finite measure $\mu$ such that
\begin{equation*}
K(t)=\int_{[0,\infty)}e^{-\lambda t}\mu(d\lambda).	
\end{equation*}
$\mu$ is finite if and only if $K(0+)<\infty$.
\end{thm}

\section{Continuum seed-bank diffusion: Well-posedness}\label{sec:wellposedness}

Under the condition (\ref{condition}), equation (\ref{eq:SEE}) can be formulated as an SEE in $E=\mathbb{R}\times L^{1}(\mu)$ equipped with the norm $\|\cdot\|=|\cdot|+\|\cdot\|_{L^{1}}$, which is a separable Banach space. 

First, the linear operator $A$ is defined as $D(A)=\mathbb{R}\times\{y\in L^{1}(\mu):\lambda y(\lambda)\in L^{1}(\mu)\}$, and $A(x,y(\lambda))=(-cx,-\lambda y(\lambda))$ for $(x,y(\lambda))\in D(A)$. Since initially it is not clear whether equation (\ref{eq:SEE}) has a solution taking values in 
\begin{equation}
D=[0,1]\times \left\{y:(0, \infty) \rightarrow \mathbb{R} \text { is Borel measurable, and } 0 \leq y \leq 1, \mu \text {-a.e.}\right\},	
\end{equation}
in order to make the equation well-defined, it is necessary to impose a restriction on the coefficients. In other words, the following ``restricted" equation should be examined first:
\begin{equation}\label{eq:restricted}
\left\{\begin{array}{l}
d Z_{t}=A Z_{t} d t+\tilde{F}\left(Z_{t}\right) d t+\tilde{B}\left(Z_{t}\right) d W_{t}, t>0, \\
Z_{0}=\zeta,
\end{array}\right.
\end{equation}
where $\tilde{F}(x, y(\lambda))=(\int_{(0,\infty)}p(y(\lambda))\mu(d\lambda), \lambda p(x))$, $\tilde{B}(x, y(\lambda))=(\sqrt{p(x)(1-p(x))}, 0)$, $p(x)=x^{+}\wedge 1$, and $\zeta$ is an $E$-valued $\mathcal{F}_{0}$-measurable random variable.

Given that $\int_{(0,\infty)}\lambda \mu(d\lambda)<\infty$, there are some simple observations about $D$, $D(A)$ and the coefficients.

\begin{prop}\label{prop:domain}~

\begin{enumerate}
	\item $D\subseteq D(A)$, $D$ is a closed subset of $E$, and $D(A)$ is a dense Borel subset of $E$;
	\item $\tilde{F}$ and $\tilde{B}$ are Lipschitz continuous and $\frac{1}{2}$-H\"{o}lder continuous mappings from $E$ to itself, respectively.
\end{enumerate}
\end{prop}

Moreover, it can be proved that $A$ generates a strongly continuous semigroup
$\{S(t)\}_{t\geq 0}$ on $E$, and in general it is not a compact semigroup.

\begin{prop}\label{prop:semigroup}
$A$ is the generator of the strongly continuous contraction semigroup
\begin{equation}
S(t): (x, y(\lambda))\mapsto (e^{-ct}x, e^{-\lambda t}y(\lambda)), t\geq 0,
\end{equation}
on $E$.
\end{prop}

\begin{prop}\label{prop:compactsemigroup}
The semigroup $\{S(t)\}_{t\geq 0}$ is compact if and only if $\mu=\sum\limits_{i=1}^{n}c_{i}\delta_{\lambda_{i}}$ for $n\in \mathbb{N}_{0}$, or $\mu=\sum\limits_{i=1}^{\infty}c_{i}\delta_{\lambda_{i}}$,
where $c_{i}>0$, $\lambda_{i}\neq\lambda_{j}$ when $i\neq j$, and $\{\lambda_{i}\}_{i\in\mathbb{N}}$ is unbounded.
\end{prop}

See the Appendix for the proofs of Proposition \ref{prop:domain}, \ref{prop:semigroup} and \ref{prop:compactsemigroup}.

\begin{rem}
The purpose of Proposotion \ref{prop:compactsemigroup} is to show that: Even assuming that $\int_{(0,\infty)}\lambda^{2}\mu(d\lambda)<\infty$ and considering the above set-ups in the Hilbert space $\mathbb{R}\times L^{2}(\mu)$, the semigroup $\{S(t)\}_{t\geq 0}$ is generally not compact by a similar proof. It is well-known that equation (\ref{eq:restricted}) has a weak solution taking values in $\mathbb{R}\times L^{2}(\mu)$ if $A$ generates a compact semigroup (see \cite{gkatarek1994weak}), but unfortunately that is not the case here. Therefore, it is not necessary to adopt a Hilbertian setting. Another reason for using the Banach space $\mathbb{R}\times L^{1}(\mu)$ is that $L^{1}(\mu)$ can be isometrically embedded into $\mathcal{M}(0,\infty)$, as will be employed in Section 4.
\end{rem}

The adjoint operator $A^{\star}$ of $A$ needs to be considered for defining 
the analytically weak solution (see Definition \ref{defn:weaka}) to the SEE (\ref{eq:restricted}). Due to the non-reflexivity of $E$, the adjoint semigroup $\{S(t)^{\star}\}_{t\geq 0}$ may not be strongly continuous on $E^{\star}\cong \mathbb{R}\times L^{\infty}(\mu)$, but its restriction $\{S(t)^{+}\}_{t\geq 0}$ to $\overline{D(A^{\star})}$ is. The generator $A^{+}$ of $\{S(t)^{+}\}_{t\geq 0}$ is the part of $A^{\star}$ in $\overline{D(A^{\star})}$ (see Definition \ref{defn:part}), and there are some simple observations about $A^{\star}$ and $A^{+}$, which are proved in the Appendix.

\begin{prop}\label{prop:adjoint}~

\begin{enumerate}
  \item $D(A^{\star})=\mathbb{R}\times\{h \in L^{\infty}(\mu):\lambda h(\lambda)\in L^{\infty}(\mu)\}$, and $A^{\star}(f^{(1)}, f^{(2)}(\lambda))=(-cf^{(1)},-\lambda f^{2}(\lambda))$ for $f=(f^{(1)},f^{(2)})\in D(A^{\star})$;
  \item $D(A^{\star})$ is neither dense and nor closed;
  \item $D(A^{+})=\{(\frac{f^{(1)}}{\alpha+c}, \frac{f^{(2)}(\lambda)}{\alpha+\lambda}): (f^{1}, f^{(2)}(\lambda))\in \overline{D(A^{\star})}, \alpha\in\rho(A)\}$, and $D(A^{\star 2})\subsetneq D(A^{+})\subsetneq D(A^{\star})$.
\end{enumerate}
\end{prop}

The relationship between the different notions of solutions to the SEE (\ref{eq:restricted}) (see Definition \ref{defn:mild}, \ref{defn:weakened} and \ref{defn:weaka}) is as follows, by which the existence of a weak solution can be further proved. It is actually a special case of the main result in \cite{gorajski2014equivalence}. See the Appendix for a direct proof.

\begin{prop}\label{prop:equi}
The following statements are equivalent:
\begin{enumerate}
  \item $Z$ is a mild solution to the SEE (\ref{eq:restricted});
  \item $Z$ is a weakened solution to the SEE (\ref{eq:restricted});
  \item $Z$ is an analytically weak solution to the SEE (\ref{eq:restricted}).
\end{enumerate}
\end{prop}

\begin{rem}\label{rem:1}~

\begin{enumerate}
	\item Any mild solution to the SEE (\ref{eq:restricted}) is of the form 	
\begin{equation*}
\begin{aligned}
Z_{t} &=\left[\begin{array}{c}
e^{-c t}\zeta^{(1)}+\int_{0}^{t} e^{-c(t-s)} \int_{(0, \infty)} p\left(Y_{s}(\lambda)\right) \mu(d \lambda) d s+\int_{0}^{t} e^{-c(t-s)} \sqrt{p\left(X_{s}\right)\left(1-p\left(X_{s}\right)\right)} d W_{s} \\
e^{-\lambda t} \zeta^{(2)}(\lambda)+\int_{0}^{t} e^{-\lambda(t-s)} \lambda p\left(X_{s}\right) ds
\end{array}\right], a.s.,
\end{aligned}
\end{equation*}
and it can be verified directly that $Z$ has a continuous version;
\item When $Z$ takes values in $D$, Proposition \ref{prop:equi} still holds with the restriction ``$p$" in the coefficients removed, and the solution is analytically strong.
\end{enumerate}
\end{rem}

By the equivalence of 1 and 2 in Proposition \ref{prop:equi} and the similar approach in \cite{da2014stochastic}, the following result can be proved:

\begin{prop}\label{prop:solution}
For any $T>0$, the SEE (\ref{eq:restricted}) has a continuous weak solution on $[0, T]$.
\end{prop}

\begin{proof}

We employ the Euler approximation. For fixed $T>0$, define
$$F_{n}:[0,T]\times C([0,T];E)\rightarrow E~\mbox{and}~B_{n}:[0,T]\times C([0,T];E)\rightarrow E$$
as $F_{n}(t, Z)=\tilde{F}(Z(\xi_{n}(t))$ and $B_{n}(t, Z)=\tilde{B}(Z(\xi_{n}(t))$, where $\xi_{n}(t)=\frac{kT}{2^{n}}$ if $\frac{kT}{2^{n}}\leq t<\frac{(k
+1)T}{2^{n}}$ for $k=0,1,\cdots, 2^{n}-1$. Then, there exist unique strong solutions $\{Z_{n}\}_{n\in\mathbb{N}}$ to the following equations, which can be viewed as $C([0,T];E)$-valued random variables since $E$ is separable.
$$Z_{n}(t)=S(t)Z_{0}+\int_{0}^{t}S(t-s)F_{n}(s,Z_{n})ds+\int_{0}^{t}S(t-s)B_{n}(s,Z_{n})dW_{s}.$$
Indeed, $Z_{n}$ is defined inductively as follows:
\begin{eqnarray*}
&&Z_{n}(0)=Z_{0};\text{ For }\frac{kT}{2^{n}}<t\leq \frac{(k+1)T}{2^{n}}, k=0,1,\cdots, 2^{n}-1, \\
&&Z_{n}(t)=S(t-\frac{kT}{2^{n}})Z_{n}(\frac{kT}{2^{n}})+\int_{\frac{kT}{2^{n}}}^{t}S(t-s)\tilde{F}(Z_{n}(\frac{kT}{2^{n}}))ds+\int_{\frac{kT}{2^{n}}}^{t}S(t-s)\tilde{B}(Z_{n}(\frac{kT}{2^{n}}))dW_{s}.
\end{eqnarray*}

For $Z_{n}=(X_{n}, Y_{n})$, we apply Theorem \ref{lemma1} to show that the laws $\{\mathcal{L}((X_{n}, Y_{0}))\}_{n\in\mathbb{N}}$ form a tight family.

It is obvious that $S^{(1)}(t): x\mapsto e^{-ct}x, t\geq 0$ is a compact semigroup on $\mathbb{R}$. If $0<\frac{1}{p}<\alpha<\frac{1}{2}$, then by the stochastic Fubini theorem (see e.g. \cite{da2014stochastic}), we have the factorization formula
$$\int_{0}^{t}S^{(1)}(t-s)B_{n}^{(1)}(s,Z_{n})dW_{s}=\int_{0}^{t}(t-s)^{\alpha-1}S^{(1)}(t-s)K^{(n)}_{\alpha}(s)ds, t\in[0,T],$$
where $B_{n}^{(1)}$ denotes the first component of $B_{n}$, and
$$K^{(n)}_{\alpha}(t)=\frac{sin(\alpha\pi)}{\pi}\int_{0}^{t}(t-s)^{-\alpha}S^{(1)}(t-s)B_{n}^{(1)}(s,Z_{n})dW_{s}.$$
Therefore,
\begin{equation}\label{eq:eq1}
X_{n}(t)=S^{(1)}(t)X_{0}+G_{1}(F_{n}^{(1)}(\cdot, Z_{n}))(t)+G_{\alpha}(K^{(n)}_{\alpha}(\cdot))(t),	
\end{equation}
where $F_{n}^{(1)}$ denotes the first component of $F_{n}$.

For $K^{(n)}_{\alpha}$, by the Burkholder-Davis-Gundy inequality and $\sup\limits_{n\in\mathbb{N}}\sup\limits_{t\in[0,T]}|B_{n}^{(1)}(t,Z_{n})|\leq1$, we have
$$\sup_{n\in\mathbb{N}}E[\int_{0}^{T}|K^{(n)}_{\alpha}(t)|^{p}dt]\leq C,$$
for any $p>2$, where $C$ is a constant depending only on $\alpha,p$ and $T$. In addition, we have
$$\sup_{n\in\mathbb{N}}E[\int_{0}^{T}|F_{n}^{(1)}(t, Z_{n})|^{p}dt]\leq c^{p}T,$$
by $\sup\limits_{n\in\mathbb{N}}\sup\limits_{t\in[0,T]}|F_{n}^{(1)}(t,Z_{n})|\leq c$. Then, by the Chebyshev inequality, for any $\epsilon>0$,
there exists $r>0$ such that
$$\inf_{n\in\mathbb{N}}P(\{||K^{(n)}_{\alpha}||_{L^{p}}\leq r\}\bigcap\{||F_{n}^{(1)}(\cdot,Z_{n})||_{L^{p}}\leq r\})>1-\frac{\epsilon}{2}.$$

Moreover, $\mathcal{L}(X_{0})$ is tight i.e. 
$P(X_{0}\in K_{\epsilon})> 1-\frac{\epsilon}{2}$ for some compact subset $K_{\epsilon}$ of $\mathbb{R}$.

Take $\Omega_{\epsilon}=\{\omega\in\Omega: X_{0}(\omega)\in K_{\epsilon}\}$, by the Arzel\`{a}-Ascoli theorem, we know that the set $\{S^{(1)}(\cdot)X_{0}(\omega): \omega\in\Omega_{\epsilon}\}$ is precompact in $C([0,T];\mathbb{R})$. 

Finally, we take the set
$$K=\{S^{(1)}(\cdot)x_{0}+G_{1}x_{1}(\cdot)+G_{\alpha}x_{2}(\cdot)\in C([0,T];\mathbb{R}): x_{0}\in K_{\epsilon}, ||x_{1}||_{L^{p}}\leq r, ||x_{2}||_{L^{p}}\leq r\}.$$
By (\ref{eq:eq1}) and Theorem \ref{lemma1}, $K$ is precompact, and we have 
\begin{equation}\label{eq:eq2}
\inf\limits_{n\in\mathbb{N}}P(X_{n}\in \overline{K})\geq \inf\limits_{n\in \mathbb{N}} P(\Omega_{\epsilon}\bigcap ||K^{(n)}_{\alpha}||_{L^{p}}\leq r\}\bigcap\{||F_{n}^{(1)}(\cdot,Z_{n})||_{L^{p}}\leq r)\geq1-\epsilon.	
\end{equation}
Since $\mathcal{L}(Y_{0})$ is tight i.e. $P(Y_{0}\in L_{\epsilon})\geq 1-\epsilon$ for some compact subset $L_{\epsilon}$ of $L^{1}(\mu)$, then by (\ref{eq:eq2}), it can be shown that $\{\mathcal{L}((X_{n}, Y_{0}))\}_{n\in\mathbb{N}}$ is also tight.

Now, by the Prokhorov theorem, $\{\mathcal{L}((X_{n},Y_{0}))\}_{n\in\mathbb{N}}$ has a subsequence which weakly converges to a measure $\nu$ on $C([0,T];\mathbb{R})\times L^{1}(\mu)$. Then, by the Skorohod representation theorem (see e.g. \cite{ethier2009markov}), there exists a probability space $(\tilde{\Omega},\tilde{\mathcal{F}},\tilde{P})$ on which there are $C([0,T];\mathbb{R})\times L^{1}(\mu)$-valued random variables $Z_{0}^{(n)}=(\tilde{X}_{n}, \tilde{Y}^{(n)}_{0})\overset{d}=(X_{n}, Y_{0}),n\in\mathbb{N}$ and $(\tilde{X}, \tilde{Y}_{0})\sim \nu$, such that $(\tilde{X}_{n}, \tilde{Y}_{0}^{(n)})\rightarrow (\tilde{X}, \tilde{Y}_{0}), a.s.,$ as $n\rightarrow\infty$.  

Let $\tilde{Y}_{n}(t)=e^{-\lambda t}\tilde{Y}^{(n)}_{0}+\int_{0}^{t}e^{-\lambda(t-s)}\lambda p(\tilde{X}_{n}(s))ds, t\geq 0$ for $n\in\mathbb{N}$, and let $\tilde{Y}(t)=e^{-\lambda t}\tilde{Y}_{0}+\int_{0}^{t}e^{-\lambda(t-s)}\lambda p(\tilde{X}(s))ds$, then $\tilde{Z}_{n}=(\tilde{X}_{n}, \tilde{Y}_{n})$ converges to $\tilde{Z}=(\tilde{X}, \tilde{Y})$, a.s., as $n\rightarrow \infty$. Then, by the definition of the Riemann integral, we have $\tilde{Z}_{n}\overset{d}=Z_{n}$.

For $n\in\mathbb{N}$, by $\sup\limits_{n\in\mathbb{N}}\sup\limits_{t\in[0,T]}|\tilde{B}_{n}^{(1)}(t,Z_{n})|\leq1$, we know that
$$N_{n}(t)=X_{n}(t)-X_{0}+c\int_{0}^{t}X_{n}(\xi_{n}(s))ds-\int_{0}^{t}F_{n}^{(1)}(s, Z_{n}(s))ds,~t\in[0,T],$$
are square integrable martingales with respect to $\{\mathcal{F}_{t}^{Z_{n}}\}_{t\geq 0}$.
Moreover,
$$N_{n}^{2}(t)-\int_{0}^{t}p(X_{n}(\xi_{n}(s))(1-p(X_{n}(\xi_{n}(s))ds,~t\in[0,T],$$
are also square integrable martingales.

Then, by the martingale property, for any $0\leq t_{1}<t_{2}<\cdots<t_{m}\leq s\leq t\leq T$, and any bounded continuous functions $\phi_{i}, 1\leq i\leq m$, we have
\begin{equation}\label{eq1}
E[(N_{n}(t)-N_{n}(s))\prod_{i=1}^{m}\phi_{i}(Z_{n}(t_{i}))]=0,	
\end{equation}
and
\begin{equation}\label{eq2}
E[(N_{n}^{2}(t)-N_{n}^{2}(s)-\int_{s}^{t}p(X_{n}(\xi_{n}(s))(1-p(X_{n}(\xi_{n}(s)))ds)\prod_{i=1}^{m}\phi_{i}(Z_{n}(t_{i}))]=0.
\end{equation}
By the fact that $\tilde{Z}_{n}\overset{d}=Z_{n}$, $n\in\mathbb{N}$ and the definition of the Riemann integral, both (\ref{eq1}) and (\ref{eq2}) also hold for $\tilde{Z}_{n}=(\tilde{X}_{n}, \tilde{Y}_{n})$. 
Therefore, for each $n$,
$$\tilde{N}_{n}(t)=\tilde{X}_{n}(t)-\tilde{X}_{0}+c\int_{0}^{t}\tilde{X}_{n}(\xi_{n}(s))ds-\int_{0}^{t}F_{n}^{(1)}(s, \tilde{Z}_{n}(s))ds,~t\in[0,T],$$ 
is a square integrable martingale with respect to $\{\mathcal{F}_{t}^{\tilde{Z}_{n}}\}_{t\geq 0}$ of which the quadratic variation process is
$$[\tilde{N}_{n}]_{t}=\int_{0}^{t}p(\tilde{X}_{n}(\xi_{n}(s))(1-p(\tilde{X}_{n}(\xi_{n}(s))ds,~t\in[0,T].$$

Let $\tilde{M}_{n}(t)=(\tilde{N}_{n}(t), 0)$, and in order to take the limit, for any given $\alpha>0$, we take $\tilde{M}_{n}^{(\alpha)}(t)=(A-\alpha)^{-1}\tilde{M}_{n}(t)$. By Proposition \ref{prop:equi} and $(A-\alpha)^{-1}A=A(A-\alpha)^{-1}$ on $D(A)$, 
\begin{equation*}
\tilde{M}_{n}^{(\alpha)}(t)=	(A-\alpha)^{-1}\tilde{Z}_{n}(t)-(A-\alpha)^{-1}\tilde{Z}_{0}^{(n)}-A(A-\alpha)^{-1}\int_{0}^{t}\tilde{Z}_{n}(s)ds-
(A-\alpha)^{-1}\int_{0}^{t}\tilde{F}(\tilde{Z}_{n}(\xi_{n}(s)))ds.
\end{equation*}
We denote that $\tilde{M}_{n}^{(\alpha)}(t)=(\tilde{N}_{n}^{(\alpha)}(t), 0)$ for $\tilde{N}_{n}^{(\alpha)}(t)=\frac{1}{\alpha+c}\tilde{N}_{n}(t)$. In the same manner, but from $N_{n}^{(\alpha)}(t)=\frac{1}{\alpha+c}N_{n}(t)$, we have
\begin{equation}\label{eq3}
E[(\tilde{N}^{(\alpha)}_{n}(t)-\tilde{N}^{(\alpha)}_{n}(s))\prod_{i=1}^{m}\phi_{i}(\tilde{Z}_{n}(t_{i}))]=0
\end{equation}
and
\begin{equation}\label{eq4}
E[(\tilde{N}_{n}^{(\alpha)2}(t)-\tilde{N}_{n}^{(\alpha)2}(s)-\frac{1}{(\lambda+c)^{2}}\int_{s}^{t}p(\tilde{X}_{n}(\xi_{n}(s))(1-p(\tilde{X}_{n}(\xi_{n}(s)))ds)\prod_{i=1}^{m}\phi_{i}(\tilde{Z}_{n}(t_{i}))]=0.
\end{equation}
Since $\tilde{Z}_{n}\rightarrow\tilde{Z}$, a.s, on $C([0,T];E)\times L^{1}(\mu)$, then\footnote{Omit ``a.s." from now on}$\int_{0}^{t}\tilde{Z}_{n}(s)ds\rightarrow\int_{0}^{t}\tilde{Z}(s)ds$. By 2 of Proposition \ref{prop:domain}, we know that $\tilde{F}(\tilde{Z}_{n}(\xi_{n}(t)))\rightarrow \tilde{F}(\tilde{Z}(t))$ for $t\in [0,T]$, and thus $\int_{0}^{t}\tilde{F}(\tilde{Z}_{n}(\xi_{n}(s))ds\rightarrow \int_{0}^{t}\tilde{F}(\tilde{Z}(s))ds$ by the dominated convergence theorem. Furthermore, $A(A-\alpha)^{-1}$ is a bounded operator on $E$, hence for any $t\in[0,T]$,
$$\tilde{M}_{n}^{(\alpha)}(t)\rightarrow \tilde{M}^{(\alpha)}(t):=(A-\alpha)^{-1}\tilde{Z}(t)-(A-\alpha)^{-1}\tilde{Z}_{0}-A(A-\alpha)^{-1}\int_{0}^{t}\tilde{Z}(s)ds-
(A-\alpha)^{-1}\int_{0}^{t}\tilde{F}(\tilde{Z}(s))ds,$$
and then $\tilde{N}_{n}^{(\alpha)}(t)\rightarrow \tilde{N}^{(\alpha)}(t)$ if we denote $\tilde{M}^{(\alpha)}(t)=(\tilde{N}^{(\alpha)}(t),0)$.

Now let $n\rightarrow\infty$ in (\ref{eq3}) and (\ref{eq4}), by the uniform integrability of the integrands, we have
\begin{equation*}
E[(\tilde{N}^{(\alpha)}(t)-\tilde{N}^{(\alpha)}(s))\prod_{i=1}^{m}\phi_{i}(\tilde{Z}(t_{i}))]=0
\end{equation*}
and
\begin{equation*}
E[(\tilde{N}^{(\alpha)2}(t)-\tilde{N}^{(\alpha)2}(s)-\frac{1}{(\alpha+c)^{2}}\int_{s}^{t}p(\tilde{X}(s))(1-p(\tilde{X}(s)))ds)\prod_{i=1}^{m}\phi_{i}(\tilde{Z}(t_{i}))]=0,
\end{equation*}
i.e. $\tilde{N}^{(\alpha)}$ is a square integrable martingale with respect to $\{\mathcal{F}_{t}^{\tilde{Z}}\}_{t\geq 0}$ of which the quadratic variation process is
$$[\tilde{N}^{(\alpha)}]_{t}=\frac{1}{(\alpha+c)^{2}}\int_{0}^{t}p(\tilde{X}(s))(1-p(\tilde{X}(s))ds,~t\in[0,T].$$

Finally, by the martingale representation theorem (see e.g. \cite{ikeda2014stochastic}), there exists a filtered probability space $(\hat{\Omega},\hat{\mathcal{F}}, \hat{P}, \{\hat{\mathcal{F}}_{t}\}_{t\geq 0})$ and a standard $\{\mathcal{F}_{t}^{\tilde{Z}}\times\hat{\mathcal{F}}_{t}\}_{t\geq 0}$-Brownian motion $\{\bar{W}_{t}\}$ on $(\tilde{\Omega}\times\hat{\Omega}, \tilde{\mathcal{F}}\times\hat{\mathcal{F}}, \tilde{P}\times
\hat{P})$, such that
\begin{eqnarray*}
  \bar{N}^{(\alpha)}(t)&=&\frac{1}{\alpha+c}\int_{0}^{t}p(\bar{X}(s))(1-p(\bar{X}(s))d\bar{W}_{s},\\
  \bar{M}^{(\alpha)}(t)
  &=&(A-\alpha)^{-1}\bar{Z}(t)-(A-\alpha)^{-1}\bar{Z}_{0}-A(A-\alpha)^{-1}\int_{0}^{t}\bar{Z}(s)ds-
(A-\alpha)^{-1}\int_{0}^{t}\tilde{F}(\bar{Z}(s))ds,
\end{eqnarray*}
where $\bar{Z}(t,\tilde{\omega},\hat{\omega})=\tilde{Z}(t,\tilde{\omega})$, $\bar{M}^{(\alpha)}(t,\tilde{\omega},\hat{\omega})=\tilde{M}^{(\alpha)}(t,\tilde{\omega})$
and $\bar{Z}=(\bar{X}, \bar{Y})$. Therefore,
\begin{eqnarray*}
(A-\alpha)^{-1}\bar{Z}(t)&=&(A-\alpha)^{-1}\bar{Z}_{0}+A(A-\alpha)^{-1}\int_{0}^{t}\bar{Z}(s)ds\\
&&+(A-\alpha)^{-1}\int_{0}^{t}\tilde{F}(\bar{Z}(s))ds+(A-\alpha)^{-1}\int_{0}^{t}\tilde{B}(\bar{Z}(s))d\bar{W}_{s},
\end{eqnarray*}
and then by $(A-\alpha)A=A(A-\alpha)$ on $D(A)$, we have
$$\bar{Z}(t)=\bar{Z}_{0}+A\int_{0}^{t}\bar{Z}(s)ds+\int_{0}^{t}F(\bar{Z}(s))ds+\int_{0}^{t}B(\bar{Z}(s))d\bar{W}_{s}.$$
By Proposition \ref{prop:equi}, $(\tilde{\Omega}\times\hat{\Omega}, \tilde{\mathcal{F}}\times\hat{\mathcal{F}}, \tilde{P}\times
\hat{P}, \mathfrak{F}, \bar{W}, \bar{Z})$ is a continuous weak solution to the SEE (\ref{eq:SEE}) on $[0,T]$, where $\mathfrak{F}$ denotes the augmentation of $\{\mathcal{F}_{t}^{\tilde{Z}}\times\hat{\mathcal{F}}_{t}\}_{t\geq 0}$.
\end{proof}

\begin{rem} In the proof of Proposition \ref{prop:solution}, the tightness is obtained by the factorization method rather than by the Kolmogrov-Chentsov criterion which requires certain moment conditions on $\zeta$, hence one can directly deal with general initial conditions. Moreover, since the restricted coefficients are bounded, a priori estimates are not required.
\end{rem}

By concatenating solutions on consecutive finite intervals, there exists a continuous weak solution to the SEE (\ref{eq:restricted}) on $\mathbb{R}_{+}$. 
Then, following \cite{shiga1980infinite}, when $\mathcal{L}(\zeta)(D)=1$, it can be shown that the solution $Z$ is a $D$-valued process. 

\begin{prop}\label{prop:D}
If $Z_{0}\in D$, a.s., then any weak solution to the SEE (\ref{eq:restricted}) satisfies $Z_{t}\in D$ for 
$t>0$, a.s. 
\end{prop}
\begin{proof}
It is obvious that $Y_{t}(\lambda)=e^{-\lambda t}Y_{0}+\int_{0}^{t}e^{-\lambda(t-s)}p(X_{s})ds$ takes values in $[0,1]$, a.e., a.s., if $X_{t}$ does. Therefore, we only need to show that $X_{t}\in[0,1]$ for any $t>0$, a.s.

For any given $\epsilon>0$, take $F_{\epsilon}(z)=\frac{1}{x+\epsilon}$. Define $\tau_{\delta}=\inf\{t\geq0: X_{t}\leq -\delta\}$ for any $\delta\in(0,\epsilon)$, then for any $s\geq0$, $X_{s}^{\tau_{\delta}}\geq-\delta>-\epsilon$. By It\^{o}'s formula, we have
\begin{eqnarray*}
\frac{1}{X^{\tau_{\delta}}_{t}+\epsilon}&=&\frac{1}{X_{0}+\epsilon}+\int_{0}^{t\wedge\tau_{\delta}}g_{\epsilon}(Z_{s})ds+M_{t}\\	
&=&\frac{1}{X_{0}+\epsilon}+\int_{0}^{t\wedge\tau_{\delta}}
[\frac{1}{(X_{s}+\epsilon)^{3}}p(X_{s})(1-p(X_{s}))-\frac{1}{(X_{s}+\epsilon)^{2}}(\int_{\mathbb{R}_{+}}p(Y_{s}(\lambda))\mu(d\lambda)-cX_{s})]ds+M_{t},
\end{eqnarray*}
where $\{M_{t}\}_{t\geq 0}$ is a continuous martingale such that $M_{0}=0$.

We split the second term as 
$$\int_{0}^{t\wedge\tau_{\delta}}g_{\epsilon}(Z_{s})ds=\int_{0}^{t\wedge\tau_{\delta}}g_{\epsilon}(Z_{s})I_{\{X_{s}\geq0\}}ds+\int_{0}^{t\wedge\tau_{\delta}}
g_{\epsilon}(Z_{s})I_{\{X_{s}<0\}}ds,$$
and note that 
\begin{eqnarray*}
g_{\epsilon}(Z_{s})I_{\{X_{s}\geq0\}}&\leq&\frac{2c\epsilon^{2}+c\epsilon+1}{\epsilon^{3}},\\
g_{\epsilon}(Z_{s})I_{\{X_{s}<0\}}&<&0.
\end{eqnarray*}
Therefore, by taking expectation on both sides, we have
$$E[\frac{1}{X^{\tau_{\delta}}_{t}+\epsilon}]\leq E[\frac{1}{X_{0}+\epsilon}]+C_{\epsilon}t,$$
where $C_{\epsilon}>0$ is a constant only depending on $\epsilon$. By the continuity of $X$, actually $X^{\tau_{\delta}}_{t}(\omega)=-\delta$, and thus
$$P(\tau_{\delta}\leq t)\leq (\epsilon-\delta)(E[F_{\epsilon}(X_{0})]+C_{\epsilon}t).$$
Letting $\delta$ tend to $\epsilon$, we know that $P(\tau_{\epsilon}\leq t)=0$ holds for any $\epsilon>0$ and $t\geq 0$, which then implies  $$P(X_{t}\geq0~\mbox{for any}~t\geq0)=1.$$

Similarly, take $G_{\epsilon}(z)=\frac{1}{1+\epsilon-x}$ and $\tau^{\prime}_{\delta}=\inf\{t\geq0: X_{t}\geq 1+\delta\}$ for any $\delta\in(0,\epsilon)$, we can prove that 
$$P(X_{t}\leq1~\mbox{for any}~t\geq0)=1.$$
The proof is completed.
\end{proof}

\begin{rem}\label{rem:mu}
In Proposition \ref{prop:D}, if $Y_{0}$ does not depend on $\mu$, then 
$0\leq Y_{t}(\lambda)\leq 1$ for all $t>0$ and $\lambda\in(0,\infty)$ i.e. the solution $Z$ takes values in $\mathbb{R}\times L^{\infty}$. This will be used in Section 6.
\end{rem}

Now, by Proposition \ref{prop:D} and Remark \ref{rem:1}, it is known that 
the unrestricted equation (\ref{eq:SEE}) has a $D$-valued continuous weak solution.
The next step is prove the pathwise uniqueness of the solution, which is equivalent to that of SVE (\ref{eq:v}), but here it is proved directly by the classical method in \cite{yamada1971uniqueness}. See the Appendix for the proof.

\begin{prop}\label{prop:pathwise}
The solution to equation (\ref{eq:SEE}) is pathwise unique.
\end{prop}

Finally, the first main result Theorem \ref{th:wellposed} can be proved as follows:

\begin{proof}[\textbf{Proof of Theorem \ref{th:wellposed}}]~

1. Note that the Yamada-Watanabe theorem still holds in this case since the stochastic integral is the classical one, and $D$ is a Polish space, its proof follows that of Theorem 4.1.1 in \cite{ikeda2014stochastic}. Alternatively, by the equivalence between equation (\ref{eq:SEE}) and the SVE (\ref{eq:v}) with a random initial condition, and the fact that equation (\ref{eq:v}) is a path-dependent SDE, we can also conclude that equation (\ref{eq:SEE}) has a $D$-valued continuous unique strong solution, based on Proposition \ref{prop:solution}, \ref{prop:D} and \ref{prop:pathwise}. 

2. As in the proof of Proposition \ref{prop:pathwise} except that we replace $Z$ and $\tilde{Z}$ with
$Z^{1}=(X^{1}, Y^{1})$ and $Z^{2}=(X^{2}, Y^{2})$, we have
\begin{eqnarray*}
&&E[\sup_{t\in[0, T]}\phi_{n}(X^{1}_{t}-X^{2}_{t})]\\
&\leq &E[\phi_{n}(X^{1}_{0}-X^{2}_{0})]+\int_{0}^{T}E[|\int_{(0,\infty)}(Y^{1}_{t}(\lambda)-Y^{2}_{t}(\lambda))\mu(d\lambda)|]dt+c\int_{0}^{T}E[|X^{1}_{t}-X^{2}_{t}|]dt\\
&&+E[\sup_{t\in[0,T]}|\int_{0}^{t}\phi^{\prime}_{n}(X^{1}_{s}-X^{2}_{s})(\sqrt{X^{1}_{s}(1-X^{1}_{s})}-\sqrt{X^{2}_{s}(1-X^{2}_{s})} )dW_{s}|]+\frac{T}{n},
\end{eqnarray*}

and then by the Burkholder-Davis-Gundy inequality, and $\sup\limits_{t\in \mathbb{R}_{+}}|X^{1}_{t}-X^{2}_{t}|\leq 1$, letting $n\rightarrow \infty$, we have
\begin{equation*}
E[\sup_{t\in[0,T]}|X^{1}_{t}-X^{2}_{t}|]\leq E[|X^{1}_{0}-X^{2}_{0}|]+(c+3)\int_{0}^{T}E[\sup_{s\in [0,t]}|X^{1}_{s}-X^{2}_{s}|]dt+\int_{0}^{T}E[||Y^{1}_{t}-Y^{2}_{t}||_{L^{1}}]dt.
\end{equation*}
Moreover, 
\begin{eqnarray*}
	E[||Y^{1}_{t}-Y^{2}_{t}||_{L^{1}}]&=&E[\int_{(0,\infty)}|Y^{1}_{0}-Y^{2}_{0}+\int_{0}^{t}e^{-\lambda(t-s)}\lambda(X^{1}_{s}-X^{2}_{s})ds|\mu(d\lambda)]\\
	&\leq &E[||Y^{1}_{0}-Y^{2}_{0}||_{L^{1}}]+c^{\prime}\int_{0}^{t}E[|X^{1}_{s}-X^{2}_{s}|]ds.
\end{eqnarray*}
Finally, by the Gr\"{o}nwall inequality,
\begin{equation}\label{t3}
E[\sup_{t\in[0, T]}|X^{1}_{t}-X^{2}_{t}|]\leq e^{c^{\prime}T+c+3}(E[|X^{1}_{0}-X^{2}_{0}|]+TE[||Y^{1}_{0}-Y^{2}_{0}||_{L^{1}}]),
\end{equation} 

and note that
\begin{equation}\label{t4}
E[\sup_{t\in[0,T]}||Y^{1}_{t}-Y^{2}_{t}||_{L^{1}}]\leq E[||Y^{1}_{0}-Y^{2}_{0}||_{L^{1}}]+c^{\prime}TE[\sup_{t\in[0,T]}|X^{1}_{t}-X^{2}_{t}|],
\end{equation}

Combining (\ref{t3}) and (\ref{t4}) together, we get the desired result. 	
\end{proof}

\section{Continuum seed-bank diffusion: Markov Property}\label{sec:markov}
The aim of this section is to show the Markov properties of the solution to equation (\ref{eq:SEE}) in both the cases when the state space is endowed with the original and the weak$^{\star}$ topology, and then the corresponding martingale problem formulations will be derived.

Let $Z^{z}$ be the strong solution on some $(\Omega, \mathcal{F}, P, \mathfrak{F}=\{\overline{\mathcal{F}}_{t}^{W}\}_{t\geq 0}, W)$ such that $Z_{0}=z$, a.s. For $f\in B(D)$, define 
\begin{equation}
T_{t}f(z)=E[f(Z^{z}_{t})]	
\end{equation}
for $t\geq 0$ and $z\in D$. $\{T_{t}\}_{t\geq 0}$ are linear contraction mappings from $B(D)$ to itself satisfying the following properties:
\begin{prop}\label{prop:semigroups}~

\begin{enumerate}
  \item Feller property: For fixed $t\geq 0$, if $f\in C_{b}(D)$, then $T_{t}f\in C_{b}(D)$;
  \item Measurability: For fixed $f\in B(D)$, $T_{t}f(z): \mathbb{R}_{+}\times D\rightarrow \mathbb{R}$ is $\mathcal{B}(\mathbb{R}_{+})\times \mathcal{B}(D)$-measurable;
  \item For fixed $f\in B(D)$, $T_{t}f(\cdot): \mathbb{R}_{+}\rightarrow B(D)$ is Borel measurable.
\end{enumerate}
\end{prop}
By the well-posedness of equation (\ref{eq:SEE}) given by Theorem \ref{th:wellposed}, it is classical to show that the solution $Z$ is a $D$-valued time-homogeneous strong Markov process. As the result, by 2 of Proposition \ref{prop:semigroups}, $\{T_{t}\}_{t\geq 0}$ is a measurable contraction semigroup on $B(D)$.

\begin{prop}\label{prop:markovs}
For any $f\in B(D)$, $z\in D$, $t\geq 0$ and $\mathfrak{F}$-stopping time $\tau<\infty$, a.s., 
\begin{equation}
E[f(Z^z_{t+\tau})|\overline{\mathcal{F}}_{\tau}^{W}]=T_{t}f(Z^z_{\tau}), a.s.	
\end{equation}
Consequently, $T_{t+s}f=T_{t}\circ T_{s}f$ for any $t\geq 0$	. 
\end{prop}

See the Appendix for the proofs of Proposition \ref{prop:semigroups} and \ref{prop:markovs}.

\begin{rem}\label{rem:randomi}
When the initial condition is random, replacing the filtration $\mathfrak{F}$ with that in Remark \ref{rem:filtration}, the strong Markov property still holds.
\end{rem}

For the measurable Markov semigroup $\{T_{t}\}_{t\geq 0}$, its full generator 
is a multi-valued operator defined as
\begin{equation}
\hat{\mathcal{L}}=\{(f, g)\in B(D)\times B(D): T_{t}f=f+\int_{0}^{t}T_{s}gds,~t\geq 0\},	
\end{equation}
where $(\int_{0}^{t}T_{s}gds)(z)=\int_{0}^{t}T_{s}g(z)ds$ for $z\in D$.
By 2 of Proposition \ref{prop:semigroups}, $\int_{0}^{t}T_{s}gds\in B(D)$.

The full generator $\hat{\mathcal{L}}$ has the following properties, see e.g. Theorem 1.7.1 in \cite{ethier2009markov}. The same result also holds for the single-valued operator 
\begin{equation}
\tilde{\mathcal{L}}=\{(f, g)\in C_{b}(D)\times C_{b}(D): T_{t}f=f+\int_{0}^{t}T_{s}gds,~t\geq 0\}.	
\end{equation}
\begin{prop}\label{prop:full}~

\begin{enumerate}
  \item $(f, g)\in \hat{\mathcal{L}}$ (or $\tilde{\mathcal{L}}$) if and only if $f(Z^{z}_{t})-f(z)-\int_{0}^{t}g(Z^{z}_{s})ds$ is a martingale for any $z\in D$;
  \item The largest subspace of $B(D)$ (or $C_{b}(D)$) on which $\{T_{t}\}_{t\geq 0}$ is strongly continuous is 
  $\overline{D(\hat{\mathcal{L}})}$ (or $\overline{D(\tilde{\mathcal{L}})}$).
\end{enumerate}
\end{prop}

Now, take the space of functions
\begin{equation}
\tilde{H}=\{h(z)=h_{c}(\langle z, f_{1}\rangle, \cdots, \langle z, f_{n}\rangle), z\in E,~\text{for}~n\in\mathbb{N}, h_{c}\in C_{c}^{\infty}(\mathbb{R}^{n})~\text{and}~f_{i}\in D(A^{+}) \}.\end{equation} 
When it is restricted to $D$, $\tilde{H}|_{D}$ is a subalgebra of $C_{b}(D)$. By the equivalence of 1 and 3 in Proposition \ref{prop:equi}, $2$ of Remark \ref{rem:1} and It\^{o}'s formula, it can be shown that
\begin{equation*}
h(Z^{z}_{t})-h(z)-\int_{0}^{t} \tilde{\mathcal{L}} h(Z^{z}_{s})ds, t\geq 0	
\end{equation*}
is a continuous martingale for any $h\in \tilde{H}|_{D}$ and $z\in D$, where  \footnote{ ``$\cdots$" refers to ``$\langle z, f_{i}\rangle$ for $i=1,2,\cdots, n$", and $\langle \nabla^{2} h(z), B(z) \otimes B(z)\rangle$ denotes $\nabla^{2} h(z)(B(z), B(z))$.}
\begin{equation}\label{gene1}
\begin{aligned}
\tilde{\mathcal{L}} h(z) &=\left\langle z, \sum_{i=1}^{n} \partial_{x_{i}} h_{c}(\cdots) A^{+} f_{i}\right\rangle+\left\langle F(z), \sum_{i=1}^{n} \partial_{x_{i}} h_{c}(\cdots) f_{i}\right\rangle+\frac{1}{2} \sum_{i,j=1}^{n} \partial_{x_{i} x_{j}} h_{c}(\cdots) x(1-x) f_{i}^{(1)}f_{j}^{(1)}\\
&=\langle A z, \nabla h(z)\rangle+\langle F(z), \nabla h(z)\rangle+\frac{1}{2}\left\langle \nabla^{2} h(z), B(z) \otimes B(z)\right\rangle.
\end{aligned}
\end{equation}
Since $D$ is bounded, and $F, B$ are bounded continuous mappings on $D$, then $\tilde{\mathcal{L}} h\in C_{b}(D)$. Therefore, by 1 of Proposition \ref{prop:full}, $\tilde{H}|_{D}$ is contained in $D(\tilde{\mathcal{L}})$\footnote{The same notation has been used for $\tilde{\mathcal{L}}$ and its restriction to $\tilde{H}|_{D}$. The same applies to $\mathcal{L}$ below.}.
Conversely, by the martingale representation theorem, any solution to the $C_{\mathbb{R}_{+}}(D)$-martingale problem for $(\tilde{\mathcal{L}}, \nu)$ is a weak solution to equation (\ref{eq:SEE}).

\begin{rem}\label{rem:poly}
In the definition of $\tilde{H}$, since $\langle z, f_{i} \rangle$ for $i=1,2,\cdots,n$ are bounded, then the space of $h_{c}$ can be replaced by $C^{\infty}(\mathbb{R})$ or multivariate polynomials. Consider the subalgebra $\tilde{H}_{p}$ generated by $\{\langle z, f\rangle: z\in E, f_{i}\in D(A^{+})\}$, it can be verified directly that $\tilde{\mathcal{L}}$ maps $\tilde{H}_{p}$ to itself. Consequently, the solution $Z$ is an infinite-dimensional polynomial process as defined in \cite{cuchiero2021infinite}. 
\end{rem}

The following martingale problem formulation serves as a summary of the above discussion. See the Appendix for the proof.

\begin{prop}\label{prop:mps} 
For any initial distribution $\nu$ on $D$, the $C_{\mathbb{R}_{+}}(D)$-martingale problem for $(\tilde{\mathcal{L}}, \nu)$ has a unique solution, and the solution $Z^{\nu}\sim P_{\nu}$ is an $\mathfrak{F}^{Z^{\nu}}$-strong Markov process. Denote the solution for $(\tilde{\mathcal{L}}, \delta_{z})$ as $P_{z}$, then $P_{z}(B)$ is Borel measurable in $z$ for each $B\in \mathcal{B}(C_{\mathbb{R}_{+}}(D))$ and $P_{\nu}=\int_{D}P_{z}d\nu$.
\end{prop}

By 2 of Proposition \ref{prop:full}, if $D(\mathbb{\tilde{\mathcal{L}}})$ is dense in $C_{b}(D)$, then the Markov semigroup $\{T_{t}\}_{t\geq 0}$ is strongly continuous on $C_{b}(D)$. However, this is not clear as the following Proposition indicates that the state space $D$ endowed with the subspace topology of $E$ is generally not locally compact. More importantly, the lack of local compactness brings about technical difficulties in the proofs of certain approximation problems. For example, when $\mu$ is not absolutely continuous, there is no readily applicable precompact criterion in $L^{1}(\mu)$.

\begin{prop}\label{prop:original}
$D$ is a locally compact subset of $E$ if and only if $\mu=\sum\limits_{i=1}^{\infty}c_{i}\delta_{\lambda_{i}}$, where $c_{i}>0$ and $\lambda_{i}\neq\lambda_{j}$ when $i\neq j$. 
\end{prop}

\begin{proof}
The sufficient part follows immediately from a weighted version of Theorem 4 in \cite{hanche2010kolmogorov}, see also the proof of Proposition \ref{prop:compactsemigroup}. Alternatively, since $D$ endowed with the subspace topology of $\mathbb{R}\times l^{1}(w)$ is homeomorphic to $[0,1]^{\mathbb{N}_{0}}$ endowed with the product topology, then the Tychonoff theorem leads to the conclusion.

For the necessity, when $\mu$ is not discrete i.e. it has a non-atomic component $\mu_{c}\neq 0$. For any $\epsilon>0$ and $n\in\mathbb{N}$, define $$f^{\epsilon}_n(x)= \begin{cases}\epsilon, & x \in \bigcup\limits_{k=0}^{2^{n-1}-1}(a^{n}_{2k}, a^{n}_{2k+1}],\\ 0, & x \in\bigcup\limits_{k=1}^{2^{n-1}-1}(a^{n}_{2k-1}, a^{n}_{2k}]\bigcup\left(a^{n}_{2^{n}-1}, \infty\right),\end{cases}$$
where for each $n$, $a^{n}_{0}=0$ and $a^{n}_{k}$ satisfies $\mu_{c}((0, a^{n}_{k}])=\frac{k}{2^{n}}\mu_{c}(0,\infty)$. It is obvious that $$||f^{\epsilon}_{i}-f^{\epsilon}_{j}||_{L^{1}}\geq ||f^{\epsilon}_{i}-f^{\epsilon}_{j}||_{L^{1}(\mu_{c})}=\frac{\epsilon}{2}\mu_{c}(0,\infty)>0,~\text{for any}~i\neq j,$$
and $\sup\limits_{n\in\mathbb{N}}||f^{\epsilon}_{n}||_{L^{1}}\leq c\epsilon$. Since $\epsilon$ can be taken arbitrarily small, this implies that $0\in D$ has no compact neighborhood, thus in this case $D$ is not locally compact.   
\end{proof}

In order to circumvent the challenges posed by Proposition \ref{prop:original}, $D$ will be isometrically embedded into $\mathbb{R}\times\mathcal{M}(0,\infty)$ through the mapping $i_{\mu}: (x, f)\mapsto (x, f.\mu)$. Recall that $\mathcal{M}(0,\infty)$ is the space of all finite signed measures on $(0,\infty)$ equipped with the total variation norm, and $\mathcal{M}(0,\infty)\cong C_{0}(0,\infty)^{\star}$ by the Riesz-Markov-Kakutani representation theorem (see e.g. \cite{cohn2013measure}). In fact, any $y\in L^{1}(\mu)$ can be viewed as a continuous linear functional $T_{y}$ on $C_{0}(0, \infty)$ by
$\langle T_{y}, h\rangle=\int_{(0, \infty)}h(\lambda)y(\lambda)\mu(d\lambda)$
i.e. a finite signed measure $y.\mu(\cdot)=\int_{\cdot}y(\lambda)\mu(d\lambda)$.
$y$ is the Radon-Nikodym derivative of $y.\mu$ with respect to $\mu$, and 
$\|y.\mu\|_{TV}=\|y\|_{L^{1}}$.

Since the image $i_{\mu}(D)$ is contained in the closed ball 
\begin{equation*}
B_{1+c}=\{z=(x, y.\mu)\in \mathbb{R}\times\mathcal{M}(0,\infty):|x|+||y.\mu||_{TV}\leq 1+c\},	
\end{equation*}
then by the Alaoglu theorem (see e.g. \cite{schaefer1971locally}), it is weak$^{\star}$-relatively compact. Moreover, $i_{\mu}(D)$ is metrizable under the weak$^{\star}$ topology. Actually, by Theorem 3.29 in \cite{brezis2011functional}, the metric $d$ can be defined as:
\begin{equation}\label{eq:metric}
d(z_{1}, z_{2})=|x_{1}-x_{2}|+\sum_{n=1}^{\infty}\frac{1}{2^{n}}|\langle y_{1}-y_{2}, f_{n}\rangle|,
\end{equation}
for $z_{1}=(x_{1},y_{1}.\mu)$, $z_{2}=(x_{2},y_{2}.\mu)$, and a contable dense subset $\{f_{n}\}_{n\in\mathbb{N}}$ of the closed unit ball in $C_{0}(0, \infty)$. 

\begin{rem}\label{rem:5}~

\begin{enumerate}
	\item The weak$^{\star}$ convergence on $\mathcal{M}(0,\infty)$ is also referred to as the vague convergence;
	\item By the Dunford-Pettis theorem (see e.g. \cite{kallenberg1997foundations}) and the Eberlin-\^{S}mulian Theorem (see e.g. \cite{schaefer1971locally}), $D$ is also weakly precompact in $E$, but the closed balls are not metrizable under the weak topology since $E^{\star}\cong \mathbb{R}\times L^{\infty}(0,\infty)$ is not separable.
\end{enumerate}
\end{rem} 

Regarding $D$ and $i_{\mu}(D)$ as the same, in the subsequent text, the state space $D$ will be replaced by $(D, d)$ i.e. the original topology is replaced by the weak$^{\star}$ topology.

There are some observations about the new state space $(D, d)$.

\begin{prop}\label{prop:wstate}~

\begin{enumerate}
	\item $D$ is weak$^{\star}$ closed in $\mathbb{R}\times\mathcal{M}(0, \infty)$, hence $(D, d)$ is a compact metric space;
	\item The Borel $\sigma$-algebra on $(D, d)$ is the same as that on $D$;
	\item The vague convergence on $\left\{y:(0, \infty) \rightarrow \mathbb{R} \text { is measurable and } 0 \leqslant y \leqslant 1, \mu \text {-a.e.}\right\}$ is equivalent to the weak convergence of measures. 
\end{enumerate}
\end{prop}
\begin{proof}~

1. Suppose that $\{(f_{n},g_{n}.\mu)\}_{n\in\mathbb{N}}\subseteq (D,d)$ weak$^{*}$ converges to $(f,\nu)\in \mathbb{R}\times\mathcal{M}(0,\infty)$, then $0\leq f\leq 1$. For $\nu$, note that
$\{g_{n}\}_{n\in \mathbb{N}}$ is uniformly integrable, thus by 2 of Remark \ref{rem:5}, there exists a subsequence $\{g_{n_{k}}\}_{k\in\mathbb{N}}$ weakly converges in $L^{1}(\mu)$ to some $g$ i.e. $\int g_{n_{k}}h d\mu\rightarrow \int gh d\mu$ as $k\rightarrow \infty$ for any $h\in L^{\infty}(\mu)$. Since $\{g_{n_{k}}.\mu\}_{k\in\mathbb{N}}$ also vaguely converges to $g.\mu$, by the uniqueness, we know that $\nu=g.\mu$. Then, we only need to show that $0\leq g\leq 1,\mu-a.e.$ Take $h_{n}=I_{\{g \leq -\frac{1}{n}\}}$ for $n\in\mathbb{N}$, as the limit, we have $\int h_{n}g d\mu\geq 0$, hence $-\frac{1}{n}\mu\{g\leq-\frac{1}{n}\}\geq 0$. As the result, $\mu\{g\leq-\frac{1}{n}\}=0$, which implies that $g\geq 0,\mu-a.e.$ Similarly, by taking $h_{n}^{\prime}=I_{\{1-g \leq -\frac{1}{n}\}}$, we know that $g\leq 1, \mu-a.e.$.

2. It is obvious that $\mathcal{B}(D,d)\subseteq \mathcal{B}(D)$. For ``$\supseteq$", note that $D$ is separable, then by the isometry, the image $i_{\mu}(D)$ is also separable which implies that it has a countable topological basis composed of open balls. Due to the weak$^{\star}$-lower semicontinuity of the norm, each open ball $B_{r}=\{z=(x, y.\mu):|x|+||y.\mu||_{TV}<r\}$ is $\mathcal{B}(D,d)$-measurable, then the result follows.

3. Since $\mu$ is tight, by definition and the boundedness, $\{y_{n}.\mu\}_{n\in\mathbb{N}}$ is also tight, then the conclusion follows from a slight modification of Lemma 4.20 in \cite{kallenberg1997foundations}.
\end{proof}

Convergence in norm on $D$ implies weak$^{\star}$-convergence on $(D,d)$, hence the strong solution $Z$ given by Theorem \ref{th:wellposed} is also a continuous $(D, d)$-valued process. Moreover, the continuity with respect to the initial condition still holds as follows:

\begin{prop}\label{prop:weakc}
On some $(\Omega, \mathcal{F}, P, W)$, let $\{Z^{n}=(X^{n}, Y^{n})\}_{n\in\mathbb{N}}$ be a sequence of strong solutions to equation (\ref{eq:SEE}) with initial conditions $Z^{n}_{0}\in (D, d)$, and let $Z=(X, Y)$ be the strong solution with initial condition $Z_{0}\in (D, d)$. If $\lim\limits_{n\rightarrow \infty}v(Z^{n}_{0}, Z_{0})=0$, a.s., then for any $t\geq 0$, \begin{equation}
	\lim_{n\rightarrow \infty}E[d(Z^{n}_{t}, Z_{t})]=0.
\end{equation}
\end{prop}

\begin{proof}
Fix $t\in [0, T]$, for any $g\in C_{0}(0, \infty)$, we have
$$\left\langle Y^{n}_t-Y_t, g\right\rangle=\langle Y_0^n-Y_0, e^{-\lambda t}g(\lambda)\rangle+\int_0^t\langle\lambda e^{-\lambda(t-s)}, g\rangle\left(X_s^n-X_s\right) d s,
$$ 
thus
\begin{equation}\label{eq62}
E[|\langle Y_t^n-Y_t, g\rangle|] \leq E[|\langle Y_{0}^{n}-Y_{0}, e^{-\lambda t} g(\lambda)\rangle|]+c^{\prime}\|g\| \int_0^t E\left[\left|X_s^n-X_s\right|\right] ds.
\end{equation}

Since $\mu$ is tight, for any $\epsilon>0$, there exists a compact subset $K_{\epsilon}$ of $(0,\infty)$ such that
$\mu(K^{c}_{\epsilon})<\epsilon$. Take an open subset $O_{\epsilon}$ satisfying $K_{\epsilon}\subsetneq O_{\epsilon}\subsetneq (0,\infty)$, then there exists a partition of unity $\{f_{\epsilon}, f^{c}_{\epsilon}\}\subseteq C^{\infty}_{c}(0,\infty)$ such that $f_{\epsilon}\wedge f^{c}_{\epsilon}\geq 0$, supp $f_{\epsilon}\subseteq O_{\epsilon}$, supp $f^{c}_{\epsilon}\subseteq K^{c}_{\epsilon}$, and $f_{\epsilon}+f^{c}_{\epsilon}=1$. 

As the proof of 2 in Theorem \ref{th:wellposed}, and taking the notations therein, we have
\begin{eqnarray*}
\phi_{N}(X^{n}_{t}-X_{t})&=&\int_{0}^{t}\phi^{\prime}_{N}(X^{n}_{s}-X_{s})\langle Y^{n}_{s}-Y_{s}, f_{\epsilon} \rangle ds+\int_{0}^{t}\phi^{\prime}_{N}(X^{n}_{s}-X_{s})\langle Y^{n}_{s}-Y_{s}, f^{c}_{\epsilon} \rangle ds\\
&&-c\int_{0}^{t}\phi^{\prime}_{N}(X^n_{s}-X_{s})(X^n_{s}-X_{s})ds+\int_{0}^{t}\phi^{\prime}_{N}(X^{n}_{s}-X_{s})(\sqrt{X^{n}_{s}(1-X^{n}_{s})}-\sqrt{X_{s}(1-X_{s})} )dW_{s}\\
&&+\frac{1}{2}\int_{0}^{t}\phi_{N}^{\prime\prime}(X^{n}_{s}-X_{s})|\sqrt{X^{n}_{s}(1-X^{n}_{s})}-\sqrt{X_{s}(1-X_{s})}|^{2}ds,
\end{eqnarray*}
Then, by $\sup\limits_{n\in\mathbb{N}}\sup\limits_{t\geq 0}|Y^{n}_{t}-Y_{t}|\leq 1$, 
\begin{equation*}
E[\phi_{N}(X^{n}_{t}-X_{t})]\leq \int_{0}^{t}E[|\langle Y^{n}_{s}-Y_{s}, f_{\epsilon}\rangle|]ds+\epsilon t+ c\int_{0}^{t}E[|X^{n}_{s}-X_{s}|]ds+\frac{t}{N}.
\end{equation*}
 
Letting $N\rightarrow \infty$, by the Gr\"{o}nwall inequality, we have
\begin{equation}\label{eq63}
E[|X^{n}_{t}-X_{t}|]\leq M_{\epsilon}(t)+c\int_{0}^{t}e^{c(t-s)}M_{\epsilon}(s)ds,
\end{equation}  
where $M_{\epsilon}(t)=\int_{0}^{t}E[|\langle Y^{n}_{s}-Y_{s}, f_{\epsilon}\rangle|]ds+\epsilon t$.

By (\ref{eq63}), taking $g=f_{\epsilon}$ in (\ref{eq62}) yields that
\begin{eqnarray*}
E[|\langle Y_t^n-Y_t, f_{\epsilon}\rangle|] \leq N^{(n)}_{\epsilon}(t)+c^{\prime}\int_0^t (M_{\epsilon}(s)+c\int_{0}^{s}e^{c(s-u)}M_{\epsilon}(u)du)ds,
\end{eqnarray*}
where $N^{(n)}_{\epsilon}(t)=E[|\langle Y_{0}^{n}-Y_{0}, e^{-\lambda t} f_{\epsilon}(\lambda)\rangle|]$.

After direct calculations, we have
$$E[|\langle Y_t^n-Y_t, f_{\epsilon}\rangle|]\leq (N^{(n)}_{\epsilon}(t)+C\epsilon)+C\int_{0}^{t}E[|\langle Y_s^n-Y_s, f_{\epsilon}\rangle|]ds,$$
where $C$ is a constant depending only on $c, c^{\prime}$ and $T$, and it may be different from line to line.

By the Gr\"{o}nwall inequality again, we know that
\begin{equation*}
E[|\langle Y_t^n-Y_t, f_{\epsilon}\rangle|]\leq (N^{(n)}_{\epsilon}(t)+C\epsilon)+C\int_{0}^{t}e^{C(t-s)}(N^{(n)}_{\epsilon}(s)+C\epsilon)ds.
\end{equation*}
Therefore, by the definition, 
\begin{equation}\label{eq64}
M_{\epsilon}(t)\leq C\int_{0}^{t}N^{(n)}_{\epsilon}(s)ds+C\epsilon.
\end{equation}
Now, plugging (\ref{eq64}) into (\ref{eq63}), we have
$$E[|X^{n}_{t}-X_{t}|]\leq C\int_{0}^{t}N^{(n)}_{\epsilon}(s)ds+C\epsilon.$$
By the condition and the dominated convergence theorem, for fixed $\epsilon>0$, we know that $N^{(n)}_{\epsilon}(t)\rightarrow 0$ for any $t\geq 0$ as $n\rightarrow \infty$. Note that $\sup\limits_{n\in\mathbb{N}}\sup\limits_{t\in [0,T]}|N^{n}_{\epsilon}(t)|\leq c$, by the dominated convergence theorem again, we have $\lim\limits_{n\rightarrow\infty}E[|X^{n}_{t}-X_{t}|]\leq C\epsilon$. Since $\epsilon$ is arbitrary, the convergence of the first component is proved.

Finally, by (\ref{eq:metric}), we have
$$E[d(Z^{n}_{t}, Z_{t})]\leq E[|X^{n}_{t}-X_{t}|]+\sum_{m=1}^{\infty}\frac{1}{2^{m}}E[|\langle Y^{n}_{t}-Y_{t}, f_{m}\rangle|],$$
where $\{f_{m}\}_{m\in\mathbb{N}}$ is a countable dense subset of the closed unit ball in $C_{0}(0,\infty)$. For each $m$, by (\ref{eq62}) and the fact that $\lim\limits_{n\rightarrow\infty}E[|X^{n}_{t}-X_{t}|]=0$, we know that $\lim\limits_{n\rightarrow\infty}E[|\langle Y^{n}_{t}-Y_{t}, f_{m}\rangle|]=0$. By the discrete version of the dominated convergence theorem, the desired result follows.
\end{proof}

By Proposition \ref{prop:weakc} and the dominated convergence theorem, the following operator still satisfies the Feller property i.e. it maps $C(D,d)$ to $C(D,d)$:
\begin{equation}
P_{t}f(z)=E(f(Z^{z}_{t})),	
\end{equation}
for $t\geq 0$, $z\in(D, d)$, and $f\in B(D, d)$, where $\{Z^{z}_{t}\}_{t\geq 0}$ is the $(D, d)$-valued strong solution such that $Z_{0}=z$, a.s.

Similar to Proposition \ref{prop:markovs}, it is classical to show that the $(D, d)$-valued solution is strong Markov with respect to $\{P_{t}\}_{t\geq 0}$. Actually, by 2 of Proposition \ref{prop:wstate}, all concepts related to the measurability remain unchanged when the state space $D$ is replaced by $(D,d)$. For example, the functional $\Phi$ in Definition \ref{defn:strongp} is $\mathcal{B}(D, d)\times \mathcal{B}(\mathbb{R}_{+}) \times \overline{\mathcal{B}}_{\infty}/\mathcal{B}(D,d)$-measurable, and any $(D, d)$-valued progressively measurable process is a $D$-valued progressively measurable process. 

\begin{prop}\label{prop:markovw}
For any $f\in B(D, d)$, $z\in (D, d)$, $s\geq 0$ and $\mathfrak{F}=\{\overline{\mathcal{F}}_{t}^{W}\}_{t\geq 0}$-stopping time $\tau<\infty$, a.s., we have
\begin{equation}
E[f(Z_{s+\tau}^{z})|\overline{\mathcal{F}}_{\tau}^{W}]=P_{s}f(Z^{z}_{\tau}),~a.s.
\end{equation}
Consequently, $P_{t+s}f=P_{t}\circ P_{s}f$ for any $t\geq 0$	.
\end{prop}

\begin{rem}\label{rem:88}
Parallel to Remark \ref{rem:randomi}, if the initial condition is random, replacing the filtration $\mathfrak{F}$ with that in Remark \ref{rem:filtration}, the strong Markov property still holds.
\end{rem}

Now, since the strong Markov process $\{Z^{z}_{t}\}_{t\geq 0}$ takes values in a compact metric space $(D,d)$, and the Markov semigroup $\{P_{t}\}_{t\geq 0}$ has the Feller property, then it is well-known that $\{P_{t}\}_{t\geq 0}$ is strongly continuous on $C(D, v)$(see also Proposition \ref{prop:stone} below). As the result, the following infinitesimal generator $\mathcal{L}$ can be defined. Moreover, Proposition \ref{prop:full} still holds. 
\begin{eqnarray}
D(\mathcal{L}) &=&\left\{f \in C(D,d): \lim_{t\rightarrow 0} \frac{P_{t} f-f}{t} \text { exists in } C(D,d)\right\}, \\
\mathcal{L}f &=&\lim_{t \rightarrow 0} \frac{P_{t} f-f}{t} \text { for } f \in D(\mathcal{L}).\nonumber
\end{eqnarray}
Take the space of functions 
\begin{equation}
H=\left\{h(z)=h_{c}\left(\left\langle z, f_{1}\right\rangle, \cdots,\left\langle z, f_{n}\right\rangle\right), z \in \mathbb{R}\times \mathcal{M}(0, \infty) \text {, for } n \in \mathbb{N}, h_{c} \in C_{c}^{\infty}\left(\mathbb{R}^{n}\right) \text { and } f_{i} \in \mathbb{R}\times C_{c}(0, \infty)\right\},	
\end{equation}
and then define the notion of weak$^{\star}$ solution as follows:

\begin{defn}\label{defn:weakstar}
A $(D, d)$-valued progressively measurable process $Z$ on $(\Omega,\mathcal{F}, P, \mathfrak{F}, W)$ is  called a weak$^{\star}$ solution if for $t\geq 0$ and $f\in \mathbb{R}\times C_{0}(0, \infty)$,
\begin{equation}\label{eq:weakstardefn}
\langle Z_{t}, f \rangle=\langle Z_{0}, f \rangle +\int_{0}^{t} \langle A Z_{s}, f\rangle ds+\int_{0}^{t} \langle F\left(Z_{s}\right), f\rangle d s+\int_{0}^{t} \langle B\left(Z_{s}\right) d W_{s}, f\rangle,a.s.
\end{equation}
\end{defn}

\begin{rem}
In equation (\ref{eq:weakstardefn}), $\langle z, f \rangle=xg+\int_{(0, \infty)}y(\lambda)h(\lambda)\mu(d\lambda)$ for $z=(x, y)\in E$, and $f=(g, h)\in \mathbb{R}\times C_{0}(0, \infty)$.
\end{rem}

Parallel to Proposition \ref{prop:equi}, the following statement is trivial:

\begin{prop}\label{prop:weakk1}
$Z$ is a $(D, d)$-valued analytically strong solution if and only if $Z$ is a weak$^{\star}$ solution.
\end{prop}

Consequently, as before, it can be shown that 
\begin{equation*}
h(Z_{t}^{z})-h(z)-\int_{0}^{t}\mathcal{L}h(Z^{z}_{s})ds	
\end{equation*}
is a continuous martingale for any $h\in H|_{(D, d)}$ and $z\in (D, d)$, where 
\begin{equation}\label{eq:generatorw}
\begin{aligned}
\mathcal{L}h(z) &=\left\langle Az, \sum_{i=1}^{n} \partial_{x_{i}} h_{c}(\cdots) f_{i}\right\rangle+\left\langle F(z), \sum_{i=1}^{n} \partial_{x_{i}} h_{c}(\cdots) f_{i}\right\rangle+\frac{1}{2} \sum_{i, j=1}^{n} \partial_{x_{i} x_{j}} h_{c}(\cdots) x(1-x) f_{i}^{1} f_{j}^{1}.
\end{aligned}
\end{equation}

By Definition \ref{defn:frechet}, the following lemma can be verified directly, hence
\begin{equation}
\mathcal{L}h(z) =\langle A z, \nabla h(z)\rangle+\langle F(z), \nabla h(z)\rangle+\frac{1}{2}\left\langle\nabla^{2} h(z), B(z) \otimes B(z)\right\rangle.	
\end{equation}
See the Appendix for the proof.

\begin{lemma}\label{lem:derivative}
For any $h\in H$, $\nabla h(z)=\sum\limits_{i=1}^{n} \partial_{x_{i}} h_{c}(\langle z, f_{1}\rangle \cdots \langle z, f_{n}\rangle) f_{i}$, and 
$$\nabla^{2}h(z)(z_{1}, z_{2})=\sum_{i, j=1}^{n} \partial_{x_{i} x_{j}} h_{c}(\langle z, f_{1}\rangle \cdots \langle z, f_{n}\rangle)\langle z_{1}, f_{i}\rangle \langle z_{2}, f_{j}\rangle,$$ for $z_{1}, z_{2}\in \mathbb{R}\times \mathcal{M}(0, \infty)$.
\end{lemma}

Note that in the definition of $H$, $f_{i}\in\mathbb{R}\times C_{c}(0,\infty)$, and thus $\mathcal{L}h\in C(D,d)$ for any $h\in H|_{(D,d)}$. Then, based on the above statements, it is known that $H|_{(D,d)}$ is contained in $D(\mathcal{L})$. Conversely, by the isometry and the martingale representation theorem, any solution to the $C_{\mathbb{R}_{+}}(D,d)$-martingale problem for $(\tilde{\mathcal{L}}, \nu)$ is a weak solution to equation (\ref{eq:SEE}).

The following martingale problem formulation serves as a summary of the above discussion. The proof is similar to that of Proposition \ref{prop:mps}.

\begin{prop}\label{prop:mpw}
For any initial distribution $\nu$ on $(D, d)$, the $C_{\mathbb{R}_{+}}(D,d)$-martingale problem for $(\mathcal{\bar{L}}, \nu)$ has a unique solution, and the solution $Z^{\nu}\sim P_{\nu}$ is an $\{\overline{\mathcal{F}}_{t+}^{Z^{\nu}}\}$-strong Markov process. Denote $P_{z}$ the solution for $(\mathcal{\bar{L}}, \delta_{z})$, then $P_{z}(B)$ is Borel measurable in $z$ for each $B\in \mathcal{B}(C_{\mathbb{R}_{+}}(D,d))$ and $P_{\nu}=\int_{(D, d)}P_{z}d\nu$.
\end{prop}

With 2 of Proposition \ref{prop:full} applied to $\mathcal{L}$, the fact that $\{P_{t}\}_{t\geq 0}$ is strongly continuous on $C(D, d)$ can also be proved by showing that $H|_{(D,d)}$ is dense in $C(D, d)$.

\begin{prop}\label{prop:stone}
$\{P_{t}\}_{t\geq 0}$ is strongly continuous on $C(D, d)$.
\end{prop} 
\begin{proof}Since $H|_{(D,d)}$ is a subalgebra of $C(D,d)$ and it vanishes nowhere, by the Stone-Weierstrass theorem, we only need to show that $H|_{(D,d)}$ separates points. For $x, y\in (D,d)$, if $h(x)=h(y)$ for any $h\in H|_{(D,d)}$, then for any $f\in \mathbb{R}\times C_{c}(0,\infty)$, we have $\langle x, f\rangle =\langle y, f\rangle$. In fact, given $f$, $\langle x, f\rangle$ and $\langle y, f\rangle$ are contained in some compact subset, thus we can take $h_{c}$ as a smooth bump function. Finally, note that $C_{c}(0,\infty)$ is dense in $C_{0}(0,\infty)$, hence $x=y$.
\end{proof}

Finally, it is obvious that $\{P_{t}\}_{t\geq 0}$ is a positive contraction semigroup and $\mathcal{L}1=0$, hence by definition, the $(D,d)$-valued solution $Z$ is a Feller process.

\section{Continuum seed-bank diffusion: Scaling Limit Interpretation}\label{sec:scaling}

In this section, a discrete-time Wright-Fisher type model is introduced, and then Theorem \ref{th:scaling} which states that the scaling limit of the allele frequency process is the $(D,d)$-valued solution $Z$ to equation (\ref{eq:SEE}) will be proved. 

Consider a haploid population of fixed size $N$ which evolves in discrete generations $0,1,\cdots$. Assume that each individual carries an allele of type from $\{A, a\}$, and
there are finitely many seed-banks labeled by their respective sizes $M^{N}_{i}, i=1,2,\cdots, n$, which consists of the dormant individuals. 

For given $c\in \mathbb{N}$, assume that $c\leq \min\{N, M^{N}_{i}, i=1,2,\cdots, n\}$, and $c=\sum\limits_{i=1}^{n}c^{n}_{i}$ for $c^{n}_{i}>0$. The dynamics of the model are described as follows:

\begin{enumerate}
\item $N-c$ active individuals in generation $1$ is produced by active individuals in generation $0$ by multinomial sampling with equal weights i.e. each of them selects a parent from generation $0$ uniformly, and inherits its genetic type;
\item The remaining $c$ individuals select their parents in the same way, but they go dormant afterwards. Each of them selects a seed-bank to enter according to the probability distribution $\{\frac{c^{n}_{i}}{c}, i=1,2,\cdots, n\}$. Let $C^{n}_{i}$ be the number of individuals entering the $i$-th seed-bank, then $\{C^{n}_{i}, i=1,2,\cdots, n\}$ is a multinomial distributed random vector; 
\item According to the realization of $2$, there are $C^{n}_{i}$ dormant individuals in the $i$-th seed-bank to revive in generation $1$, and they are selected by sampling without replacement;
\item The remaining $\sum\limits_{i=1}^{n}M^{N}_{i}-c$ dormant individuals stay in the seed-banks.
\end{enumerate}

Since the active population and the seed-banks are exchanging the same number of individuals, their sizes $\{N, M_{i}, i=1,2,\cdots, n\}$ remain unchanged.

Following the notations in \cite{blath2016}, the type-$A$ allele frequency process is defined as 
\begin{eqnarray}
X^{N}_{k}=\frac{1}{N}\sum_{p=1}^{N}I_{\{\xi_{k}(p)=A\}} \text{ and } Y^{M^{N}_{i}}_{k}=\frac{1}{M^{N}_{i}}\sum_{q=1}^{M^{N}_{i}}I_{\{\eta^{i}_{k}(q)=A\}},
\end{eqnarray}
for $i=1,2,\cdots,n$ and $k\in\mathbb{N}_{0}$, where $\xi_{k}(p)$ and $\eta_{k}^{i}(q)$ are random variables taking values in $\{A, a\}$. All of them are discrete-time time-homogeneous Markov chains taking values in
\begin{eqnarray}
J^{N}=\{0,\frac{1}{N}, \frac{2}{N},\cdots,1\} \text{ and }
J^{M^{N}_{i}}=\{0,\frac{1}{M^N_{i}}, \frac{2}{M^N_{i}},\cdots,1\},
\end{eqnarray}
respectively.

Denote $Y_{k}^{\vec{M}^{N}}=(Y_{k}^{M^N_{1}}, Y_{k}^{M^N_{2}}, \cdots, Y_{k}^{M^N_{n}})$, then for $X^{N}_{0}=x\in J^{N}$ and $Y^{\vec{M}^{N}}_{0}=\vec{y}\in \prod\limits_{i=1}^{n}J^{M^N_{i}}$, the transition probability of $\{(X^N_k, Y^{\vec{M}^N}_k)\}_{k\in\mathbb{N}_{0}}$ is 
\begin{eqnarray}
&&P_{x,\vec{y}}(X_{1}^{N}=x^{\prime}, Y_{1}^{\vec{M}^N}=\vec{y}^{\prime})=\sum_{i_{1}+\cdots+i_{n}=c}\{P(C^{n}_{1}=i_{1},\cdots, C^n_{n}=i_{n})\cdot\\
&&\quad P_{x,\vec{y}}(U=x^{\prime}N-\sum_{j=1}^{n}Z^n_{j}, V^n_{j}=(y_{j}^{\prime}-y_{j})M^{N}_{j}+Z^n_{j},\text{for}~j=1,2,\cdots,n|C^n_{1}=i_{1},\cdots, C^n_{n}=i_{n})\},\nonumber
\end{eqnarray}
where 
\begin{enumerate}
	\item $U$ is the number of active individuals in generation $1$ that are offsprings of type-A active individuals in generation $0$;
	\item $V^{n}_{j}$, for $j=1,\cdots, n$, is the number of dormant individuals in the $j$-th seed-bank of generation $1$ that are offsprings of type-A active individuals in generation $0$;
	\item $Z^{n}_{j}$, for $j=1,\cdots, n$, is the number of type-A individuals to revive in generation 1 from the $j$-th seed-bank in generation 0.
\end{enumerate}

According to the mechanism of the model, it is known that 
\begin{enumerate}
	\item $U$,$V^n_j$,$Z^n_j$ are conditionally independent with respect to $C^n_j$, for $j=1,\cdots, n$;
	\item $U\sim Bin(N-c, x)$;
	\item $V^n_{j}\sim Bin(C^n_{j}, x)$\footnote{It refers to that $V^n_{j}|_{C^{n}_{j}=i}\sim Bin(i,x)$. The same applies to 4 below.} for $j=1,\cdots, n$;
	\item $Z^n_{j}\sim Hyp(M^{N}_{j}, y_{j}M^N_{j}, C^n_j)$ for $j=1,\cdots, n$;
	\item $X^{N}_{1}=\frac{U+\sum\limits_{j=1}^{n}Z^{n}_{j}}{N}$, and 
	$Y^{M^N_{j}}_{1}=\frac{y_{j}M^{N}_{j}-Z^n_{j}+V^n_{j}}{M^{N}_{j}}$ for $j=1,\cdots, n$.
\end{enumerate}

Now, in order to prove the desired result, the vector valued process $Z^{N}=(X^{N}, Y^{M^{N}})$ should be mapped to a 
$(D,d)$-valued process. Therefore, the following measurable mapping is introduced:

\begin{defn}\label{defn:mapping}
\begin{eqnarray*}
\eta_{N,n}: E_{N,n}&\rightarrow &(D,d)\\
(x,\vec{y})&\mapsto &(x, \tilde{y}.\mu),
\end{eqnarray*}
where $E_{N, n}=J^{N}\times \prod\limits_{j=1}^{n}J^{M^{N}_{j}}$, $\tilde{y}=\sum\limits_{i=1}^{n}y_{i}I_{(\lambda^{n}_{i-1}, \lambda^{n}_{i}]}\in L^{1}(\mu)$, $\lambda^{n}_{i}=\frac{c^{n}_{i}N}{M^{N}_{i}}$ for $i=1,2,\cdots,n$, and $\lambda^{n}_{0}=0$.
\end{defn}

Let $\mu_{n}=\sum\limits_{i=1}^{n}c^{n}_{i}\delta_{\lambda^{n}_{i}}$, if $c^{n}_{i}=\mu(\lambda^{n}_{i-1}, \lambda^{n}_{i}]$, then 
\begin{equation}\label{eq:equal}
\int_{(0,\infty)}\tilde{y}(\lambda)\mu_{n}(d\lambda)=\int_{(0,\infty)}\tilde{y}(\lambda)\mu(d\lambda).	
\end{equation}
At the end of the following proof, the support set $\{\lambda^{n}_{i}\}_{i=1,2,\cdots, n_{r}}$ will be determined at first, then take $c^{n}_{i}=\mu(\lambda^{n}_{i-1}, \lambda^{n}_{i}]$, and for given $N_{r}$, let $M^{N_{r}}_{i}=\frac{c^{n}_{i}N_{r}}{\lambda^{n}_{i}}$. In this way, the corresponding mapping $\eta_{N_{r},n_{r}}$ can be defined, and equation (\ref{eq:equal}) holds.

\tikzset{global scale/.style={
    scale=#1,
    every node/.append style={scale=#1}
  }
}
\begin{center}
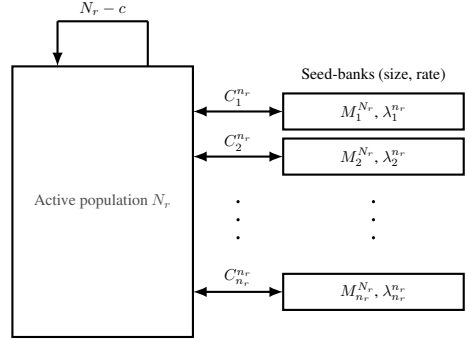

\begin{tikzpicture}[global scale=0.6]
\node (rect) at (0,0) [draw, thick, fill=white, fill opacity=0.7,
minimum width=4cm,minimum height=6cm] {Active population $N_{r}$}; 
\draw[-,thick] (1,3) -- (1,4) node[right] {};
\draw[-,thick] (1,4) -- (-1,4) node[above] at (0,4){$N_{r}-c$};
\draw[-latex,thick] (-1,4) -- (-1,3) node[right] {};
\draw[latex-latex, thick] (2,2) -- (4,2) node[above] at (3,2){$C^{n_{r}}_{1}$};
\draw[latex-latex, thick] (2,1) -- (4,1) node[above] at (3,1){$C^{n_{r}}_{2}$};
\fill (3,0) circle (1pt);
\fill (3,-0.4) circle (1pt);
\fill (3,-0.8) circle (1pt);
\draw[latex-latex, thick] (2,-2) -- (4,-2) node[above] at (3,-2){$C^{n_{r}}_{n_r}$};
\draw[thick] (4,2.4) rectangle (8,1.6) node[] at (6, 2){$M^{N_{r}}_{1}$, $\lambda^{n_{r}}_{1}$};
\draw[thick] (4,1.4) rectangle (8,0.6) node[] at (6, 1){$M^{N_{r}}_{2}$, $\lambda^{n_{r}}_{2}$};
\draw[thick] (4,-1.6) rectangle (8,-2.4) node[] at (6, -2){$M^{N_{r}}_{n_{r}}$, $\lambda^{n_{r}}_{n_r}$};
\fill (6,0) circle (1pt);
\fill (6,-0.4) circle (1pt);
\fill (6,-0.8) circle (1pt);
\node at (6, 2.8){Seed-banks (size, rate)};
\end{tikzpicture}
\captionof{figure}{The $r$-th discrete-time model in Theorem \ref{th:scaling}}	
\end{center}

\begin{proof}[\textbf{Proof of Theorem \ref{th:scaling}}]

We consider the restriction of the generator $\mathcal{L}$ to $H|_{(D,d)}$, and still denote it as  $\mathcal{L}$. Since $(D, v)$ is a compact metric space and $H|_{(D,d)}$ is dense in $C(D,d)$ (see Proposition \ref{prop:stone}), then by Corollary 3.9.3 and Theorem 3.9.4 in \cite{ethier2009markov}, we know that the Markov chains $\{\eta_{N_{r},n_{r}}(X^{N_{r}}, Y^{\vec{M_{r}}})\}_{r\in\mathbb{N}}$ are relatively compact. Furthermore, the $D_{\mathbb{R}_{+}}(D,d)$-martingale problem for $(\mathcal{L},\nu)$ has a unique solution by the same proof as that of Proposition \ref{prop:markovw}, hence by Corollary 4.8.11 in \cite{ethier2009markov} which is based on Corollary 4.8.9 and Theorem 4.8.10 therein, we only need to prove that:
\begin{eqnarray}\label{eq:target}
\lim\limits_{r\rightarrow \infty}\sup_{z\in E_{N_{r},n_{r}}}|\mathcal{L}^{N_{r},n_{r}}(h\circ \eta_{N_{r},n_{r}})(z)-(\mathcal{L}h)\circ \eta_{N_{r},n_{r}}(z)|=0,	
\end{eqnarray}
where 
\begin{equation*}
\mathcal{L}^{N,n}(h\circ \eta_{N,n})(z)=N(E_{z}[(h\circ \eta_{N,n})(X^{N}_{1}, Y^{\vec{M}^{N}}_{1})-(h\circ \eta_{N,n})(z)]),	
\end{equation*}
and $E_{z}$ denotes that the initial condition is $z=(x,\vec{y})$.

We are going to estimate
$$\sup_{z\in E_{N,n}}|\mathcal{L}^{N,n}(h\circ \eta_{N,n})(z)-(\mathcal{L}h)\circ \eta_{N,n}(z)|.$$
By the Taylor expansion and Definition \ref{defn:mapping}, we have
\begin{eqnarray*}\label{eq:es}
&&E_{z}[(h\circ \eta_{N,n})(X^{N}_{1}, Y^{\vec{M}^{N}}_{1})-(h\circ \eta_{N,n})(z)]
\nonumber\\
&=&E_{z}[\underbrace{\sum_{i=1}^{m}\partial_{x_{i}}h_{c}(\cdots)\langle (X^{N}_{1}-x, (\tilde{Y}^{\vec{M}^{N}}_{1}-\tilde{y}).\mu), f_{i}\rangle}_{\text{(1)}}+\\
&&\underbrace{\frac{1}{2}\sum_{i,j=1}^{m}\partial_{x_{i}x_{j}}h_{c}(\cdots)\langle (X^{N}_{1}-x, (\tilde{Y}^{\vec{M}^{N}}_{1}-\tilde{y}).\mu), f_{i}\rangle \langle (X^{N}_{1}-x, (\tilde{Y}^{\vec{M}^{N}}_{1}-\tilde{y}).\mu), f_{j}\rangle}_{\text{(2)}}+R_{N}],\nonumber
\end{eqnarray*}
where ``$\cdots$" refers to ``$\langle z, f_{i}\rangle$ for $i=1,2,\cdots, n$",
and the residual term 
$$R_{N}=\sum_{|\alpha|=3}\frac{3}{\alpha !}\int_{0}^{1}(1-t)^{2}\partial_{\alpha}h_{c}(\langle (x+t(X^{N}_{1}-x), (\tilde{y}+t(\tilde{Y}^{\vec{M}^{N}}_{1}-\tilde{y})).\mu), f_{i}\rangle, \text{ for } i)dt\cdot \langle(X^{N}_{1}-x, (\tilde{Y}^{\vec{M}^{N}}_{1}-\tilde{y}).\mu), \vec{f}~\rangle ^{\alpha}.$$
In the expression of $R_{N}$, the multi-index notations are employed, and $\langle(X^{N}_{1}-x, (\tilde{Y}^{\vec{M}^{N}}_{1}-\tilde{y}).\mu), \vec{f}~\rangle$ is a vector of which the $i$-th component is $\langle(X^{N}_{1}-x, (\tilde{Y}^{\vec{M}^{N}}_{1}-\tilde{y}).\mu), f_{i}\rangle$. 
Throughout this proof, ``for $i$" is the shorthand for ``$i=1,2,\cdots,m$", and 
``for $j$" is the shorthand for ``$j=1,2,\cdots,n$".

In the following estimations, the aforementioned properties of random variables $U, V^{n}_{j}, Z^{n}_{j}$, for $j=1,2,\cdots,n$, and the fact that $U$,$V^n_j$,$Z^n_j$ are conditionally independent with respect to $C^n_j$ will be used repeatedly.

\textbf{Estimation of (1):}

First, for $f_{i}=(f^{(1)}_{i},f^{(2)}_{i})$, we have
$$
E_{z}[\langle (X^{N}_{1}-x, (\tilde{Y}_{1}^{\vec{M}^{N}}-\tilde{y}).\mu), f_{i}\rangle]=E_{z}[(X^{N}_{1}-x)f^{(1)}_{i}+\int_{(0,\infty)}(\tilde{Y}_{1}^{\vec{M}^{N}}-\tilde{y})(\lambda)f^{(2)}_{i}(\lambda)\mu(d\lambda)].
$$

Then, by $X^{N}_{1}=\frac{U+\sum\limits_{j=1}^{n}Z^{n}_{j}}{N}$, and noting that
\begin{eqnarray*}
&&E_{z}[U]=(N-c)x, E_{z}[Z^{n}_{j}|C^{n}_{j},\text{ for }j]=C^{n}_{j}y_{j}, E_{z}[C^{n}_{j}]=c^{n}_{j},\\
&&||\tilde{y}.\mu_{n}||_{TV}=\int_{(0,\infty)}\tilde{y}(\lambda)\mu_{n}(d\lambda)
=\sum\limits_{j=1}^{n}c^{n}_{j}y_{j},
\end{eqnarray*}

we get
\begin{eqnarray*}
E_{z}[(X^{N}_{1}-x)]&=&E_{z}[\frac{U+\sum\limits_{j=1}^{n}Z^{n}_{j}}{N}-x]=E_{z}[E_{z}[\frac{U+\sum\limits_{j=1}^{n}Z^{n}_{j}}{N}-x|C^{n}_{j},\text{ for }j]]\\
&=& E_{z}[\frac{(N-c)x+\sum\limits_{j=1}^{n}C^{n}_{j}y_{j}}{N}-x]= \frac{||\tilde{y}.\mu_{n}||_{TV}-cx}{N}.
\end{eqnarray*}

In addition, by $Y^{M^N_{i}}_{1}=\frac{y_{i}M^{N}_{i}-Z^n_{i}+V^n_{i}}{M^{N}_{i}}$ and $\lambda^{n}_{i}=\frac{c^{n}_{i}N}{M^{N}_{i}}$, for $i=1,\cdots, n$, and noting that

$$E_{z}[\frac{V^{n}_{i}-Z^{n}_{i}}{M^{N}_{i}}]=E_{z}[E_{z}[\frac{V^{n}_{i}-Z^{n}_{i}}{M^{N}_{i}}|C^{n}_{j},\text{ for }j]]=E_{z}[\frac{C^{n}_{i}(x-y_{i})}{M^{N}_{i}}]=\frac{c^{n}_{i}(x-y_{i})}{M^{N}_{i}},$$
we have   
\begin{eqnarray*}
E_{z}[\int_{(0,\infty)}(\tilde{Y}_{1}^{\vec{M}^{N}}-\tilde{y})(\lambda)f^{(2)}_{i}(\lambda)\mu(d\lambda)]&=&\sum_{i=1}^{n}\int_{(\lambda^n_{i-1}, \lambda^{n}_{i}]}E_{z}[(\tilde{Y}_{1}^{\vec{M}^{N}}-\tilde{y})(\lambda)]f^{(2)}_{i}(\lambda)\mu(d\lambda)\\
&=&	\sum_{i=1}^{n}\int_{(\lambda^n_{i-1}, \lambda^{n}_{i}]}E_{z}[\frac{V^n_i-Z^n_i}{M^N_{i}}]f^{(2)}_{i}(\lambda)\mu(d\lambda)\\
&=& \sum_{i=1}^{n}\int_{(\lambda^n_{i-1}, \lambda^{n}_{i}]}\frac{c^{n}_{i}(x-y_{i})}{M^N_{i}}f^{(2)}_{i}(\lambda)\mu(d\lambda)\\
&=& \sum_{i=1}^{n}\int_{(\lambda^n_{i-1}, \lambda^{n}_{i}]}\frac{\lambda^{n}_{i}(x-y_{i})}{N}f^{(2)}_{i}(\lambda)\mu(d\lambda),
\end{eqnarray*}

Therefore, we conclude that 
\begin{eqnarray}\label{eq:es1}
E_{z}[(1)]&=&E_{z}[\sum_{i=1}^{m}\partial_{x_{i}}h_{c}(\cdots)\langle (X^{N}_{1}-x, (\tilde{Y}^{\vec{M}^{N}}_{1}-\tilde{y}).\mu), f_{i}\rangle]\\
&=& \sum_{i=1}^{m}\partial_{x_{i}}h_{c}(\cdots)f^{(1)}_{i}\frac{||\tilde{y}.\mu_{n}||_{TV}-cx}{N}+\sum_{i=1}^{m}\partial_{x_{i}}h_{c}(\cdots)\sum_{j=1}^{n}\int_{(\lambda^n_{i-1}, \lambda^{n}_{i}]}\frac{\lambda^{n}_{j}(x-y_{j})}{N}f^{(2)}_{i}(\lambda)\mu(d\lambda).\nonumber
\end{eqnarray}

\textbf{Estimation of (2):}

First, by definition, 
\begin{eqnarray}\label{eq:three}
&&\frac{1}{2}\sum_{i,j=1}^{m}\partial_{x_{i}x_{j}}h_{c}(\cdots)\langle (X^{N}_{1}-x, (\tilde{Y}^{\vec{M}^{N}}_{1}-\tilde{y}).\mu), f_{i}\rangle \langle (X^{N}_{1}-x, (\tilde{Y}^{\vec{M}^{N}}_{1}-\tilde{y}).\mu), f_{j}\rangle\nonumber\\
&=& \frac{1}{2}\sum_{i,j=1}^{m}\partial_{x_{i}x_{j}}h_{c}(\cdots)(X^{N}_{1}-x)^{2}f^{(1)}_{i}f^{(1)}_{j}\\
&&+\sum_{i,j=1}^{m}\partial_{x_{i}x_{j}}h_{c}(\cdots)(X^{N}_{1}-x)f^{(1)}_{i}\int_{(0,\infty)}(\tilde{Y}^{\vec{M}^{N}}_{1}-\tilde{y})(\lambda)f^{(2)}_j(\lambda)\mu(d\lambda)\nonumber\\
&&+\frac{1}{2}\sum_{i,j=1}^{m}\partial_{x_{i}x_{j}}h_{c}(\cdots)\int_{(0,\infty)}(\tilde{Y}^{\vec{M}^{N}}_{1}-\tilde{y})(\lambda)f^{(2)}_i(\lambda)\mu(d\lambda)\int_{(0,\infty)}(\tilde{Y}^{\vec{M}^{N}}_{1}-\tilde{y})(\lambda)f^{(2)}_j(\lambda)\mu(d\lambda),\nonumber
\end{eqnarray}

For the first term of equation (\ref{eq:three}), by $c=\sum\limits_{i=1}^{n}c^{n}_{i}$, we have
\begin{eqnarray*}
&&E_{z}[(X^{N}_{1}-x)^{2}]\\
&=&\frac{1}{N^{2}}E_{z}[(U-(N-c)x+\sum_{i=1}^{n}(Z^{n}_{i}-c^{n}_{i}x))^{2}]\\
&=& \frac{1}{N^{2}}(\underbrace{E_{z}[(U-(N-c)x)^{2}]+2E_{z}[(U-(N-c)x)\sum_{i=1}^{n}(Z^{n}_{i}-c^{n}_{i}x)]+E_{z}[\sum_{i,j=1}^{n}(Z^{n}_{i}-c^{n}_{i}x)(Z^{n}_{j}-c^{n}_{j}x)])}_{\text{(3)}},
\end{eqnarray*}

Next, we need to further estimate (3):

First we know that 
\begin{eqnarray*}
&&E_{z}[(U-(N-c)x)^{2}]=(N-c)x(1-x),\\
&&E_{z}[(U-(N-c)x)\sum_{i=1}^{n}(Z^{n}_{i}-c^{n}_{i}x)]=0.
\end{eqnarray*}
Since $Z^{n}_{i}\leq C^{n}_{i}$ for $i=1,2,\cdots,n$, and $c=\sum\limits_{i=1}^{n}C^{n}_{i}$, then
\begin{eqnarray*}
E_{z}[|\sum_{i,j=1}^{n}(Z^{n}_{i}-c^{n}_{i}x)(Z_{j}-c^{n}_{j}x)|]&=&E_{z}[E_{z}[|\sum_{i,j=1}^{n}(Z^{n}_{i}-c^{n}_{i}x)(Z^{n}_{j}-c^{n}_{j}x)|C^{n}_{j},\text{ for }j]]\\
&\leq & E_{z}[\sum_{i,j=1}^{n}(C^{n}_{i}+c^{n}_{i})(C^{n}_{j}+c^{n}_{j})]\\
&=& 4c^{2}.
\end{eqnarray*}
Therefore, 
\begin{equation}\label{eq250}
E_{z}[(X^{N}_{1}-x)^{2}]=\frac{1}{N^{2}}E_{z}[(3)]=\frac{1}{N}x(1-x)+O(\frac{1}{N^{2}}).
\end{equation}

For the second term of (\ref{eq:three}), by definition, we have
\begin{eqnarray*}
&&E_{z}[(X^{N}_{1}-x)\int_{(0,\infty)}(\tilde{Y}^{\vec{M}^{N}}_{1}-\tilde{y})(\lambda)f^{(2)}_{i}(\lambda)\mu(d\lambda)]\\
&=&\sum_{j=1}^{n}\int_{(\lambda^n_{j-1}, \lambda^{n}_{j}]}E_{z}[(X^{N}_{1}-x)(\frac{V^{n}_{j}-Z^{n}_{j}}{M^N_{j}})]
f^{(2)}_{i}(\lambda)\mu(d\lambda),
\end{eqnarray*}
in which
\begin{eqnarray*}
&&E_{z}[(X^{N}_{1}-x)(\frac{V^{n}_{j}-Z^{n}_{j}}{M^{N}_{j}})]\\
&=& \frac{1}{NM^{N}_{j}}E_{z}[(U-(N-c)x)(V^{n}_{j}-Z^{n}_{j})]+\frac{1}{NM^{N}_{j}}
E_{z}[\sum_{k=1}^{n}(Z^{n}_{k}-c^{n}_{k}x)(V^{n}_{j}-Z^{n}_{j})]\\
&=& O(\frac{1}{NM^{N}}),
\end{eqnarray*}
where $M^{N}=\min\{M^{N}_{i}, i=1,2,\cdots,n\}$.

Consquently, \begin{equation}\label{eq251}
	E_{z}[(X^{N}_{1}-x)\int_{(0,\infty)}(\tilde{Y}^{\vec{M}^{N}}_{1}-\tilde{y})(\lambda)f^{(2)}_{i}(\lambda)\mu(d\lambda)]=O(\frac{1}{NM^{N}})\cdot \int_{(0,\infty)}f^{(2)}_{i}(\lambda)\mu(d\lambda).
\end{equation}

For the last term of (\ref{eq:three}), by $|V^{n}_{i}-Z^{n}_{i}|\leq 2c$, 
\begin{eqnarray}\label{eq252}
&&E_{z}[\int_{(0,\infty)}(\tilde{Y}^{M^{N}}_{1}-\tilde{y})(\lambda)f^{(2)}_{i}(\lambda)\mu(d\lambda)\int_{(0,\infty)}(\tilde{Y}^{M^{N}}_{1}-\tilde{y})(\lambda)f^{(2)}_{j}(\lambda)\mu(d\lambda)]\nonumber\\
&=&\sum_{k,l=1}^{n}\int_{(\lambda^n_{k-1}, \lambda^{n}_{k}]}\int_{(\lambda^n_{l-1}, \lambda^{n}_{l}]}E_{z}[(\frac{V^{n}_{k}-Z^{n}_{k}}{M^N_{k}})(\frac{V^{n}_{l}-Z^{n}_{l}}{M^N_{l}})]
f^{(2)}_{i}(\lambda_{1})f^{(2)}_{j}(\lambda_{2})\mu(d\lambda_{1})\mu(d\lambda_{2})\\
&=& O(\frac{1}{(M^{N})^{2}})\cdot \int_{(0,\infty)}f^{(2)}_{i}(\lambda)\mu(d\lambda)\int_{(0,\infty)}f^{(2)}_{j}(\lambda)\mu(d\lambda),\nonumber
\end{eqnarray}

Finally, by (\ref{eq250}), (\ref{eq251}) and (\ref{eq252}), we conclude that 

\begin{eqnarray}\label{eq:es2}
	E_{z}[(2)]&=&\frac{1}{2}E_{z}[\sum_{i,j=1}^{m}\partial_{x_{i}x_{j}}h_{c}(\cdots)\langle (X^{N}_{1}-x, (\tilde{Y}^{\vec{M}^{N}}_{1}-\tilde{y}).\mu), f_{i}\rangle \langle (X^{N}_{1}-x, (\tilde{Y}^{\vec{M}^{N}}_{1}-\tilde{y}).\mu), f_{j}\rangle]\nonumber\\
	&=& \frac{1}{2}\sum_{i,j=1}^{m}\partial_{x_{i}x_{j}}h_{c}(\cdots)f^{(1)}_{i}f^{(1)}_{j}\cdot (\frac{x(1-x)}{N}+O(\frac{1}{N^{2}}))\\
	&&+\sum_{i,j=1}^{m}\partial_{x_{i}x_{j}}h_{c}(\cdots)\int_{(0,\infty)}f^{(2)}_{j}(\lambda)\mu(d\lambda) f^{(1)}_{i}\cdot O(\frac{1}{NM^{N}})\nonumber\\
	&&+\frac{1}{2}\sum_{i,j=1}^{m}\partial_{x_{i}x_{j}}h_{c}(\cdots)\int_{(0,\infty)}f^{(2)}_{i}(\lambda)\mu(d\lambda)\int_{(0,\infty)}f^{(2)}_{j}(\lambda)\mu(d\lambda)\cdot O(\frac{1}{(M^{N})^{2}})\nonumber
\end{eqnarray}

\textbf{Estimation of $\mathbf{R_{N}}$:}

By the mean value theorem for integrals, we know that
$$|E_{z}[R_{N}]|=|E_{z}[\sum_{|\alpha|=3}\frac{3}{\alpha !}\partial_{\alpha}h_{c}(\langle (x+t_{0}(X^{N}_{1}-x), (\tilde{y}+t_{0}(\tilde{Y}^{\vec{M}^{N}}_{1}-\tilde{y})).\mu), f_{i}\rangle, \text{ for } i)\langle(X^{N}_{1}-x, (\tilde{Y}^{M^{N}}_{1}-\tilde{y}).\mu), \vec{f}~\rangle ^{\alpha}]|,$$
for some $t_{0}=t_{0}(\omega)\in (0,1)$.

In order to simplify the notation\footnote{Note that any $h\in H$ is also an infinitely differentiable mapping on $E$, and its Fr\'{e}chet derivatives maintain the same forms.}, we consider the third order Fr\'{e}chet derivative $\nabla^{3}h$ for $h\in H$, of which the value is a bounded trilinear functional on $(\mathbb{R}\times \mathcal{M}(0,\infty))^3$. By Definition \ref{defn:frechet}, and following the proof of Lemma \ref{lem:derivative}, we have 
\begin{eqnarray*}
|E_{z}[R_{N}]|&\leq & E_{z}[||\nabla^{3}h((x+t_{0}(X^{N}_{1}-x), (\tilde{y}+t_{0}(\tilde{Y}^{\vec{M}^{N}}_{1}-\tilde{y})).\mu)||_{3}\cdot ||(X^{N}_{1}-x, (\tilde{Y}_{1}^{\vec{M}^{N}}-\tilde{y}).\mu)||^{3}]\\
&\leq & \sup_{z\in (D,d)}||\nabla^{3}h(z)||_{3}\cdot E_{z}[||(X^{N}_{1}-x, (\tilde{Y}_{1}^{\vec{M}^{N}}-\tilde{y}).\mu)||^{3}]\\
&=& \sup_{z\in (D,d)}||\nabla^{3}h(z)||_{3}\cdot E_{z}[(|X^{N}_{1}-x|+\int_{(0,\infty)} |\tilde{Y}_{1}^{\vec{M}^{N}}-\tilde{y}(\lambda)|\mu(d\lambda))^{3}]
\end{eqnarray*}
where $||\cdot||_{3}$ denotes the norm of a bounded trilinear functional.

Then as before,
\begin{eqnarray*}
	&&E_{z}[(|X^{N}_{1}-x|+\int_{(0,\infty)} |(\tilde{Y}_{1}^{\vec{M}^{N}}-\tilde{y})(\lambda)|\mu(d\lambda))^{3}]\\
	&\leq & E_{z}[(|\frac{U-(N-c)x}{N}|+|\frac{\sum\limits_{j=1}^{n}(Z^{n}_{j}-c^{n}_{j}x)}{N}|+\sum_{j=1}^{n}\int_{(\lambda^n_{j-1}, \lambda^{n}_{j}]}|\frac{V^{n}_{j}-Z^{n}_{j}}{M^{N}_{j}}|\mu(d\lambda))^{3}]
\end{eqnarray*}

Apply the Marcinkiewicz-Zygmund inequality (\cite{marcinkiewicz1938quelques}) to the binomial distributed random variable $U$, we get
$$E_{z}[|U-(N-c)x|^{3}]=O(N^{\frac{3}{2}}),$$

and then by the boundedness of $V^{n}_{j}$ and $Z^{n}_{j}$ for $j=1,2,\cdots,n$,
it can be shown that 
\begin{equation}\label{eq:es3}
	E_{z}[R_{N}]= \sup_{z\in (D,d)}||\nabla^{3}h(z)||_{3}\cdot(O(\frac{1}{N^{\frac{3}{2}}})+O(\frac{1}{(M^{N})^{2}})+O(\frac{1}{NM_{N}})).
\end{equation}

Now, combining equation (\ref{eq:es}),(\ref{eq:es1}),(\ref{eq:es2}) and (\ref{eq:es3}) together, and by Lemma \ref{lem:derivative}, we have the following estimation:

\begin{eqnarray*}
&&\mathcal{L}^{N,n}(h\circ \eta_{N,n})(z)\\
&=&N(E_{z}[(h\circ \eta_{N,n})(X^{N}_{1}, Y^{\vec{M}^{N}}_{1})-(h\circ \eta_{N,n})(z)])\\
&=& \sum_{i=1}^{m}\partial_{x_{i}}h_{c}(\cdots)f^{(1)}_{i}(||\tilde{y}.\mu_{n}||_{TV}-cx)+\sum_{i=1}^{m}\partial_{x_{i}}h_{c}(\cdots)\sum_{j=1}^{n}\int_{(\lambda^n_{j-1}, \lambda^{n}_{j}]}\lambda^{n}_{j}(x-y_{j})f^{(2)}_{i}(\lambda)\mu(d\lambda)\\
&&+\frac{1}{2}\sum_{i,j=1}^{m}\partial_{x_{i}x_{j}}h_{c}(\cdots)f^{(1)}_{i}f^{(1)}_{j}x(1-x)+\sup_{z\in (D,d)}||\nabla^{2}h(z)||_{2}\cdot (O(\frac{1}{N})+O(\frac{1}{M^{N}})+O(\frac{N}{(M^{N})^{2}}))\\
&&+\sup_{z\in (D,d)}||\nabla^{3}h(z)||_{3}\cdot(O(\frac{1}{N^{\frac{1}{2}}})+O(\frac{N}{(M^{N})^{2}})+O(\frac{1}{M_{N}})),
\end{eqnarray*}
where $||\cdot||_{2}$ denotes the norm of a bounded bilinear functional.

By (\ref{eq:generatorw}), we have
\begin{eqnarray*}
(\mathcal{L}h)\circ \eta_{N,n}(z)&=&\sum_{i=1}^{m}\partial_{x_{i}}h_{c}(\cdots)\{(||\tilde{y}.\mu||_{TV}-cx)f^{(1)}_{i}+\sum_{j=1}^{n}\int_{(\lambda^n_{j-1}, \lambda^{n}_{j}]}\lambda(x-y_{j})f^{(2)}_{i}(\lambda)\mu(d\lambda)\\	
&&+\int_{(\lambda^{n}_{n},\infty)}\lambda(x-y_{j})f^{(2)}_{i}(\lambda)\mu(d\lambda)\}+\frac{1}{2}\sum_{i,j=1}^{m}\partial_{x_{i}x_{j}}h_{c}(\cdots)x(1-x)f^{(1)}_{i}f^{(1)}_{j}.
\end{eqnarray*}

Moreover, recall that if $c^{n}_{i}=\mu(\lambda^{n}_{i-1}, \lambda^{n}_{i}]$, then (\ref{eq:equal}) holds i.e. $||\tilde{y}.\mu_{n}||_{TV}=||\tilde{y}.\mu||_{TV}$.

Consequently, we get
\begin{eqnarray*}
&&\sup_{z\in E_{N,n}}|\mathcal{L}^{N,n}(h\circ \eta_{N,n})(z)-(\mathcal{L}h)\circ \eta_{N,n}(z)|\\
&=& \sup_{z\in E_{N,n}}|\sum_{i=1}^{m}\partial_{x_{i}}h_{c}(\cdots)(\sum_{j=1}^{n}\int_{(\lambda^n_{j-1}, \lambda^{n}_{j}]}(\lambda^{n}_{j}-\lambda)(x-y_{j})f^{(2)}_{i}(\lambda)\mu(d\lambda)+\int_{(\lambda^{n}_{n},\infty)}\lambda(x-y_{j})f^{(2)}_{i}(\lambda)\mu(d\lambda))|\\
&&+(\sup_{z\in (D,d)}||\nabla^{2}h(z)||_{2}+\sup_{z\in (D,d)}||\nabla^{3}h(z)||_{3})\cdot (O(\frac{1}{N^{\frac{1}{2}}})+O(\frac{1}{M^{N}})+O(\frac{N}{(M^{N})^{2}}))\\
&\leq & |||h|||_{C^{3}}\cdot(\sum_{j=1}^{n}\int_{(\lambda^n_{j-1}, \lambda^{n}_{j}]}(\lambda^{n}_{j}-\lambda)\mu(d\lambda)+\int_{(\lambda^{n}_{n},\infty)}\lambda\mu(d\lambda))+O(\frac{1}{N^{\frac{1}{2}}})+O(\frac{1}{M^{N}})+O(\frac{N}{(M^{N})^{2}}))\\
&\leq & |||h|||_{C^{3}}\cdot(\underbrace{c\max_{1\leq i\leq n}|\lambda^{n}_{i}-\lambda^{n}_{i-1}|+\int_{(\lambda^{n}_{n},\infty)}\lambda\mu(d\lambda))}_\text{(4)}+O(\frac{1}{N^{\frac{1}{2}}})+O(\frac{1}{M^{N}})+O(\frac{N}{(M^{N})^{2}})),
\end{eqnarray*}
where $|||h|||_{C^{3}}=\sum\limits_{i=0}^{3}||\nabla^{i}h||_{i}$.

Take $\{\lambda^{n}_{i}\}_{i=1,2,\cdots,n}$ such that $\lim\limits_{n\rightarrow \infty}\lambda^{n}_{n}=\infty$ and $\lim\limits_{n\rightarrow\infty}\max\limits_{1\leq i\leq n}|\lambda^{n}_{i}-\lambda^{n}_{i-1}|=0$, they by the Condition (\ref{condition}), we know that $\lim\limits_{n\rightarrow \infty}(4)=0$. For each $r\in\mathbb{N}$, first take some $n_{r}$ and $\{\lambda^{n_{r}}_{i}\}_{i=1,2,\cdots,n_{r}}$ such that $|(4)|\leq \frac{1}{2^{r+1}}$. Then, take $c^{n_{r}}_{i}=\mu(\lambda^{n_{r}}_{i-1}, \lambda^{n_{r}}_{i}]$ as required by the above estimations, and finally let $M^{N}_{i}=\frac{c^{n_{r}}_{i}N}{\lambda^{n_{r}}_{i}}$ for $i=1,2,\cdots,n_{r}$, which is also indispensable.
Now, we have
$$\sup_{z\in E_{N,n_{r}}}|\mathcal{L}^{N,n_{r}}(h\circ \eta_{N,n_{r}})(z)-(\mathcal{L}h)\circ \eta_{N,n_{r}}(z)|\leq |||h|||_{C^{3}}\cdot (\frac{1}{2^{r+1}}+O(\frac{1}{N^{\frac{1}{2}}})+O(\frac{1}{N})),$$
which implies that we can take some $N_{r}$, and accordingly let $M^{N_{r}}_{i}=\frac{c^{n_{r}}_{i}N_{r}}{\lambda^{n_{r}}_{i}}$ to get the following desired result:
\begin{eqnarray}\label{eq:estimation}
\sup_{z\in E_{N_{r},n_{r}}}|\mathcal{L}^{N_{r},n_{r}}(h\circ \eta_{N_{r},n_{r}})(z)-(\mathcal{L}h)\circ \eta_{N_{r},n_{r}}(z)|\leq \frac{1}{2^{r}}|||h|||_{C^{3}},
\end{eqnarray}
and at the same we need $c\leq \min\{N_{r}, M^{N_{r}}_{i}, i=1,2,\cdots, n_{r}\}$, which is required by the discrete-time model.

Letting $r\rightarrow\infty$ in the estimation (\ref{eq:estimation}), we get (\ref{eq:target}), and the proof is completed.
\end{proof}

\section{Continuum seed-bank coalescent}\label{sec:coalescent}

The coalescent process refers to a stochastic model used in population genetics to study the genealogical history of a sample of individuals from a population, which is particularly useful for understanding the effects of various evolutionary forces and their interactions during lineage formation.
In this section, the coalescent process of the continuum seed-bank diffusion i.e. the solution $Z$ to equation ($\ref{eq:SEE}$) will be constructed. 

The primary task is to identify the block counting process of the coalescent, which should be the dual process of the Markov process $Z$. Unfortunately, neither $\mathbb{R}\times L^{1}(\mu)$ nor $\mathbb{R}\times \mathcal{M}(0,\infty)$ is a suitable space for establishing a duality relation which is supposed to degenerate into the case when the measure $\mu$ is discrete i.e. there are at most countably many seed-banks.

However, by Remark \ref{rem:mu}, when the initial value $\eta(\lambda)$ in equation (\ref{eq:SEE}) is defined everywhere without dependence on a measure $\mu$, the solution $Z$ takes valued in $\mathbb{R}\times L^{\infty}$\footnote{Note that $Z$ does not need to be a Markov process in this new state space.}. It is well-known that $(\mathbb{R}\times L^{\infty})^{\star}\cong(\mathbb{R}\times ba((0,\infty),\mathcal{B}(0,\infty))$, where $ba((0,\infty),\mathcal{B}(0,\infty))$ includes all finite measures. As the result, the following dual function can have a measure valued first order Fr\'{e}chet derivative with respect to $y$ by adding a perturbation:
\begin{equation}\label{eq:dualfunction}
F[(x,y(\lambda)),(n,m(d\lambda))]=x^{n}e^{\int_{(0,\infty)} lny(\lambda)m(d\lambda)},
\end{equation}
where $(x,y(\lambda))\in \mathbb{R}\times L^{\infty}$, $y$ is non-negative, $n\in\mathbb{N}_{0}$, and $m$ is a finite measure on $((0,\infty),\mathcal{B}(0,\infty))$. Since $ln y(\lambda)$ is upper bounded, $F$ is well defined even if $y(\lambda)$ may be $0$. Specifically, if $m$ is integer valued, then $F$ is of the same form as (\ref{dualfunction}) which is employed by \cite{blath2016} and \cite{greven2022spatial}. Actually, $m$ can be represented as a weighted sum of Dirac measures (see e.g. Theorem 2.1.6.2 in \cite{kadets2018course}). Therefore, in the subsequent text, $m$ is of the form 
\begin{equation}\label{eq:form}
m(d\lambda)=\sum_{i=1}^{K}m_{i}\delta_{\lambda_{i}}(d\lambda),	
\end{equation}
for $K\in \mathbb{N}_{0}$, $m_{i}\in\mathbb{N}$ and $\lambda_{i}\in (0,\infty)$.
Clearly, $\int_{(0,\infty)}\lambda m(d\lambda)<\infty$.

The classical method for deriving duality relations is to find certain martingale statements. For fixed $(n, m)$, It\^{o}'s formula should be applied to $F((X_{t}, Y_{t}(\lambda)),(n, m(d\lambda)))$ at first, where $\{(X_{t}, Y_{t}(\lambda))\}_{t\geq0}$ is the 
$$[0,1]\times \{y\in L^{\infty}: 0\leq y(\lambda)\leq 1\}$$
valued solution to equation (\ref{eq:SEE}). The differentiation with respect to $y$ seems not feasible as $y(\lambda)$ may be $0$, but this can resolved by adding suitable perturbations.

For the sequence of functions $\{F[(x,y(\lambda)+\epsilon_{k}),(n,m(d\lambda))]\}_{k\in\mathbb{N}}$, where $\epsilon_{k}>0$ is a constant function and $\epsilon_{k}\downarrow 0$ as $k\rightarrow \infty$, the following result can be proved by direct calculations.

\begin{lemma}\label{lem:frechet}
	The function $\{F[(x,y(\lambda)+\epsilon_{k}),(n,m(d\lambda))]\}_{k\in\mathbb{N}}\in C^{2,1}(\mathbb{R}\times L^{\infty})$,
	and \begin{equation}
	D_{y}F[(x, y(\lambda)+\epsilon_{k}), (n, m(d\lambda))]=x^{n}e^{\int ln(y(\lambda^{\prime})+\epsilon_{k}) (m-\delta_{\lambda})(d\lambda^{\prime})}. m(d\lambda)\in ba((0,\infty),\mathcal{B}(0,\infty)).	
	\end{equation}\end{lemma}
\begin{proof}
The differentiability with respect to $x$ and the continuity of $D_{y}F[(x, y(\lambda)+\epsilon_{k}), (n, m(d\lambda))]$ is obvious. By definition, for any increment $h\in L^{\infty}$,
\begin{eqnarray*}
&&F[(x, y(\lambda)+\epsilon_{k}+h), (n, m(d\lambda))]-F[(x, y(\cdot)+\epsilon_{k}), (n, m(d\lambda))]\\
&=& x^{n}[e^{\int ln(y(\lambda)+\epsilon_{k}+h(\lambda))m(d\lambda)}-e^{\int ln(y(\lambda)+\epsilon_{k})m(d\lambda)}]\\
&=& x^{n}e^{\int ln(y(\lambda)+\epsilon_{k})m(d\lambda)}[e^{\int (\frac{h(\lambda)}{y(\lambda)+\epsilon+{k}}-\frac{h^{2}(\lambda)}{(y(\lambda)+\epsilon_{k}+\theta h(\lambda))^{2}})m(d\lambda)}-1]\\
&=& x^{n}e^{\int ln(y(\lambda)+\epsilon_{k})m(d\lambda)}\cdot [\int (\frac{h(\lambda)}{y(\lambda)+\epsilon_{k}}-\frac{h^{2}(\lambda)}{(y(\lambda)+\epsilon_{k}+\theta h(\lambda))^{2}})m(d\lambda)\\
&&+\frac{1}{2}\tilde{\theta}(\int (\frac{h(\lambda)}{y(\lambda)+\epsilon_{k}}-\frac{h^{2}(\lambda)}{(y(\lambda)+\epsilon_{k}+\theta h(\lambda))^{2}})m(d\lambda))^{2}]\\
&=& x^{n}e^{\int ln(y(\lambda)+\epsilon_{k})m(d\lambda)}\int \frac{h(\lambda)}{y(\lambda)+\epsilon_{k}}m(d\lambda)+R_{1}(\epsilon_{k})+R_{2}(\epsilon_{k}),
\end{eqnarray*}

where $\theta, \tilde{\theta}\in (0,1)$, and 
$$R_{1}(\epsilon_{k})=-x^{n}e^{\int ln(y(\lambda)+\epsilon_{k})m(d\lambda)}\int \frac{h^{2}(\lambda)}{(y(\lambda)+\epsilon_{k}+\theta h(\lambda))^{2}}m(d\lambda),$$
$$R_{2}(\epsilon_{k})=\frac{1}{2}\tilde{\theta}x^{n}e^{\int ln(y(\lambda)+\epsilon_{k})m(d\lambda)}[\int (\frac{h(\lambda)}{y(\lambda)+\epsilon_{k}}-\frac{h^{2}(\lambda)}{(y(\lambda)+\epsilon_{k}+\theta h(\lambda))^{2}})m(d\lambda)]^{2}.$$

As $||h||_{\infty}\rightarrow 0$, we have $\epsilon_{k}+\theta h(\lambda)\geq\frac{\epsilon_{k}}{2}$, thus $\frac{h^{2}(\lambda)}{(y(\lambda)+\epsilon_{k}+\theta h(\lambda))^{2}}\leq \frac{4}{\epsilon_{k}^{2}}||h||^{2}_{\infty}$. Since $m(0,\infty)<\infty$, we have $R_{1}(\epsilon_{k})+R_{2}(\epsilon_{k})=O(||h||_{\infty}^{2})$

Therefore,
\begin{eqnarray*}
&&F[(x, y(\cdot)+\epsilon_{k}+h), (n, m(d\lambda))]-F[(x, y(\cdot)+\epsilon_{k}), (n, m(d\lambda))]\\
&=& \langle h, x^{n}e^{\int ln(y(\lambda^{\prime})+\epsilon_{k}) (m-\delta_{\lambda})(d\lambda^{\prime})}. m(d\lambda)\rangle+o(||h||_{\infty}),
\end{eqnarray*} 
which completes the proof.
\end{proof}

\begin{rem}\label{rem:ob}
There are two important observations regarding Lemma \ref{lem:frechet}.

\begin{enumerate}
	\item By the monotone convergence theorem, 
	\begin{equation}
	\lim_{k\rightarrow\infty}F[(x,y(\lambda)+\epsilon_{k}),(n,m(d\lambda))]=F[(x,y(\lambda)),(n,m(d\lambda))];	
	\end{equation}	
	\item In the expression of $D_{y}F[(x,y(\lambda)+\epsilon_{k}),(n,m(d\lambda))]$, $m-\delta_{\lambda}$ is still a measure since ``$.m(d\lambda)$'' allows ``$-\delta_{\lambda}$" to occur only on the support of $m$.
\end{enumerate}
\end{rem}

Now, apply It\^{o}'s formula to $F[(X_{t},Y_{t}(\lambda)+\epsilon_{k}),(n,m(d\lambda))]$, and then take the limit, the following Proposition is the first martingale statement required for the duality relation:

\begin{prop}\label{prop:firststate}

\begin{eqnarray*}
F[(X_{t}, Y_{t}(\lambda)), (n, m(d\lambda))]-F[(X_{0}, Y_{0}(\lambda)), (n, m(d\lambda))]-\int_{0}^{t}\mathcal{G}F[(X_{s}, Y_{s}(\lambda)),(n,m(d\lambda))]ds	
\end{eqnarray*}
is a continuous martingale, where
\begin{eqnarray}
&&\mathcal{G}F[(X_{t}, Y_{t}(\lambda)),(n,m(d\lambda))] \nonumber \\
&=&n\int_{(0,\infty)}(F[(X_{t}, Y_{t}(\lambda)),(n-1,m+\delta_{\lambda})]-F[(X_{t}, Y_{t}(\lambda)),(n,m(d\lambda))] )\mu(d\lambda) \nonumber \\
&&+\int_{(0,\infty)}(F[(X_{t}, Y_{t}(\lambda)), (n+1, m-\delta_{\lambda})]-F[(X_{t}, Y_{t}(\lambda)),(n,m(d\lambda))])\lambda m(d\lambda)\\
&&+ C^{2}_{n}(F[(X_{t}, Y_{t}(\lambda)), (n-1, m(d\lambda))]-F[(X_{t}, Y_{t}(\lambda)),(n,m(d\lambda))]).\nonumber
\end{eqnarray}
\end{prop}

\begin{proof}
By It\^{o}'s formula and equation (\ref{eq:SEE}), we have
\begin{eqnarray*}
&&F[(X_{t}, Y_{t}(\lambda)+\epsilon_{k}), (n, m(d\lambda))]\\
&=&F[(X_{0}, Y_{0}(\lambda)+\epsilon_{k}), (n, m(d\lambda))]+\underbrace{\int_{0}^{t}nX_{s}^{n-1}e^{
\int_{(0,\infty)} ln(Y_{s}(\lambda)+\epsilon_{k}) m(d\lambda)}dX_{s}}_\text{(1)}\\
&&+\underbrace{\int_{0}^{t}\langle dY_{s}(\lambda), X_{s}^{n}e^{\int_{(0,\infty)} ln(Y_{s}(\lambda^{\prime})+\epsilon_{k}) (m-\delta_{\lambda})(d\lambda^{\prime})}. m(d\lambda)\rangle}_\text{(2)}\\
&&+\underbrace{\frac{1}{2}n(n-1)\int_{0}^{t}X_{s}^{n-2}e^{
\int_{(0,\infty)} ln(Y_{s}(\lambda)+\epsilon_{k}) m(d\lambda)}X_{s}(1-X_{s})ds}_\text{(3)}+M^{(k)}_{t},
\end{eqnarray*}
where $M^{(k)}$ is a continuous martingale due to the boundedness of the integrand.

Then, replacing $c$ with $\int_{(0,\infty)}\mu(d\lambda)$, 
\begin{eqnarray*}
	(1)&=&n\int_{0}^{t}\int_{(0,\infty)}X_{s}^{n-1}e^{
\int_{(0,\infty)} ln(Y_{s}(\lambda^{\prime})+\epsilon_{k}) m(d\lambda^{\prime})}Y_{s}(\lambda)\mu(d\lambda)ds\\
&&-n\int_{0}^{t}\int_{(0,\infty)}X_{s}^{n-1}e^{
\int_{(0,\infty)} ln(Y_{s}(\lambda^{\prime})+\epsilon_{k}) m(d\lambda^{\prime})}\mu(d\lambda)X_{s}ds\\
&=& n\int_{0}^{t}\int_{(0,\infty)}X_{s}^{n-1}e^{
\int_{(0,\infty)} ln(Y_{s}(\lambda^{\prime})+\epsilon_{k}) m(d\lambda^{\prime})} Y_{s}(\lambda)-F[(X_{s}, Y_{s}(\lambda)+\epsilon_{k}),(n,m(d\lambda))] \mu(d\lambda)ds
\end{eqnarray*}

and by the pairing between $L^{\infty}$ and $ba((0,\infty),\mathcal{B}(0,\infty))$ which is defined as $\langle y(\lambda), m(d\lambda)\rangle=\int_{(0,\infty)}y(\lambda)m(d\lambda)$,

\begin{eqnarray*}
	(2)&=&\int_{0}^{t}\langle X_{s}, \lambda X_{s}^{n}e^{\int_{(0,\infty)} ln(Y_{s}(\lambda^{\prime})+\epsilon_{k}) (m-\delta_{\lambda})(d\lambda^{\prime})}. m(d\lambda)\rangle ds\\
	&&-\int_{0}^{t}\langle Y_{s}(\lambda), \lambda X_{s}^{n}e^{\int_{(0,\infty)} ln(Y_{s}(\lambda^{\prime})+\epsilon_{k}) (m-\delta_{\lambda})(d\lambda^{\prime})}. m(d\lambda)\rangle ds\\
	&=&\int_{0}^{t}\int_{(0,\infty)}(F[(X_{s}, Y_{s}(\lambda)+\epsilon_{k}), (n+1, m-\delta_{\lambda})]-X_{s}^{n}e^{\int_{(0,\infty)} ln(Y_{s}(\lambda^{\prime})+\epsilon_{k})(m-\delta_{\lambda})(d\lambda^{\prime})}Y_{s}(\lambda))\lambda m(d\lambda)ds,
\end{eqnarray*}

In addition, it is obvious that
\begin{eqnarray*}
	(3)=C^{2}_{n}(F[(X_{t}, Y_{t}(\lambda)+\epsilon_{k}), (n-1, m(d\lambda))]-F[(X_{t}, Y_{t}(\lambda)+\epsilon_{k}),(n,m(d\lambda))]).
\end{eqnarray*}

Combining (1), (2), (3) together, we get
\begin{eqnarray*}
M^{(k)}_{t}&=&F[(X_{t}, Y_{t}(\lambda)+\epsilon_{k}), (n, m(d\lambda))]-F[(X_{0}, Y_{0}(\lambda)+\epsilon_{k}), (n, m(d\lambda))]\\
&&-\int_{0}^{t}\mathcal{G}_{k}F[(X_{s}, Y_{s}(\lambda)+\epsilon_{k}),(n,m(d\lambda))]ds,
\end{eqnarray*}
where 
\begin{eqnarray*}
&&\mathcal{G}_{k}F[(X_{t}, Y_{t}(\lambda)+\epsilon_{k}),(n,m(d\lambda))]\\
&=&n\int_{(0,\infty)}X_{t}^{n-1}e^{
\int ln(Y_{t}(\lambda^{\prime})+\epsilon_{k}) m(d\lambda^{\prime})} Y_{t}(\lambda)-F[(X_{t}, Y_{t}(\cdot)+\epsilon_{k}),(n,m(d\lambda))] \mu(d\lambda)\\
&&+\int_{(0,\infty)}(F[(X_{t}, Y_{t}(\lambda)+\epsilon_{k}), (n+1, m-\delta_{\lambda})]-X_{t}^{n}e^{\int_{(0,\infty)} ln(Y_{t}(\lambda^{\prime})+\epsilon_{k})(m-\delta_{\lambda})(d\lambda^{\prime})}Y_{t}(\lambda)) \lambda m(d\lambda)\\
&&+ C^{2}_{n}(F[(X_{t}, Y_{t}(\lambda)+\epsilon_{k}), (n-1, m(d\lambda))]-F[(X_{t}, Y_{t}(\lambda )+\epsilon_{k}),(n,m(d\lambda))]).
\end{eqnarray*}

Finally, note that the martingales $\{M^{(k)}\}_{k\in\mathbb{N}}$ are uniformly bounded by a constant. Letting $k\rightarrow \infty$, by the monotone convergence theorem and the dominated convergence theorem, it can be verified that the pointwise limit 
\begin{equation*}
M_{t}=F[(X_{t}, Y_{t}(\lambda)), (n, m(d\lambda))]-F[(X_{0}, Y_{0}(\lambda)), (n, m(d\lambda))]-\int_{0}^{t}\mathcal{G}F[(X_{s}, Y_{s}(\lambda)),(n,m(d\lambda))]ds
\end{equation*}
is a continuous martingale, where the expression of $\mathcal{G}$ comes from the fact that $Y_{s}=e^{\int_{(0,\infty)}Y_{s}(\lambda^{\prime})\delta_{\lambda}(d\lambda^{\prime})}$. The proof is completed.
\end{proof}

By observing Proposition \ref{prop:firststate}, the dual process should be a Markov jump process $\{(N_{t}, M_{t})\}_{t\geq 0}$ with the following bounded generator
\begin{eqnarray}
	&&\mathcal{H}f(n, m(d\lambda))\nonumber\\
	&=& n\int_{(0,\infty)}(f(n-1, m+\delta_{\lambda})-f(n, m(d\lambda)))\mu(d\lambda)+\int_{(0,\infty)}(f(n+1, m-\delta_{\lambda})-f(n, m(d\lambda)))\lambda m(d\lambda)\nonumber\\
	&&+ C^{2}_{n}(f(n-1, m(d\lambda))-f(n, m(d\lambda))),
\end{eqnarray}
for any bounded measurable function $f$ on the state space.

Since $N_{t}\in \mathbb{N}$ and $M_{t}=\sum\limits_{i=1}^{K_{t}}m_{i,t}\delta_{\lambda_{i,t}}$ for some $K_{t}\in\mathbb{N}_{0}$, $m_{i,t}\in\mathbb{N}$ and $\lambda_{i,t}\in(0,\infty)$, the state space for this Markov jump process will be denoted by $\mathbb{N}_{0}\times \bigoplus\limits_{(0,\infty)}\mathbb{N}_{0}$, where the direct sum $\bigoplus\limits_{(0,\infty)}\mathbb{N}_{0}$ is the subspace of $\mathbb{N}_{0}^{(0,\infty)}$ for whose elements all but finitely many components are $0$. 

The $\mathbb{N}_{0}\times \bigoplus\limits_{(0,\infty)}\mathbb{N}_{0}$-valued Markov jump process $\{(N_{t}, M_{t})\}_{t\geq 0}$ with initial distribution $\nu$ can be constructed as follows:

Let $\{\tilde{M}_{k},k\in \mathbb{N}_{0}\}$ be a discrete-time time-homogeneous Markov chain in $\mathbb{N}_{0}\times \bigoplus\limits_{(0,\infty)}\mathbb{N}_{0}$ with initial distribution $\nu$, and its transition kernel is 
\begin{equation}
\mu((n, m), \Gamma)=\left\{\begin{array}{l}
\frac{n \mu(B)}{c_{n,m}}, \Gamma=\left\{\left(n-1, m+\delta_\lambda\right), \lambda \in B\right\}, \\
\frac{\lambda m(\{\lambda \})}{c_{n,m}}, \Gamma=\left(n+1, m-\delta_\lambda\right), \\
\frac{C_n^2}{c_{n,m}}, \Gamma=(n-1, m),
\end{array}\right.
\end{equation}
where $c_{n,m}=cn+\int_{(0,\infty)} \lambda m(d\lambda)+C^2_n$. Then, $\{(N_{t}, M_{t})\}_{t\geq 0}$ can be represented as 
\begin{eqnarray}\label{eq:constructed}
(N_{t}, M_{t})=\sum_{k=0}^{\infty}\tilde{M}_{k}I_{\{[\sum\limits_{j=0}^{k-1}\frac{\Delta_{j}}{c_{Y_{j}}}, \sum\limits_{j=0}^{k}\frac{\Delta_{j}}{c_{Y_{j}}})\}},	
\end{eqnarray}
where $c_{Y_{j}}=c_{n,m}$ if $Y_{j}=(n,m)$, $\{\Delta_{j}\}_{j\in\mathbb{N}_{0}}$ are mutually independent and independent of $\{Y_{j}\}_{j\in\mathbb{N}_{0}}$, and they are exponentially distributed with parameter $1$. 

$\{(N_{t}, M_{t})\}_{t\geq 0}$ can also be described by its transition rates, see (\ref{eq:char1}). By the Markov property and the measurability of $F$ for fixed $(x, y(\lambda))$, 
it is well-known that 
\begin{equation}\label{secondstate}
F[(x, y(\lambda)), (N_{t}, M_{t}(d\lambda))]-F[(x, y(\lambda)), (N_{0}, M_{0}(d\lambda))]-\int_{0}^{t}\mathcal{H}F[(x, y(\cdot)),(N_{s},M_{s}(d\lambda))]ds
\end{equation}
is a martingale. Consequently, Theorem \ref{th:dual} can be proved immediately.

\begin{proof}[\textbf{Proof of Theorem \ref{th:dual}}]
Put $(X,Y)$ and $(N,M)$ into the product filered probability space so that they are independent. By Proposition \ref{prop:firststate}, the martingale 	statement (\ref{secondstate}), and Theorem 4.4.11 in \cite{ethier2009markov}, the duality relation follows. Since all of the processes are bounded, the conditions of Theorem 4.4.11 are satisfied.
\end{proof}

\begin{rem}
	In Theorem \ref{th:dual}, since $Y_{0}(\lambda)$ is a random variable for all $\lambda\in (0,\infty)$, then $Y_{t}(\lambda)$ is $\mathcal{F}_{t}$-measurable for all $t>0$ and $\lambda\in (0,\infty)$. By the definition of the integral, $\int_{(0,\infty)}Y_{t}(\lambda)m(d\lambda)$ is also $\mathcal{F}_{t}$-measurable for all $m\in ba((0,\infty), \mathcal{B}(0,\infty))$. Therefore, if we adopt the $\sigma$-algebra generated by $ba((0,\infty), \mathcal{B}(0,\infty))$ on $L^{\infty}$, then $Y$ is an $L^{\infty}$-valued adapted process. Note that $Y$ is not uniform continuous in $\lambda$, but we can still apply It\^{o}'s formula in Proposition \ref{prop:firststate} by the pointwise continuity and the particular finite-sum form (\ref{eq:form}) of $m$. 
	\end{rem}

The next step is to construct the coalescent process based on its block counting process $\{(N_{t}, M_{t})\}_{t\geq 0}$. Since $M_{t}$ is a finite measure, only samples of finite size $K$ can be considered at first.

For $K\in\mathbb{N}$, let $\mathcal{P}_{K}$ be the set of partitions of $\{1,\cdots,K\}$. Define the space of marked partitions as 
\begin{equation}
\mathcal{P}_{K}^{f}=\{(P_{K}, f)|P_{K}\in \mathcal{P}_{K}, f\in [0,\infty)^{|P_{K}|}\},	
\end{equation}
where $0$ represents ``active'', $f$ represents ``flag'', and $|\cdot|$ denote s the number of blocks. For example, when $k=6$, an element $\pi=(P_{6},f)$ may be $\{\{1\}^{0}\{2\}^{0.1}\{5\}^{1}\{3,4\}^{10}\{6\}^{0}\}$.
For two marked partitions $\pi, \pi^{\prime}$, $\pi\sqsupset \pi^{\prime}$ denotes that $\pi^{\prime}$ is obtained by merging two $0$-blocks in $\pi$ e.g. $$\{\{1\}^{0}\{2\}^{0.1}\{5\}^{1}\{3,4\}^{10}\{6\}^{0}\}\sqsupset \{\{1,6\}^{0}\{2\}^{0.1}\{5\}^{1}\{3,4\}^{10}\},$$
and $\pi\leadsto\pi^{\prime}$ denotes that $\pi^{\prime}$ is obtained by changing the flag of one block of $\pi$, e.g.,
$$\{\{1\}^{0}\{2\}^{0.1}\{5\}^{1}\{3,4\}^{10}\{6\}^{0}\}\leadsto 
\{\{1\}^{0}\{2\}^{0.1}\{5\}^{1}\{3,4\}^{10}\{6\}^{100}\}.
$$

\begin{center}
\begin{tikzpicture}[global scale=0.5]
\draw[-, thick](0,0)--(0,-1);
\draw[-, thick](-2,-1)--(2,-1);
\draw[-, thick](-2,-1)--(-2,-15);
\draw[-, thick](2,-1)--(2,-2);
\draw[-, thick](3,-2)--(1,-2);
\draw[-, thick](3,-2)--(3,-4);
\draw[dashed, thick](3,-4)--(3,-8.5);
\draw[-, thick](3,-8.5)--(3,-10.2);
\draw[dashed, thick](3,-10.2)--(3,-11.4);
\draw[-, thick](3,-11.4)--(3,-13.5);
\draw[-, thick](3,-13.5)--(3,-15);
\draw[-, thick](1,-2)--(1,-3);
\draw[dashed, thick](1,-3)--(1,-5);
\draw[-, thick](1,-5)--(1,-6);
\draw[-, thick](0,-6)--(2,-6);
\draw[-, thick](0,-6)--(0,-7);
\draw[-, thick](2,-6)--(2,-9.5);
\draw[dashed, thick](2,-9.5)--(2,-12);
\draw[-, thick](2,-12)--(2,-15);
\draw[dashed, thick](2,-14.5)--(2,-15);
\draw[dashed, thick](0,-7)--(0,-8);
\draw[-, thick](0,-8)--(0,-9);
\draw[-, thick](-1,-9)--(1,-9);
\draw[-, thick](-1,-9)--(-1,-10);
\draw[-, thick](1,-9)--(1,-11);
\draw[dashed, thick](-1,-10)--(-1,-15);
\draw[dashed, thick](1,-11)--(1,-13);
\draw[-, thick](1,-13)--(1,-14);
\draw[-, thick](0.5,-14)--(1.5,-14);
\draw[-, thick](0.5,-14)--(0.5,-15);
\draw[-, thick](1.5,-14)--(1.5,-15);
\draw[dotted, thick](-3,-11.6)--(4,-11.6);
\node at (-2,-15.2){1};
\node at (-1,-15.2){2};
\node at (0.5,-15.2){3};
\node at (1.5,-15.2){4};
\node at (2,-15.2){5};
\node at (3,-15.2){6};
\node at (6.1,-11.6){$\{1\}^{0}\{2\}^{0.1}\{3,4\}^{10}\{5\}^{1}\{6\}^{0}$};
\node at (-1.3,-13){$0.1$};
\node at (0.7,-12){$10$};
\node at (1.8,-10.6){$1$};
\node at (2.6,-10.8){$100$};
\node at (-0.4,-7.5){$200$};
\node at (0.8,-4){$2$};
\node at (2.7,-6.5){$0.5$};
\end{tikzpicture}

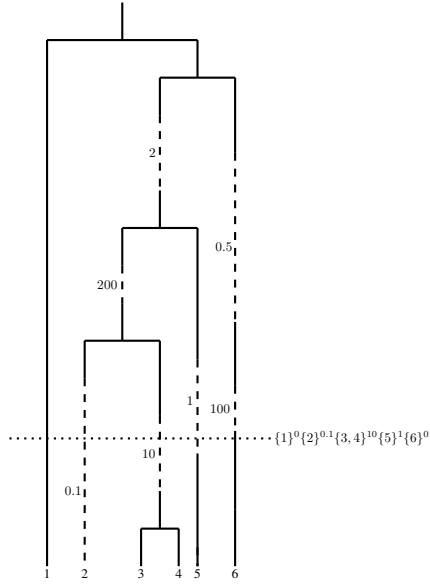
\captionof{figure}{A realization of continuum seed-bank $6$-coalescent}
\end{center}

Now, the continuum seed-bank $K$-coalescent process $\{\Pi_{t}^K\}_{t\geq 0}$ can be defined by (\ref{eq:K}) or constructed in a similar way to equation (\ref{eq:constructed}). It is obvious that $\{(N_{t}, M_{t})\}$ is indeed the block counting process of $\{\Pi_{t}^K\}_{t\geq 0}$ when $N_{0}+||M_{0}||_{TV}=K$. Furthermore, when $K$ is not finite, the following Proposition demonstrates the existence and uniqueness in distribution of the continuum seed-bank coalescent process $\{\Pi^{\infty}_{t}\}_{t\geq 0}$.

\begin{prop}\label{prop:projective}
There exists a unique process $\{\Pi^{\infty}_{t}\}_{t\geq 0}$ that has the same distribution as $\{\Pi^{K}_{t}\}_{t\geq 0}$ when it is restricted to take values in $\mathcal{P}^{f}_{K}$.		
\end{prop}

\begin{proof}
First, note that the state space   i.e. the number of different ways to partition a set with $K$-elements. It is closed since when a partition is given, the individuals in the same block should have the same flag, which is characterized by the equivalence relations of the coordinate mappings on $[0,\infty)^{K}$. As the result, $\mathcal{P}^{f}_{K}$ is a Polish space.

Then, for $K\in\mathbb{N}$, let $Q_{K}$ be the distribution of the coalescent process for the first $K$ individuals $\Pi^{K}$ on $D_{\mathbb{R}_{+}}(\mathcal{P}^{f}_{K})$ which is also Polish. Then, take the product space $\prod\limits_{K=1}^{\infty}D_{\mathbb{R}_{+}}(\mathcal{P}^{f}_{K})$ and consider its projective limit subset 
$$PL=\{\omega\in \prod\limits_{K=1}^{\infty}D_{\mathbb{R}_{+}}(\mathcal{P}^{f}_{K}):p^{J}_{L}(\omega_{J})=\omega_{L} \text{ for }J>L\},$$
where $p^{J}_{K}$ maps a $\mathcal{P}^{f}_{J}$-valued (first $J$ individuals) c\`{a}dl\`{a}g path to a $\mathcal{P}^{f}_{K}$-valued (first $K$ individuals) one as the restriction. Actually, $PL$ can be viewed as the space $D_{\mathbb{R}_{+}}(\mathcal{P}^{f}_{\infty})$ through a one-to-one correspondence. Since adding new individuals into the sample will not affect the coalescent process of the existing individuals, then we have the compatibility condition: $Q_{K}=Q_{J}\circ (p^{J}_{K})^{-1}, J>K$. 

Finally, by the Kolmogorov extension theorem, it is known that there exists a unique probability measure $Q_{\infty}$ on $PL$ such that $Q_{K}=Q_{\infty}\circ (p_{K})^{-1}$, where $p_{K}$ is the $K$-th coordinate mapping on $PL$. The proof is completed.
\end{proof}

As the dual of Theorem \ref{th:scaling}, it can be shown that the partition-valued ancestral process for a sample of size $K$ from the discrete-time model weakly converges to the continuum seed-bank $K$-coalescent process.

\begin{prop}\label{prop:ancestral}
For the sequence of discrete-time models in Theorem \ref{th:scaling}, their ancestral processes $\{\Pi^{N_{r},K}_{n}\}_{n\in\mathbb{N}_{0}}$ for a sample of size $K$ satisfy the following property:

As the initial distributions weakly converge to $\mathcal{L}(\Pi^{K}_{0})$,
$\{\Pi^{N_{r},K}(\lfloor N_{r}t\rfloor)\}_{t\geq 0}$ converges in distribution to the continuum seed-bank $K$-coalescent $\{\Pi^{K}_{t}\}_{t\geq 0}$ on $D_{\mathbb{R}_{+}}(\mathcal{P}^{f}_{K})$.
\end{prop}

\begin{proof}
Note that all of $\{\Pi^{N_{r},K}_{n}\}_{n\in\mathbb{N}_{0}}$ and $\{\Pi^{K}_{t}\}_{t\geq 0}$ are $\mathcal{P}^{f}_{K}$-valued Markov processes with bounded generators, and by the proof of Proposition \ref{prop:projective}, we will regard them as  $\{1,2,\cdots,B_{k}\}\times[0,\infty)^{K}$-valued. According to the dynamics of the discrete-time Wright-Fisher model given in Section 4, we know that the transition probability is:
\begin{eqnarray*}
	P(\Pi^{N_{r},K}_{1}=\pi^{\prime}|\Pi^{N_{r},K}_{0}=\pi)=\left\{\begin{array}{cl}
\frac{c^{n_{r}}_{i}}{N_{r}}, & \pi \leadsto \pi^{\prime}, \text{a}~0~\text{becomes}~\lambda^{n_{r}}_{i}, \text{ for } i=1,2,\cdots, n_{r},\\
\frac{\lambda^{n_{r}}_{i}}{N_{r}}, & \pi \leadsto \pi^{\prime}, \text{a}~\lambda^{n_{r}}_{i} ~\text{becomes}~0, \text{ for } i=1,2,\cdots,n_{r},\\
(1-\frac{c}{N_{r}})^{2}\frac{1}{N_{r}}, & \pi \sqsupset \pi^{\prime},\\
O(\frac{1}{N_{r}^{2}}),& \text{ otherwise,} 
\end{array}\right.
\end{eqnarray*}
where ``a'' means ``exactly one'' if there exists at least one.

Actually, similar to the proof of Proposition 3.4. in \cite{blath2016}, while transitioning from $\pi$ to $\pi^{\prime}$, multiple mergers and multiple flag-changings may happen at the same time, but most of them are of the order $O(\frac{1}{N_{r}^2})$ except that:

\begin{enumerate}
	\item $\pi \leadsto \pi^{\prime}$ and a $\lambda^{n_{r}}_{i}$ becomes $0$: the probability is $\sum\limits_{j=1}^{c}\frac{j}{N_{r}}P\{C^{n}_{i}=j\}=\frac{E[C^{n}_{i}]}{N_{r}}=\frac{c^{n}_{i}}{N_{r}}$;
	\item $\pi \leadsto \pi^{\prime}$ and a $0$ becomes $\lambda^{n_{r}}_{i}$: the probability is $\sum\limits_{j=1}^{c}\frac{j}{M^{N_{r}}_{i}}P\{C^{n}_{i}=j\}=\frac{E[C^{n}_{i}]}{M^{N_{r}}_{i}}=\frac{\lambda^{n}_{i}}{N_{r}}$;
	\item  $\pi \sqsupset \pi^{\prime}$ i.e. two 0-blocks merge: the probability is $(1-\frac{c}{N_{r}})^{2}\frac{1}{N_{r}^{2}}\cdot N_{r}$.
\end{enumerate} 

For any continuous bounded function $f$ on the locally compact Polish space $\mathcal{P}^{f}_{K}$, define
\begin{eqnarray*}
	\mathcal{A}_{n_{r}}f(\pi)&=&N_{r}E_{\pi}[f(\Pi^{N_{r},K}_{1})-f(\pi)]\\
	&=& \sum(f(\pi^{\prime}_{coalesce})-f(\pi))+\sum\sum_{i=1}^{n_{r}}\int_{(\lambda^{n^{r}}_{i-1},\lambda^{n_{r}}_{i}]}(f(\pi^{\prime}_{i})-f(\pi))\mu(d\lambda)+\sum\sum_{i=1}^{n_{r}}(f(\pi^{\prime\prime}_{i})-f(\pi))\lambda^{n_{r}}_{i}\\
	&&+O(\frac{1}{N_{r}}),
\end{eqnarray*}
where $\pi^{\prime}_{coalesce}$ denotes the partition after merging, $\pi^{\prime}_{i}$ denotes the partition after changing flag from $0$ to $\lambda^{n_{r}}_{i}$, $\pi^{\prime\prime}_{i}$ denotes the partition after changing flag from $\lambda^{n_{r}}_{i}$ to $0$, and ``$\sum$" refers to summing different cases up. Then, recall that by the transition rates (\ref{eq:K}), the generator $\mathcal{A}$ for $\{\Pi^{K}_{t}\}_{t\geq 0}$ is: 
\begin{eqnarray*}
\mathcal{A}f(\pi)=\sum(f(\pi^{\prime}_{coalesce})-f(\pi))+\sum\int_{(0,\infty)}(f(\pi^{\prime}_{\lambda})-f(\pi))\mu(d\lambda)+\sum\sum_{i=1}^{n_{r}}(f(\pi^{\prime}_{i})-f(\pi))\lambda^{n_{r}}_{i},	
\end{eqnarray*}
where $\pi^{\prime}_{\lambda}$ denotes the partition after changing flag from $0$ to $\lambda$. 
Therefore, we have
\begin{eqnarray*}
&&\sup_{\pi\in\mathcal{P}^{f}_{K}}|\mathcal{A}_{n_{r}}f(\pi)-\mathcal{A}f(\pi)|\\
&\leq &\sup_{\pi\in\mathcal{P}^{f}_{K}}\{\sum\sum_{i=1}^{n_{r}}\int_{(\lambda^{n^{r}}_{i-1},\lambda^{n_{r}}_{i}]}|f(\pi^{\prime}_{\lambda})-f(\pi_{i}^{\prime})|\mu(d\lambda)+\sum\int_{(\lambda^{n_{r}}_{n_{r}},\infty)}|f(\pi^{\prime}_{\lambda})-f(\pi)|\mu(d\lambda)\}+O(\frac{1}{N_{r}}).
\end{eqnarray*}
Since $f$ is bounded, the ``$\sum$" has finite terms and this number depends only on $K$, $\lim\limits_{r\rightarrow \infty}\lambda^{n_{r}}_{n^{r}}=0$, and  $\lim\limits_{n\rightarrow\infty}\max\limits_{1\leq i\leq n}|\lambda^{n}_{i}-\lambda^{n}_{i-1}|=0$, we first take a large enough $R$ such that the last two terms are of the same order of the given $\epsilon>0$. Then, by the fact that $\pi^{\prime}_{\lambda}$ and $\pi_{i}^{\prime}$ are the same except for the flag of one block i.e. $\lambda$ versus $\lambda^{n_{r}}_{i}$, we know that for each partition $i$, and each possible block which is changing flag and composed of individuals $\{i_{1},i_{2},\cdots,i_{j}\}$, the function $f(\{i\}\times(x_{1},\cdots,x_{K}))$ is uniformly continuous on $[0,\lambda^{n_{R}}_{n_{R}}]^{j}$ with respect to variables $(x_{i_{1}},\cdots, x_{i_{j}})$. Note that there are only finitely many cases, hence we can get the uniformity with respect to all of them. As the result, if $r$ is large enough, then $\max\limits_{1\leq i\leq n}|\lambda^{n}_{i}-\lambda^{n}_{i-1}|$ is small, and $|f(\pi^{\prime}_{\lambda})-f(\pi_{i}^{\prime})|$ are uniformly bounded by $\epsilon$, and then we obtain $\sup\limits_{\pi\in\mathcal{P}^{f}_{K}}|\mathcal{A}_{n_{r}}f(\pi)-\mathcal{A}f(\pi)|=O(\epsilon)$. 

Finally, for these Feller processes with bounded generators, we only need to apply Theorem 4.2.6 and Theorem 1.6.5 in \cite{ethier2009markov} to get the desired result.
\end{proof}



\section*{Appendix}

\begin{proof}[\textbf{Proof of Proposition \ref{prop:domain}}]~

For 1, it is obvious that $D\subseteq D(A)$ and $D$ is closed. $D(A)$ contains $\mathbb{R}\times \mathcal{S}$ which is dense in $E$, hence $D(A)$ is also a dense subset. $D(A)$ is Borel measurable as $D(A)=\bigcup\limits_{n\in\mathbb{N}}\Gamma_{n}$, where 
$$\Gamma_{n}=\mathbb{R}\times \{y\in L^{1}(\mu): ||\lambda y(\lambda)||_{L^{1}}\leq n\}$$
is closed by the Fatou's lemma. For 2, it can be verified directly by $|p(x)-p(y)|\leq|x-y|$.
\end{proof}

\begin{proof}[\textbf{Proof of Proposition \ref{prop:semigroup}}]~

It can be verified by definition that $A$ is a closed operator. In addition, by 1 of Proposition \ref{prop:domain}, we know that $D(A)$ is dense. Then, by the Hille-Yosida Theorem (see e.g. Theorem 1.3.1 in \cite{pazy2012semigroups}), we only need to show that the resolvent set $\rho(A)$ contains $(0,\infty)$ and for every $\alpha>0$, $||R(\alpha, A)||\leq\frac{1}{\alpha}$. They are obvious as $R(\alpha, A)(x, y(\lambda))=(\frac{x}{\alpha+c}, \frac{y(\lambda)}{\alpha+\lambda})$ for $\alpha>0$. Moreover, for every $\sigma>0$ and $\tau\in\mathbb{R}\backslash\{0\}$, we have $||R(\sigma+i\tau, A)||\leq\frac{1}{|\tau|}$. Consequently, by Theorem 2.5.2 and Theorem 1.7.7 in \cite{pazy2012semigroups}, $\{S(t)\}_{t\geq 0}$ can be extended to an analytic semigroup in a sector
$\Delta_{\delta}=\{z\in\mathbb{C}:|arg~z|<\delta\}$ for some $\delta\in (0,\frac{\pi}{2})$ such that
\begin{eqnarray}\label{eq:a1}
S(t)=\frac{1}{2\pi i}\int_{\Gamma_{\epsilon, \theta}}e^{\alpha t}R(\alpha, A)d\alpha,
\end{eqnarray}
where the curve $\Gamma_{\epsilon,\theta}=\Gamma_{\epsilon, \theta}^{+}\cup\Gamma_{\epsilon,\theta}^{-}\cup\Gamma_{\epsilon, \theta}^{0}$ is oriented counterclockwise, 
$\Gamma_{\epsilon,\theta}^{\pm}=\{z\in\mathbb{C}:z=re^{\pm i\theta},r\geq\epsilon\}$, and $\Gamma_{\epsilon, \theta}^{0}=\{z\in\mathbb{C}:z=\epsilon e^{i\phi},|\phi|\leq\theta\}$, for any $\epsilon>0$ and $\theta\in(\frac{\pi}{2},\frac{\pi}{2}+\delta)$. By the inversion formula (\ref{eq:a1}) and the Cauchy's integral formula, we have $S(t)=(e^{-ct}, e^{-\lambda t})$. Alternatively, since $A$ generates a strongly continuous semigroup, and it can be verified that $D(A)\subseteq D(B)$, where $B$ denotes the generator of $\{S(t)\}_{t\geq 0}$. Therefore, $A$ and $B$ must be the same, and then we conclude that $A$ indeed generates $\{S(t)\}_{t\geq 0}$.
\end{proof}

\begin{proof}[\textbf{Proof of Proposition \ref{prop:compactsemigroup}}]~

Any finite measure $\mu$ on $(0, \infty)$ can be uniquely decomposed as $\mu=\mu_{c}+\mu_{d}$, where $\mu_{c}$ is a non-atomic finite measure, and $\mu_{d}=\sum\limits_{i=1}^{\infty}c_{i} \delta_{\lambda_{i}}$ for $c_{i}\geq0$, $\lambda_{i}\neq\lambda_{j}$ when $i\neq j$. 

\textit{Case 1} : If $\mu=\sum\limits_{i=1}^{n}c_{i} \delta_{\lambda_{i}}$ for some $n\in \mathbb{N}$, then $E$ is finite-dimensional, and thus $S(t)$ is compact for $t\geq 0$.

\textit{Case 2} : If $\mu=\sum\limits_{i=1}^{\infty}c_{i} \delta_{\lambda_{i}}$ for $c_{i}>0$ and $\{\lambda_{i}\}_{i\in \mathbb{N}}$ is unbounded, then there is a subsequence $\lambda_{n_{k}}\rightarrow \infty$ as $k\rightarrow \infty$. Note that $E\cong\mathbb{R}\times l^{1}(w)$, where $l^{1}(w)$ is the $w$-weighted $l^{1}$ space. For the compactness of $S(t)$, $t>0$, we only need to show that the image of the closed unit ball $\{||y||_{l^{1}(w)}\leq 1\}$ under the mapping $\{y_{i}\}_{i\in\mathbb{N}}\mapsto \{e^{-\lambda_{i}t}y_{i}\}_{i \in\mathbb{N}}$ is totally bounded.
By a weighted version of Theorem 4 in \cite{hanche2010kolmogorov}, a subset of $l^{1}(w)$ is totally bounded if and only if it is pointwise bounded and for any $\epsilon>0$, there is some $n\in\mathbb{N}$ such that for all $y$ in the set, $\sum\limits_{i=n+1}^{\infty}c_{i}|y_{k}|<\epsilon$. Since $\sup\limits_{||y||_{l^{1}(w)}\leq 1}|e^{-\lambda_{i}t}y_{i}|\leq \frac{1}{c_{k}}$ (pointwise bounded), and 
$\sup\limits_{||y||_{l^{1}(w)}\leq 1}\sum\limits_{i=n_{k}+1}^{\infty}c_{i}|e^{-\lambda_{i}t}y_{i}|\leq e^{-\lambda_{n_{k}}t}<\epsilon$ when $k$ is large enough, the conclusion follows.

Otherwise, if $\{\lambda_{i}\}_{i\in \mathbb{N}}$ is bounded by $\lambda>0$, take $y^{k}$ as $y^{k}_{i}=\frac{1}{c_{i}}\delta_{ik}$ for each $k\in\mathbb{N}$. Then for $\epsilon_{0}=e^{-\lambda t}$ and any $n\in\mathbb{N}$, we have
$\sup\limits_{||y||_{l^{1}(w)}\leq 1}\sum\limits_{i=n+1}^{\infty}c_{i}|e^{-\lambda_{i}t}y_{i}|\geq \sum\limits_{i=n+1}^{\infty}c_{i}|e^{-\lambda_{i} t}y^{n+1}_{i}|\geq \epsilon_{0}$, thus $S(t)$ is not compact for $t>0$.

\textit{Case 3} : If $\mu\neq 0$ is non-atomic, then by the Sierpinski theorem (\cite{sierpinski1922fonctions}), it takes a continuum of values. Take some $\epsilon\in (0,1)$ such that $M_{\epsilon}=(0, -\frac{ln(\epsilon)}{t})$ has positive $\mu$ measure, and then take a decreasing sequence of sets $\{M_{n}\}_{n\in\mathbb{N}}$ such that $M_{1}\subseteq M_{\epsilon}$, $0<\mu(M_{1})<\mu(M_{\epsilon})$, and $0<\mu(M_{n+1})<\frac{1}{2}\mu(M_{n})$. Let $f_{n}(\lambda)=\frac{I_{M_{n}}(\lambda)}{\mu(M_{n})}$, then $\{z_{n}=(0, f_{n})\}_{n\in\mathbb{N}}$ is contained in the closed unit ball of $E$, and for $m>n$, 
\begin{eqnarray*}
	||S(t)(z_{m}-z_{n})||&=&\int_{(0, \infty)}e^{-\lambda t}|f_{m}(\lambda)-f_{n}(\lambda)|\mu(d\lambda),\\
	&\geq& \epsilon \int_{M_{m}}(\frac{1}{\mu(M_{m})}-\frac{1}{\mu(M_{n})  })\mu(d\lambda)+\epsilon \int_{M_{n}\backslash M_{m}}\frac{1}{\mu(M_{n})}\mu(d\lambda),\\ 
	&=& 2\epsilon \frac{\mu(M_{n} \backslash M_{m})}{\mu(M_{n})}\\
	&\geq &  \epsilon,
\end{eqnarray*}
which implies that $S(t), t>0$ is not compact.

\textit{Case 4} : 
For general $\mu=\mu_{c}+\mu_{d}$, if $\mu_{d}=\sum\limits_{i=1}^{\infty}c_{i}\delta_{\lambda_{i}}$ with bounded $\{\lambda_{i}\}_{i\in\mathbb{N}}$, take 
$y^{k}(\lambda)=\frac{1}{c_{i}}I_{\{\lambda_{k}\}}(\lambda)$ as in Case 2, then $\{e^{-\lambda t}y^{k}(\lambda)\}_{k\in\mathbb{N}}$ has no $L^{1}(\mu_{d})$-convergent subsequence. Note that $y^{k}=0,\mu_{c}-a.e.$, thus $||y^{k}||_{L^{1}(\mu)}=||y^{k}||_{L^{1}(\mu_{d})}$, and
$\{e^{-\lambda t}y^{k}(\lambda)\}_{k\in\mathbb{N}}$ has no $L^{1}(\mu)$-convergent subsequence, either. 

Finally, 
 if $\mu_{c}\neq 0$ and $\mu_{d}=\sum\limits_{i=1}^{n}c_{i}\delta_{\lambda_{i}}$ or $\mu_{d}=\sum\limits_{i=1}^{\infty}c_{i}\delta_{\lambda_{i}}$ with unbounded $\{\lambda_{i}\}_{i\in\mathbb{N}}$, then we take $M_{\epsilon}$ and $\{M_{n}\}_{n\in\mathbb{N}}$ as in Case 3 but remove all $\lambda_{i}$'s. Since the restriction of $\mu_{d}$ to $M_{\epsilon}$ is $0$, the same argument follows. \end{proof}
 
\begin{proof}[\textbf{Proof of Proposition \ref{prop:adjoint}}]~

1. For any $z=(x, y)\in D(A)$ and $f=(f^{(1)}, f^{(2)})\in E^{\star}$, $\langle Az, f\rangle= (-cf^{(1)}x, -\int_{(0,\infty)}\lambda y(\lambda)f^{(2)}(\lambda)\mu(d\lambda))$. By definition, 
$D(A^{*})=\{f\in E^{\star}: z \mapsto \langle Az,f\rangle~\text{is continuous}\}$. Then ``$\supseteq$" is obvious. For ``$\subseteq$'', $f\in D(A^{\star})$ implies that 
$\int_{(0,\infty)}\lambda h(\lambda)y(\lambda)\mu(d\lambda)=\int_{(0,\infty)}\tilde{h}(\lambda)y(\lambda)\mu(d\lambda)$ for some $\tilde{h}\in L^{\infty}(\mu)$. Since $(x, I_{B})\in D(A)$ for any $x\in \mathbb{R}$ and $B\in\mathcal{B}(0, \infty)$, then we know that $\lambda h(\lambda)=\tilde{h}(\lambda)$, $\mu$-a.e.

2. $D(A^{\star})$ is not dense since the constant function $1$ can not be approximated by functions in $\{h \in L^{\infty}(\mu): \lambda h(\lambda)\in L^{\infty}(\mu)\}$. $D(A^{\star})$ is also not closed e.g. $f_{n}(\lambda)=\frac{1}{1+\sqrt{\lambda}}I_{\{\lambda\leq n\}}\rightarrow f(\lambda)=\frac{1}{1+\sqrt{\lambda}}$ in $L^{\infty}(\mu)$, and $\lambda f_{n}(\lambda)\in L^{\infty}$, but $\lambda f(\lambda)\notin L^{\infty}$.

3. For $\alpha\in \rho(A)$, by Lemma 1.10.2. in \cite{pazy2012semigroups}, $\alpha\in\rho(A^{\star})$ and $R(\alpha, A^{\star})=R(\alpha, A)^{\star}$. Then by Theorem 1.10.4. in \cite{pazy2012semigroups}, we have $R(\alpha, A^{+})=R(\alpha, A^{\star})|_{\overline{D(A^{\star})}}=R(\alpha, A)^{\star}|_{\overline{D(A^{\star})}}$, thus
$D(A^{+})=R(\alpha, A)^{*}(\overline{D(A^{\star})})=\{(\frac{f^{(1)}}{\alpha+c}, \frac{f^{(2)}(\lambda)}{\alpha+\lambda}): (f^{(1)}, f^{(2)}(\lambda))\in \overline{D(A^{\star})}, \alpha\in\rho(A)\}$ by the proof of Proposition \ref{prop:semigroup}.

$D(A^{+})\subsetneq D(A^{\star})$: e.g. $f(\lambda)=\left\{\begin{array}{ll}1, & \lambda \leqslant 1 \\ \frac{1}{\lambda}, & \lambda>1\end{array}\right.\in \{h \in L^{\infty}(\mu): \lambda h(\lambda)\in L^{\infty}(\mu)\}$ but $\lambda f(\lambda)$ can not be approximated by functions in $\{h \in L^{\infty}(\mu): \lambda h(\lambda)\in L^{\infty}(\mu)\}$.

$D(A^{\star 2})\subsetneq D(A^{+})$: e.g. $g_{n}(\lambda)=\left\{\begin{array}{lll} \lambda, & \lambda \leqslant 1 \\ \frac{1}{\sqrt{\lambda}}, & 1<\lambda\leq n\\ 0, &, \lambda>n \end{array}\right.\in D(A^{\star})$ converges to $\lambda f(\lambda)$, where $f(\lambda)=\left\{\begin{array}{ll}1, & \lambda \leqslant 1 \\ \frac{1}{\lambda^{\frac{3}{2}}}, & \lambda>1\end{array}\right.$, thus $f\in D(A^{+})$ by definition, but $\lambda^{2}f(\lambda)\notin L^{\infty}(\mu)$.
\end{proof}

\begin{proof}[\textbf{Proof of Proposition \ref{prop:equi}}]~

1 $\Rightarrow$ 2 : It is known that $\int_{0}^{t}S(s){Z_{0}}ds\in D(A)$, and by the Fubini theorem and change of variables, we have
$$\int_{0}^{t}\int_{0}^{s}S(s-r)\tilde{F}(Z_{r})drds=\int_{0}^{t}\int_{0}^{t-r}S(s)\tilde{F}(Z_{r})dsdr.$$
Moreover, $\int_{0}^{t-r}S(s)\tilde{F}(Z_{r})ds\in D(A)$ and $A\int_{0}^{t-r}S(s)\tilde{F}(Z_{r})ds=S(t-r)\tilde{F}(Z_{r})-\tilde{F}(Z_{r})$ is Bochner integrable on $(0,t)$, thus $\int_{0}^{t}\int_{0}^{s}S(s-r)\tilde{F}(Z_{r})drds\in D(A)$ and
\begin{equation}\label{eq:1}
A\int_{0}^{t}\int_{0}^{s}S(s-r)\tilde{F}(Z_{r})drds=\int_{0}^{t}S(t-r)\tilde{F}(Z_{r})dr-\int_{0}^{t}\tilde{F}(Z_{r})dr.
\end{equation}
Similarly, by the stochastic Fubini theorem and change of variables, we have
$$\int_{0}^{t}\int_{0}^{s}S(s-r)\tilde{B}(Z_{r})dW_{r}ds=\int_{0}^{t}\int_{0}^{t-r}S(s)\tilde{B}(Z_{r})dsdW_{r}, a.s.$$
Then, it follows from Proposition 4.30 in \cite{da2014stochastic} that $\int_{0}^{t}\int_{0}^{s}S(s-r)\tilde{B}(Z_{r})dW_{r}ds\in D(A), a.s.$, and
\begin{equation}\label{eq:2}
A\int_{0}^{t}\int_{0}^{s}S(s-r)\tilde{B}(Z_{r})dW_{s}ds=\int_{0}^{t}S(t-r)\tilde{B}(Z_{r})dW_{r}-\int_{0}^{t}\tilde{B}(Z_{r})dW_{r}, a.s.
\end{equation}

In summary, $Z_{t}$ is Bochner integrable on $(0,t)$, a.s., and
$$\int_{0}^{t}Z_{s}ds=\int_{0}^{t}S(s)Z_{0}ds+\int_{0}^{t}\int_{0}^{s}S(s-r)\tilde{F}(Z_{r})drds+\int_{0}^{t}\int_{0}^{s}S(s-r)\tilde{B}(Z_{r})dW_{r}ds\in D(A), a.s.,$$
and then combining (\ref{eq:1}) and (\ref{eq:2}) together, we have
\begin{eqnarray*}
&&A\int_{0}^{t}Z_{s}ds\\
&=&\int_{0}^{t}S(s)Z_{0}ds+\int_{0}^{t}S(t-r)\tilde{F}(Z_{r})dr-\int_{0}^{t}\tilde{F}(Z_{r})dr+\int_{0}^{t}S(t-r)\tilde{B}(Z_{r})dW_{r}-\int_{0}^{t}\tilde{B}(Z_{r})dW_{r}\\
&=&Z_{t}-Z_{0}-\int_{0}^{t}\tilde{F}(Z_{r})dr-\int_{0}^{t}\tilde{B}(Z_{r})dW_{r}, a.s.
\end{eqnarray*}

2 $\Rightarrow$ 3 : 
For $t\geq 0$ and $f\in D(A^{+})$, we have
\begin{eqnarray*}
\langle Z_{t}, f\rangle &=& \langle Z_{0}, f\rangle+\langle A\int_{0}^{t}Z_{s}ds, f\rangle +\int_{0}^{t}\langle \tilde{F}(Z_{s}), f\rangle ds+\int_{0}^{t}\langle \tilde{B}(Z_{s})dW_{s},f\rangle\\
&=&\langle Z_{0}, f\rangle+\int_{0}^{t}\langle Z_{s}, A^{+}f\rangle ds+\int_{0}^{t}\langle \tilde{F}(Z_{s}), f\rangle ds+\int_{0}^{t}\langle \tilde{B}(Z_{s})dW_{s},f\rangle.
\end{eqnarray*}
Since $D(A^{+})$ is dense in $\overline{D(A^{\star})}$, then for any $f\in D(A^{*})$, it still holds by approximation.

3 $\Rightarrow$ 1 : 
First, by definition, for any $f\in D(A^{+2})=\{f\in D(A^{+}): A^{+}f\in D(A^{+})\}$, 
\begin{eqnarray*}
\langle Z_{t}, f\rangle &=& \langle Z_{0}, f\rangle+\int_{0}^{t}\langle Z_{s}, A^{*}f\rangle ds+\int_{0}^{t}\langle \tilde{F}(Z_{s}), f\rangle ds+\int_{0}^{t}\langle \tilde{B}(Z_{s})dW_{s},f\rangle\\
&=& \langle Z_{0}, f\rangle+\int_{0}^{t}\langle Z_{s}, A^{*}f\rangle ds+\int_{0}^{t}\langle \tilde{F}(Z_{s}), f\rangle ds+\int_{0}^{t}\langle \tilde{B}(Z_{s})dW_{s},S^{*}(t-s)f-\int_{0}^{t-s}S^{*}(u)A^{+}fdu\rangle.
\end{eqnarray*}

By the stochastic Fubini theorem and change of variables, we have
$$\int_{0}^{t}\langle \tilde{B}(Z_{s})dW_{s}, \int_{0}^{t-s}S^{*}(u)A^{+}fdu\rangle=\int_{0}^{t}\int_{0}^{s}\langle \tilde{B}(Z_{u})dW_{u}, S^{*}(t-s)A^{+}f\rangle ds, a.s.$$

Since $S^{*}(t-s)A^{+}f\in D(A^{+})$, then by definition,
\begin{eqnarray*}
&&\int_{0}^{s}\langle \tilde{B}(Z_{u})dW_{u}, S^{*}(t-s)A^{+}f\rangle \\
&=&\langle Z_{s}-Z_{0}, S^{*}(t-s)A^{+}f\rangle-\int_{0}^{s}\langle Z_{u}, A^{+}S^{*}(t-s)A^{+}f\rangle du-\int_{0}^{s}\langle \tilde{F}(Z_{u}), S^{*}(t-s)A^{+}f\rangle du,
\end{eqnarray*}
and by the Fubini theorem and change of variables, we get 
\begin{eqnarray*}
\int_{0}^{t}\int_{0}^{s}\langle Z_{u}, A^{+}S^{*}(t-s)A^{+}f\rangle duds&=&\int_{0}^{t}\langle Z_{u}, \int_{u}^{t}A^{+}S^{*}(s-u)A^{+}fds\rangle du,\\
&=& \int_{0}^{t}\langle Z_{u}, S^{*}(t-u)A^{+}f-A^{+}f\rangle du,
\end{eqnarray*}
and
\begin{eqnarray*}
\int_{0}^{t}\int_{0}^{s}\langle \tilde{F}(Z_{u}), S^{*}(t-s)A^{+}f\rangle duds&=&\int_{0}^{t}\langle \tilde{F}(Z_{u}), \int_{u}^{t}S^{*}(s-u)A^{+}fds\rangle du,\\
&=&\int_{0}^{t}\langle \tilde{F}(Z_{u}), S^{*}(t-u)f-f\rangle du,
\end{eqnarray*}
Therefore, 
\begin{eqnarray*}
\langle Z_{t}, f\rangle &=& \langle Z_{0}, f\rangle +\int_{0}^{t}\langle Z_{s}, A^{*}f\rangle ds+\int_{0}^{t}\langle \tilde{F}(Z_{s}), f\rangle ds+\int_{0}^{t}\langle S(t-s)\tilde{B}(Z_{s})dW_{s},f\rangle\\
&&-\int_{0}^{t}\langle Z_{s}-Z_{0}, S^{\star}(t-s)A^{+}f\rangle ds+\int_{0}^{t}\langle Z_{s}, S^{*}(t-s)A^{+}f-A^{+}f\rangle ds\\
&&+\int_{0}^{t}\langle \tilde{F}(Z_{s}), S^{*}(t-s)f-f\rangle ds\\
&=&\langle Z_{0},f\rangle +\langle \int_{0}^{t}S(t-s)\tilde{B}(Z_{s})dW_{s}, f\rangle + \langle Z_{0}, \int_{0}^{t}S^{\star}(t-s)A^{+}f\rangle ds+\langle \int_{0}^{t}S(t-s)\tilde{F}(Z_{s})ds, f\rangle\\
&=&\langle S(t)Z_{0}+\int_{0}^{t}S(t-s)\tilde{F}(Z_{s})ds+\int_{0}^{t}S(t-s)\tilde{B}(Z_{s})dW_{s}, f\rangle, a.s.
\end{eqnarray*} 

Finally, by Theorem 1.2.7 in \cite{pazy2012semigroups}, $D(A^{+2})$ is dense in $\overline{D(A^{*})}$, hence it is also weak$^{*}$ dense in $\overline{D(A^{*})}$. 
By Theorem 4.7.2 in \cite{schaefer1971locally}, $D(A^{*})$ is weak$^{\star}$ dense in $E^{\star}$ since $A$ is closed. As the result, $D(A^{+2})$ is weak$^{\star}$ dense in $E^{\star}$, and then the above equality holds for any $f\in E^{\star}$ by approximation. The proof is completed.
\end{proof}

\begin{proof}[\textbf{Proof of proposition \ref{prop:pathwise}}]~

As in the proof of Theorem 4.3.2 in \cite{ikeda2014stochastic}, there exist functions $\phi_{n}\in C^{2}(\mathbb{R})$ such that $|\phi_{n}^{\prime}(x)|\leq 1$, $\phi_{n}(x)\geq|x|-1$, $\phi_{n}(x)\uparrow |x|$ as $n\rightarrow \infty$, and $0\leq \phi^{\prime\prime}_{n}(x)\leq \frac{2}{n|x|}$. 

For any two solutions $Z=(X, Y)$ and $\tilde{Z}=(\tilde{X}, \tilde{Y})$ to equation (\ref{eq:SEE}) on the same $(\Omega, \mathcal{F}, P, \mathfrak{F}, W)$, since 
$$Y_{t}(\lambda)-\tilde{Y}_{t}(\lambda)=\int_{0}^{t}e^{-\lambda(t-s)}\lambda(X_{s}-\tilde{X}_{s})ds,$$ we only need to show the pathwise uniqueness of $X$. By It\^{o}'s formula, we have 
\begin{eqnarray*}
\phi_{n}(X_{t}-\tilde{X}_{t})&=&\int_{0}^{t}\phi^{\prime}_{n}(X_{s}-\tilde{X}_{s})\int_{(0,\infty)}(Y_{s}(\lambda)-\tilde{Y}_{s}(\lambda))\mu(d\lambda)ds-c\int_{0}^{t}\phi^{\prime}_{n}(X_{s}-\tilde{X}_{s})(X_{s}-\tilde{X}_{s})ds\\
&&+\int_{0}^{t}\phi^{\prime}_{n}(X_{s}-\tilde{X}_{s})(\sqrt{X_{s}(1-X_{s})}-\sqrt{\tilde{X}_{s}(1-\tilde{X}_{s})} )dW_{s}\\
&&+\frac{1}{2}\int_{0}^{t}\phi_{n}^{\prime\prime}(X_{s}-\tilde{X}_{s})|\sqrt{X_{s}(1-X_{s})}-\sqrt{\tilde{X}_{s}(1-\tilde{X}_{s})}|^{2}ds.
\end{eqnarray*}
Therefore, 
$$E[\phi_{n}(X_{t}-\tilde{X}_{t})]\leq \int_{0}^{t}E[|\int_{(0,\infty)}(Y_{s}(\lambda)-\tilde{Y}_{s}(\lambda))\mu(d\lambda)|]ds+c\int_{0}^{t}E{|X_{s}-\tilde{X}_{s}|}ds+\frac{t}{n}$$.
By $\phi_{n}(x)\geq|x|-1$ and the monotone convergence theorem, letting $n\rightarrow \infty$, we have 
\begin{equation}\label{t1}
E[|X_{t}-\tilde{X}_{t}|]\leq c\int_{0}^{t}E{|X_{s}-\tilde{X}_{s}|}ds+\int_{0}^{t}E[||Y_{s}-\tilde{Y}_{s}||_{L^{1}}]ds.
\end{equation}
Moreover,
\begin{eqnarray}\label{t2}
	E[||Y_{t}-\tilde{Y}_{t}||_{L^{1}}]=E[\int_{(0,\infty)}|\int_{0}^{t}e^{-\lambda(t-s)}\lambda(X_{s}-\tilde{X}_{s})ds|\mu(d\lambda)]\leq c^{\prime}\int_{0}^{t}E[|X_{s}-\tilde{X}_{s}|]ds.
\end{eqnarray}

Combining (\ref{t1}) and (\ref{t2}) together, and by the Gr\"{o}nwall's inequality, we know that $E[|X_{t}-\tilde{X}_{t}|]=0$ for any $t\geq 0$. The proof is completed.
\end{proof}

\begin{proof}[\textbf{Proof of Proposition \ref{prop:semigroups}}]~
	
1. As a corollary of 2 in Theorem \ref{th:wellposed}, for fixed $t>0$, if $Z^{n}_{0}\rightarrow Z_{0}$ in $D$, a.s., then $Z^{n}_{t}$ convergences in probability to $Z_{t}$. By the dominated convergence theorem	, we have the Feller property.

2. Let $Z^{z}$, $Z^{z^{\prime}}$ be two solutions. For any $T>0$ and $t,t+h\in [0, T]$, by 2 of Theorem \ref{th:wellposed}, we have
\begin{eqnarray*}
E[||Z^{z^{\prime}}_{t+h}-Z^{z}_t||]&\leq& E[||Z^{z^{\prime}}_{t+h}-Z^{z}_{t+h}||]+E[||Z^{z}_{t+h}-Z^{z}_t||]\\
&\leq & E[\sup_{t\in[0,T]}||Z^{z^{\prime}}_t-Z^{z}_t||]+E[||Z^{z}_{t+h}-Z^{z}_t||]\\
&\leq & C|z^{\prime}-z|+E[||Z^{z}_{t+h}-Z^{z}_t||].
\end{eqnarray*}
Since 
$$Z^{z}_{t+h}-Z^{z}_{t}=(S(t+h)-S(t))Z_{0}+\int_{0}^{t}(S(t+h-s)-S(t-s))dM
^{z}_{s}+\int_{t}^{t+h}S(t+h-s)dM^{z}_{s},$$
where $dM_{s}=F(Z^{z}_{s})ds+B(Z^{z}_{s})dW_{s}$, and $||S(t+h)-S(t)||\leq (c+c^{\prime})|h|$, then by the Burkholder-Davis-Gundy inequality, we have $E[||Z^{z}_{t+h}-Z^{z}_t||]=O(|h|+\sqrt{|h|})$.

By the fact that $Z^{z^{\prime}}_{t+h}$ converges to $Z^{z}_t$ in probability as $(z^{\prime}, t+h)$ converges to $(z, t)$, and the dominated convergence theorem, we know that $T_{t}f(z)$ is jointly continuous in $(z, t)$ when $f\in C_{b}(D)$. For $f=I_{A}$, where $A$ is a closed subset of $D$, it can be approximated by a sequence $\{f_{n}\}_{n\in\mathbb{N}}\subseteq C_{b}(D)$, and hence $E[I_{A}(Z^{z}_{t})]$ is jointly measurable in $(z, t)$. Finally, by a monotone class argument, $T_{t}f(z)$ is jointly measurable for any $f\in B(D)$.

3. Let $D_{0}$ be a countable dense subset of $D$, by 2 of Theorem \ref{th:wellposed}, and 
$\{\sup\limits_{z\in D}|T_{t}f(z)|\leq r\}=\bigcap\limits_{z\in D_{0}}\{T_{t}f(z)\leq r\}$, the conclusion follows.
\end{proof}

\begin{proof}[\textbf{Proof of Proposition \ref{prop:markovs}}]~

We prove the Markov property i.e. 
$E[f(Z_{t+s}^{z})|\overline{\mathcal{F}}_{s}^{W}]=T_{t}f(Z^{z}_{s}),a.s.$ for $s, t\geq 0 $ at first. 

Since
$$\begin{aligned} Z_{t+s}^z &=z+\int_0^{t+s} AZ_u^z d u+\int_0^{t+s} F\left(Z_u^z\right) d u+\int_0^{v+s} B\left(Z_u^z\right) d W_u \\ &=Z_s^z+\int_s^{t+s} A Z_u^z d u+\int_s^{t+s} F\left(Z_u^z\right) d u+\int_s^{t+s} B\left(Z_u^z\right) d W_u \\&=Z_s^z+\int_0^t A Z_{s+u}^z d u+\int_0^t F\left(Z_{s+u}^z\right) d u+\int_0^t B\left(Z_{s+u}^z\right) d W_u^s \end{aligned},$$
where $W^{s}_{u}=W_{s+u}-W_{s}$ is a standard Brownian Motion which is independent of $\overline{\mathcal{F}}^{W}_{s}$, then $(Z, W^{s})$ is a weak solution on $(\Omega, \mathcal{F}, P, \{\overline{\mathcal{F}}^{W}_{s}\vee \overline{\mathcal{F}}^{W^{s}}_{t}\}_{t\geq 0})$, and thus by Definition \ref{defn:strongp} and Theorem \ref{th:wellposed}, we have
$$Z^{z}_{t+s}(\omega)=\Phi(Z^{z}_{s}(\omega), t, W^{s}_{\cdot}(\omega))~\text{for}~t\geq0, a.s.$$
We need to show that $E[f(\Phi(Z^{z}_{s}, t, W^{s}_{\cdot})|\overline{\mathcal{F}}^{W}_{s}]=T_{t}f(Z^{z}_{s}), a.s.$ for all $f\in B(D)$, and this is true since $f(\Phi(\cdot,t,\cdot))$ is a bounded $\mathcal{B}(D)\times \overline{\mathcal{B}}_{\infty}/ \mathcal{B}(D)$-measurable function, $Z^{z}_{s}$ is $\overline{\mathcal{F}}^{W}_{s}$-measurable, and $W^{s}$ is independent of $\overline{\mathcal{F}}^{W}_{s}$.

Then, for any $\mathfrak{F}$-stopping time $\tau<\infty$, a.s., we need to show that
$$\int_{A}f(Z^{z}_{t+\tau})dP=\int_{A}T_{t}f(Z^{z}_{\tau})dP,~\text{for}~A\in \overline{\mathcal{F}}^{W}_{\tau}.$$
For discrete $\tau$ taking values in $\{t_{1},t_{2},\cdots\}$, we have
$$\int_{A\bigcap \{\tau=t_{i}\}}f(Z^{z}_{t+\tau})dP=\int_{A\bigcap\{\tau=t_{i}\}}f(Z^{z}_{t+t_{i}})dP=\int_{A\bigcap\{\tau=t_{i}\}}T_{t}f(Z^{z}_{t_{i}})dP=\int_{A\bigcap\{\tau=t_{i}\}}T_{t}f(Z^{z}_{\tau})dP$$
as $A\bigcap\{\tau=t_{i}\}\in \overline{\mathcal{F}}^{W}_{t_{i}}$. Adding them up yields the desired result.

For general $\tau$, it is the pathwise limit of a decreasing sequence $\{\tau_{n}\}$ of discrete stopping times, thus the result follows by the continuity of sample paths. By the Markov property, $T_{t+s}f=T_{t}\circ T_{s}f$ can be verified directly. 	
\end{proof}

\begin{proof}[\textbf{Proof of Proposition \ref{prop:mps}}]~

By Definition \ref{defn:martingale} and 1 of Proposition \ref{prop:full}, any weak solution to equation (\ref{eq:SEE}) with initial distribution $\nu$ is a solution to the $C_{\mathbb{R}_{+}}(D)$-martingale problem for $(\hat{\mathcal{L}}, \mu)$. 

Conversely, for any solution $Z$ to the martingale problem, let $\tau_{n}=\inf\{t\geq 0: ||Z_{t}||\geq n\}$ for $n\in\mathbb{N}$, by the D\'{e}but theorem (see e.g. 4-50 in \cite{dellacherie1975probabilities}), $\tau_{n}$ is an $\mathfrak{F}$-stopping time since the filtered probability space is assumed to be normal. For any $f\in D(A^{+})$, take a smooth bump function $h_{c}\in C_{c}^{\infty}(\mathbb{R})$ such that $h_{c}|_{\{|x|\leq n||f||\}}=x$, then by (\ref{gene1}), we know that
$$\langle Z_{t}, f\rangle-\langle Z_{0}, f\rangle-\int_{0}^{t}[\langle Z_{s}, A^{+}f\rangle+\langle F(Z_{s}),f\rangle]ds$$ 
is a continuous local martingale. Similarly, if we take $h_{c}\in C_{c}^{\infty}(\mathbb{R})$ such that $h_{c}|_{\{|x|\leq n||f||\}}=x^{2}$, then 
$$\langle Z_{t}, f\rangle^{2}-\langle Z_{0}, f\rangle^{2}-2\int_{0}^{t}\langle Z_{s}, f\rangle (\langle Z_{s}, A^{+}f_{i}\rangle+\langle F(Z_{s}), f_{i}\rangle)ds-\int_{0}^{t}\langle B(Z_{s}), f_{i}\rangle \langle B(Z_{s}), f_{j}\rangle ds$$
is a continuous local martingale. 
 
In particular, if $f^{(1)}=0$, then we can show that
$$\langle Y_{t}, f^{(2)}\rangle=\langle Y_{0}, f^{(2)}\rangle+\int_{0}^{t}[\langle Y_{s}, A^{+}(0, f^{(2)})\rangle+\langle F^{(2)}(Z_{s}),f^{(2)}\rangle]ds,$$
where $F^{(2)}(Z_{s})$ is the second component of $F(Z_{s})$.

If $f^{(1)}\neq 0$ and $f^{(2)}=0$, then
$$\langle X_{t}, f^{(1)}\rangle-\langle X_{0}, f^{(1)}\rangle-\int_{0}^{t}[\langle X_{s}, A^{+}(f^{(1)}, 0)\rangle+\langle F^{(1)}(Z_{s}), f^{(1)}\rangle]ds$$ 
is a continuous local martingale of which the quadratic variation process is $\int_{0}^{t}X_{s}(1-X_{s})ds$, thus by the martingale representation theorem, there exists a normal filtered probability space $(\hat{\Omega}, \hat{\mathcal{F}}, \hat{P},\hat{\mathfrak{F}})$ and a standard $\mathfrak{F}\times\hat{\mathfrak{F}}$-Brownian motion $\hat{W}$ defined on $(\Omega\times\hat{\Omega}, \mathcal{F}\times\hat{\mathcal{F}}, P\times\hat{P})$ such that
\begin{equation}\label{eq:01}
\langle \tilde{X}_{t}, f^{(1)}\rangle=\langle \tilde{X}_{0}, f^{(1)}\rangle+\int_{0}^{t}[\langle \tilde{X}_{s}, A^{+}(f^{(1)}, 0)\rangle+\langle F^{(1)}(\tilde{Z}_{s}), f^{(1)}\rangle]ds+\int_{0}^{t}\sqrt{\tilde{X}_{s}(1-\tilde{X}_{s})}f^{(1)}d\hat{W}_{s},	
\end{equation}
where $\tilde{Z}_{t}(\omega, \hat{\omega})=(\tilde{X}_{t}(\omega, \hat{\omega}),\tilde{Y}_{t}(\omega, \hat{\omega}))=(X_{t}(\omega), Y_{t}(\omega))$ for $t\geq 0$ and $(\omega, \hat{\omega}) \in \Omega\times\hat{\Omega}.$
Note that we still have 
\begin{equation}\label{eq:02}
\langle \tilde{Y}_{t}, f^{(2)}\rangle=\langle \tilde{Y}_{0}, f^{(2)}\rangle-\int_{0}^{t}[\langle \tilde{Y}_{s}, A^{+}(0, f^{(2)})\rangle+\langle F^{(2)}(\tilde{Z}_{s}),f^{(2)}\rangle]ds.	
\end{equation}
Combining (\ref{eq:01}) and (\ref{eq:02}) together, we conclude that $Z$ is a weak solution to equation (\ref{eq:SEE}). Finally, by Theorem \ref{th:wellposed}, the weak uniqueness holds, and by a monotone class argument and the disintegration theorem, the last statement follows. 
\end{proof}

\begin{proof}[\textbf{Proof of Lemma \ref{lem:derivative}}]~

For any $h(z)=h_{c}\left(\left\langle z, f_{1}\right\rangle, \cdots,\left\langle z, f_{n}\right\rangle\right)\in H$ and $e\in (D, v)$, let
$$F_{z}(e)=h(z+e)-h(z)-\langle e, \sum\limits_{i=1}^{n} \partial_{x_{i}} h_{c}(\langle z, f_{1}\rangle \cdots \langle z, f_{n}\rangle) f_{i}\rangle.$$
Since 
\begin{eqnarray*}
|F_{z}(te)|&=&|h(z+te)-h(z)-t\langle e, \sum\limits_{i=1}^{n} \partial_{x_{i}} h_{c}(\langle z, f_{1}\rangle \cdots \langle z, f_{n}\rangle) f_{i}\rangle|\\
&=&|\sum_{i=1}^{n}[\partial_{x_{i}} h_{c}(\langle z+\theta t e, f_{1}\rangle \cdots \langle z+\theta t e, f_{n}\rangle)-\partial_{x_{i}} h_{c}(\langle z, f_{1}\rangle \cdots \langle z, f_{n}]\langle te, f_{i}\rangle|\\
&=&|\sum_{i,j=1}^{n}\partial_{x_{i}x_{j}} h_{c}(\langle z+\theta^{\prime} t e, f_{1}\rangle \cdots \langle z+\theta^{\prime} t e, f_{n}\rangle)\langle \theta t e, f_{j}\rangle\langle te, f_{i}\rangle|,
\end{eqnarray*}
where $\theta, \theta^{\prime}\in (0,1)$, then for $e\in V=\bigcap\limits_{i=1}^{n}\{z\in (D, v): |\langle z, f_{i}\rangle|\leq \frac{1}{n}\}$, we have
$|F_{z}(te)|\leq M\frac{t^{2}}{n^{2}}$
for some constant $M>0$.

By the definition of the weak$^{\star}$ topology, $V$ is an open $0$-neighborhood. Take $o(t)=Mt^{2}$ and $W=[-\frac{1}{n^{2}},\frac{1}{n^{2}}]$, then $F_{z}(te)\subseteq o(t)W$. Therefore, $h$ is Fr\'{e}chet differentiable by Definition \ref{defn:frechet}, and $\nabla h(z)=\sum\limits_{i=1}^{n} \partial_{x_{i}} h_{c}(\langle z, f_{1}\rangle \cdots \langle z, f_{n}\rangle) f_{i}$. The claim on the second order Fr\'{e}chet derivative $\nabla^{2} h(z)$ can be verified similarly.
\end{proof}

\bibliographystyle{alpha}  
\bibliography{references}  
\end{document}